\documentclass[twoside]{amsart}
\usepackage{amsmath,amssymb,amscd,amsthm,amsxtra,amsfonts}
\usepackage{mathtools,mathrsfs,dsfont,xparse}
\usepackage{multirow}
\usepackage{cases}
\usepackage{enumerate}
\usepackage{xcolor}
\usepackage[margin=1 in]{geometry}
\usepackage{todonotes}
\usepackage{listings}
\usepackage{tcolorbox}
\usepackage{cancel}

\mathtoolsset{showonlyrefs}
\usepackage{circledsteps}

\allowdisplaybreaks[4]



\date{}

\usepackage{multicol,bbm}
\usepackage[colorlinks=true, pdfstartview=FitV, linkcolor=blue, citecolor=blue, urlcolor=blue]{hyperref}

\numberwithin{equation}{section}

\usepackage{cjhebrew}
\usepackage{bm}

\DeclareMathOperator{\re}{Re}
\DeclareMathOperator{\im}{Im}

{\theoremstyle {definition} \newtheorem {defn} {Definition} [section] }
{\theoremstyle {plain}  \newtheorem {thm} [defn] {Theorem}}
{\theoremstyle {plain}  }
{\theoremstyle {plain} \newtheorem {prop} [defn]{Proposition}}
{\theoremstyle {plain} \newtheorem {lem}[defn] {Lemma}}
{\theoremstyle {definition} \newtheorem {rmk}[defn] {Remark}}
{\theoremstyle {plain} }

\def\T{{\mathbb{T}}}
\def\R{{\mathbb{R}}}
\def\cR{{\mathcal{R}}}
\def\C{{\mathbb{C}}}
\def\N{{\mathcal{N}}}
\def\Z{{\mathbb{Z}}}
\def\L{{\mathcal{L}}}

\def\S{{\mathcal{S}}}
\def\P{{\mathcal{P}}}
\def\I{{\mathcal{I}}}

\def\F{{\mathcal{F}}}
\def\e{{\varepsilon}}

\def\j{{\vec{j}}}
\def\FL{{\mathcal{F} L}}

\newcommand{\abs}[1]{\left|#1\right|}
\newcommand{\norm}[1]{\left\|#1\right\|}

\newcommand{\wh}[1]{\widehat{#1}}

\usepackage[numbers]{natbib}

\title{Low-regularity invariant measure for the complex-valued mKdV}

\author[Lee]{Zachary  Lee$^1$}
\address{$^1$  
Department of Mathematics\\ 
University of Texas at Austin\\ 
2515 Speedway, Stop C1200\\
Austin, TX 78712}
\email{zl9868@utexas.edu}

\author[Pavlovi\'c]{Nata\v{s}a Pavlovi\'c$^2$}
\address{$^2$  
Department of Mathematics\\ 
University of Texas at Austin\\ 
2515 Speedway, Stop C1200\\
Austin, TX 78712}
\email{natasa@math.utexas.edu}

\author[Staffilani]{Gigliola Staffilani$^3$}
\address{$^3$ Department of Mathematics\\
Massachusetts Institute of Technology\\ 
77 Massachusetts Avenue,  Cambridge, MA 02139}
\email{gigliola@mit.edu}

\author[Visciglia]{Nicola Visciglia$^4$}
\address{$^4$ Dipartimento di Matematica\\Universit\`a di Pisa\\
Largo Bruno Pontecorvo, 5, 56127,Italy}
\email{nicola.visciglia@unipi.it}

\date{\today}

\begin{document}

\begin{abstract} 
In this paper we consider the twice-renormalized, complex-valued modified KdV (mKdV) on the torus $\T$ introduced by Chapouto. Our main result is the construction of an invariant measure supported at low-regularity. This work complements the work of Kenig et al.\  which  constructed invariant measures supported in higher-regularity spaces for the \emph{non-renormalized} mKdV. Due to the low-regularity of the support of the measure, we are forced to work in Fourier-Lebesgue spaces. The fact that we consider the complex-valued mKdV makes the problem more complicated than the real-valued case, which was previously considered.
\end{abstract}

\maketitle

\setcounter{tocdepth}{1}
\tableofcontents

\parindent = 10pt     
\parskip = 8pt

\section{Introduction}

\subsection{Complex valued mKdV in the context of the $n$NLS hierarchy}
The defocusing/focusing (corresponding to respectively the $+$ and $-$ sign) cubic nonlinear Schrödinger (NLS) equation in one dimension, 
\begin{align} \label{cubic NLS}
    i\partial_t u + \partial_{xx} u = \pm 2 u|u|^2, \quad (t,x)\in \R \times \T, \quad  u(t,x)\in \C
\end{align}
for $\T= \R / 2\pi \Z$, is an example of a completely integrable model. It has infinitely many conservations laws. By  \cite{FaddeevTakhtajan2007}, we can inductively define each of them by 
\begin{align} \label{E_n}
    E_n [u]= \int_\T \bar{u} w_n[u]\,dx,
\end{align}
where $w_n: \S(\T)\to \S(\T)$ is defined as follows:
\begin{align}
    w_1[u] &= u \\
w_{n+1}[u] \;&=\; \,i\,\partial_x w_n[u] \;\pm\; \bar{u} \sum_{k=1}^{n-1} w_k[u]\, w_{n-k}[u]
\quad n \geq 1.
\end{align}
We list the first few conservation laws (here we focus on the $+$ sign for simplicity)
\begin{align}
E_1(u) &= \int_{\mathbb{T}} |u|^2 \, dx, \label{def of E_1} \\
E_2(u) &= \operatorname{Im} \int_{\mathbb{T}} \bar{u} \, \partial_x u \, dx, \\
E_3(u) &= \int_{\mathbb{T}} \big( |\partial_x u|^2 + |u|^4 \big) \, dx, \\
E_4(u) &= \operatorname{Im} \int_{\mathbb{T}} \Big( \partial_x u \, \partial_{xx} \bar{u} + 3|u|^2 u \, \partial_x \bar{u} \Big) \, dx, \\
E_5(u) &= \int_{\mathbb{T}} \Big( |\partial_{xx} u|^2 
         + 6 |\partial_x u|^2 |u|^2 
         + |\partial_x (|u|^2)|^2 
         + 2 |u|^6 \Big) \, dx.
\end{align}
By introducing the symplectic inner product 
\begin{align} \label{omega}
    \omega_{L^2}(f, g) := 2\operatorname{Im} \int_{\T} f(x)\overline{g(x)}\,dx, 
\end{align}
we may consider for each functional $E_n(u)$ the associated Hamiltonian flow 
\begin{align} \label{nNLS}
    \partial_t u(t) = \bm{\nabla} E_n(u(t)),
\end{align}
where $\bm{\nabla}$ is the symplectic gradient. Taking inspiration from \cite{FaddeevTakhtajan2007}, \cite{KNPSV}, we refer to \eqref{nNLS} as the $n$-th nonlinear Schrödinger equation, and we denote it by $(nNLS)$. The nonlinear Schrödinger hierarchy consists of the the whole family $(nNLS)$ for $n\in \mathbb{N}$. If $n=3$, then \eqref{nNLS} reduces to \eqref{cubic NLS}. It is classical (see e.g. \cite{Palais1997}) that the whole set of energies $E_n$ are conserved by the flow of each $(nNLS)$. \par 
This paper is devoted to the analysis of the Cauchy problem associated to \eqref{nNLS} for $n=4$, that is the complex-valued modified KdV (mKdV) 

\begin{align}
\label{mKdV}
        \begin{cases}
            \partial_t u + \partial_x^3 u =\pm 6|u|^2 \partial_x u, \quad (t,x)\in \R\times \T \\
            u(0)=u_0.
        \end{cases}
\end{align}
More precisely, we are interested in performing a probabilistic analysis of the initial-value problem  \eqref{mKdV} by constructing an invariant measure associated to the conservation law $E_3(u)$, the Hamiltonian of \eqref{cubic NLS}. However, as we will explain below, we will need to study a renormalized form of \eqref{mKdV} to make sense of this invariant measure. 
\par 

We note that this paper provides an answer towards the question of obtaining invariance of the $E_3$-based measure for the complex-valued mKdV equation, which was raised by Chapouto in \cite{Chapouto2021}.  
This present work may also be seen as a natural follow-up to \cite{KNPSV}, which studied the Cauchy problem \eqref{mKdV} from a probabilistic point of view by constructing a sequence of measures, each of which associated to a conserved energy $E_{2n+1}(u), n\ge 2$ (where $E_{2n+1}(u)$ is defined as in \eqref{E_n}) and is shown to be invariant under the flow of \eqref{mKdV}. Note that in \cite{KNPSV}, there was no need to gauge the mKdV \eqref{mKdV} as the support of the measures were smooth enough. 
In particular, the measure $\mu_n$ corresponding to $E_{2n+1}(u)$ was supported in $H^s(\T)$ for $s<n-1/2$, implying for example that the measure corresponding to $E_5(u)$ was supported in $H^{3/2-}(\T)$. Note therefore that the regularity of typical functions sampled from the support of the measure was high enough that in \cite{KNPSV} one could freely use energy estimates, which in turn crucially employ the fact that by the Sobolev embedding $H^{s} \hookrightarrow L^\infty, s>1/2$, the solutions are pointwise bounded. This is no longer true of the measure studied in this paper associated to the conservation law $E_3(u)$, which is supported in $H^{1/2-}(\T)$, a space where the above Sobolev embedding fails. 
This fact forces us to use different techniques to carry out the deterministic portion of the proof of our main result, in particular techniques which are inspired by \cite{ChapoutoJDDE} to study \eqref{mKdV} in $X^{s,b}_{p,q}$ spaces.

\subsection{Renormalization of the  mKdV}

In order to make progress towards our goal, we need to appropriately renormalize the complex-valued mKdV \eqref{mKdV}. Indeed, Chapouto in \cite{Chapouto2021} showed that the complex-valued mKdV \eqref{mKdV} is ill-posed in  Fourier-Lebesgue spaces for data with regularity below $H^{1/2}(\T)$. However, this is exactly where the support for the measure corresponding to the conserved quantity $E_3(u)$ lives. On the other hand, Chapouto also showed \cite{Chapouto2021, ChapoutoJDDE} that the Cauchy problem for a suitably renormalized complex-valued mKdV (mKdV2),
\begin{align} \label{mKdV 2Cauchy}
\begin{cases}
    \partial_t u + \partial_x^3 u = \pm 6\left(|u|^2 \partial_x u-M(u)\partial_x u -iP(u)u\right) \quad (t,x)\in \R\times \T,
    \\
    u(0)=u_0
    \end{cases}
\end{align}
where $M(u)$ and $P(u)$ are  the normalized mass and momentum,
\begin{align} \label{momentum}
    M(u) :=\frac{1}{2\pi}\int_\T |u|^2\,dx, \quad 
    P(u) := \frac{1}{2\pi} \operatorname{Im}\int \overline{u}\,\partial_x u\,dx,
\end{align}
\textit{is} locally well-posed in Fourier Lebesgue spaces $\FL^{s,p}(\T)$, defined by the norm 
\begin{align}
    \|u\|_{\FL^{s,p}(\T)}= \| \langle n \rangle^s \hat{u}(n) \|_{\ell^p(\Z)},
\end{align}
for all pairs $(s,p)$ such that $s\ge 1/2$ and $1\le p<\infty$. Notably, this includes $\FL^{s,p}(\T)$ spaces on which the measure is supported. Because of ill-posedness at the relevant low regularities, we study the Cauchy problem associated to the renormalized mKdV2 \eqref{mKdV 2Cauchy}, in place of the original mKdV equation. \par
Note that if  $u\in C(\R, H^{1/2}(\T))$, the momentum is finite and this $u(t,x)$ solves mKdV \eqref{mKdV} if and only if 
\begin{align} \label{Eq linking u and v}
v(t,x) :=  e^{\mp i tP(u)} u(t, x-M(u)t)
\end{align}
solves \eqref{mKdV 2Cauchy}

\subsection{Main Results of this Paper}

In order to formulate our probabilistic result, we begin by recalling the
Gaussian measure $\mu$. On a probability space $(\Omega,\mathcal{A},p)$ with sample space $\Omega$, $\sigma$-algebra $\mathcal{A}$ and probability measure $p$, 
consider a sequence $\{g_j(\omega)\}_{j\in\mathbb{Z}}$ of centered, normalized,
independent identically distributed complex Gaussian random variables. We define the random Fourier series
\begin{equation}\label{eq:random-series}
   \varphi(x,\omega) = \frac{1}{\sqrt{2\pi}}\sum_{j\in\mathbb{Z}} 
   \frac{g_j(\omega)}{ \langle j\rangle }\, e^{ijx},
\end{equation}
where $\langle j \rangle := \sqrt{1 + j^{2}}$.
The distribution of $\varphi$ is precisely the Gaussian measure $\mu$
supported on $H^{1/2-} (\mathbb{T})$ or more generally $\FL^{s,p}(\T)$ for $s<1-1/p$ (see Section~\ref{Section Gaussian Measures}
for further discussion of $\mu$). We recall that Gaussian measures on
infinite-dimensional Hilbert spaces are standard objects (see, for instance,
\cite{Bogachev2015Gaussian}). This was appropriate for the work in \cite{KNPSV} where the data were in $H^s(\T)$.
In the present paper we  will need to  work in the  Banach spaces $\FL^{s,p}(\T)$. This is a similar  situation that the authors faced in \cite{NahmodOhReyBelletStaffilani2012}, where  the same type of Fourier-Lebesgue spaces were considered. We then refer the reader to \cite{NahmodOhReyBelletStaffilani2012} for details of the construction of $\mu$ as a measure supported in $\FL^{s,p}(\T)$.

In addition to the Gaussian measures $\mu$, we also make use of weighted
Gaussian measures. For each $R>0$, introduce the cutoff function
\[
   \chi_R(\cdot) = \chi\!\left(\tfrac{\cdot}{R}\right), 
   \qquad \chi \in C^\infty_0(\mathbb{R},\mathbb{R}), 
   \qquad \chi(x)=1 \ \text{for } |x|<1.
\]
Next, we define the densities
\begin{align} \label{Densities intro}
    F_{R, \pm}(u) := \chi_R(E_1(u)) \exp\left(\mp \int_{\T} |u|^4\,dx\right)
\end{align}
and the associated weighted Gaussian measures
\begin{align} \label{Intro def of rho}
    d\rho_{R, \pm} = F_{R,\pm} (u)d\mu.
\end{align}
The idea to truncate the density using the mass $E_1(u)$ goes back to \cite{LebowitzRoseSpeer1988}. Bourgain showed in \cite{bourgain1994} using a sophisticated proof that the above defined density (and more general ones) $F_{R, -}\in L^q(d\mu)$, for any $1\le q<\infty$ and $R>0$, making the measure $d\rho_{R, -}$ well-defined (it is simpler to see that $d\rho_{R, +}$ is well defined, and one can remove the mass cutoff if desired). In this work, we will in fact present a simplified version of this result that only uses a standard Gagliardo-Nirenberg inequality (see Section \ref{Section Gaussian Measures}).

We can now state the main results of this paper.

\begin{thm} \label{Main Result}
    For every $(s,p)$ such that $1/2\le s <1-1/p, \quad 2<p<\infty$, and for both the focusing ($-$) and defocusing ($+$) signs there exists a Borel set $\Sigma^{s,p} \subset \F L^{s,p}(\mathbb{T})$ such that:
\begin{itemize}
    \item[(i)] $\mu(\Sigma^{s,p}) = 1$;
    \item[(ii)] For every $u_0 \in \Sigma^{s,p}$, there exists a unique, strong and global solution $u\in C(\mathbb{R}, \FL^{s,p}(\mathbb{T}))$ to the equation \eqref{mKdV 2Cauchy} in the sense of Definition \ref{Definition of strong solution} below. Moreover there exists $C > 0$ such that
    \[
 \|u(t,x)\|_{\F L^{s,p}(\mathbb{T})} \leq C \sqrt{C+ \log(1+|t|) }.
    \]
    \item[(iii)] $\Sigma^{s,p}$ is closed under the nonlinear flow $\Phi(t)$ associated with \eqref{mKdV 2Cauchy} established in (ii). Furthermore, for every $R > 0$ the measure $\rho_{R,\pm}$ defined in \eqref{Intro def of rho} is invariant along the flow $\Phi(t)$ restricted on the invariant set $\Sigma^{s,p}$.
\end{itemize}
\end{thm}

\begin{rmk}
 We note admissible pairs $(s,p)$ for which Theorem \ref{Main Result} holds are given as the intersection of pairs $(s,p)$ such that the Cauchy problem \eqref{mKdV 2Cauchy} is locally well-posed in $\FL^{s,p}(\T)$ \cite{ChapoutoJDDE} and the pairs such that $\FL^{s,p}(\T)$ has full $\mu$-measure:
\begin{align}
         \{ (s,p): 
         1/2\le s <1-1/p, 2<p<\infty\}
         = \{ (s,p): 
         s\ge 1/2, 1\le p<\infty\} \cap 
         \{ (s,p): 
         s<1-1/p, 1<p\le \infty\}.
    \end{align}
\end{rmk}
\begin{rmk}
Note that on the real line, we have that the homogeneous Fourier-Lebesgue spaces $\dot{\F}L^{s,p}(\R)$ scale like $\dot{H}^\sigma(\R)$, where $\sigma=s+1/p-1/2$. Letting $s=1/2$ and $2<p<\infty$, we see that $\dot{\F}L^{1/2,p}(\R)$ scales like $\dot{H}^{1/p}(\R)$. From these heuristics, our results may be thought of as living at the scaling of $H^{\varepsilon}(\T)$ for any $0<\varepsilon< 1/2$, almost reaching $L^2(\T)$.
\end{rmk}
The results in Theorem \ref{Main Result} complete the remaining low-regularity case in the recent program of constructing invariant measures for the complex-valued mKdV equation associated with its conservation laws. In particular, it settles the open $E_3$-level problem and broadens the range of probabilistic global dynamics now understood for this model. \par 
\par
Although one may exploit the complete integrability of complex-valued mKdV to obtain deterministic global results (see \cite{ChapoutoJDDE}) with uniform-in-time bounds instead of logarithmic, the approach developed in this work is motivated by a different objective. Namely, we seek methods that rely less heavily on integrable structure and are therefore more likely to extend to non-integrable models, which currently vastly outnumber the integrable ones. \par
One might wonder what the result in Theorem \ref{Main Result} might imply for the original mKdV \eqref{mKdV} itself. While we expect that one can construct for mKdV a candidate measure like in our result for the mKdV2 \eqref{mKdV 2Cauchy}, since even the Cauchy problem for the mKdV \eqref{mKdV} in $\FL^{s,p}(\T)$ is {\it deterministically} ill-posed in the support of the measure $\mu$, the question of the invariance of the aforementioned measure associated to the mKdV  remains.

\subsection{Difficulties and Novelties}
In this subsection, we highlight the difficulties we encountered in our work and discuss how we resolved them by introducing new approaches.

\begin{itemize}

\item In \cite{KNPSV}, the measure corresponding to the conservation law $E_5(u)$ was supported in $H^{3/2-}(\T):=\cap_{s<3/2}H^s(\T)$ (with measures corresponding to higher energies having support in smoother spaces). Note therefore that the regularity of typical functions sampled from the support of the measure was high enough that in \cite{KNPSV} one could freely use energy estimates, which in turn crucially employ the fact that by the Sobolev embedding $H^{s}(\T) \hookrightarrow L^\infty(\T), s>1/2$, the solutions are pointwise bounded. Combined with Strichartz estimates, this allowed the authors of \cite{KNPSV} to obtain bounds for the truncated flow
\begin{align}\label{mKdV Cauchy truncated}
    \begin{cases}
            \partial_t u_{N} + \partial_x^3 u_{N} = \pm \Pi_N (|\Pi_N u|^2 \partial_x \Pi_N u_{N})\\
            u_N(0)=u_{0},
        \end{cases}
\end{align}
where $\Pi_N$ denotes the Fourier cutoff to frequencies $\{|n|\le N\}$. These bounds were uniform in the truncation parameter $N$, allowing the construction of the full non-truncated flow living in the support of the measure. This is no longer true of the measure studied in this paper associated to the conservation law $E_3(u)$, which is supported in $H^{1/2-}(\T)$, a space where the above Sobolev embedding fails, and where the mKdV2 is ill-posed. This fact forces us to use different techniques to carry out the deterministic portion of the proof of our main result, in particular techniques which are inspired by \cite{ChapoutoJDDE, Chapouto2021} to study \eqref{mKdV 2Cauchy} in the Fourier-Lebesgue $\FL^{s,p}(\T)$ and the associated auxiliary Bourgain-type $X^{s,b}_{p,q}(\T)$ spaces.\par

\item While local well-posedness for the Cauchy problem \eqref{mKdV 2Cauchy} was established by Chapouto in \cite{Chapouto2021,ChapoutoJDDE}, in this work we refine those arguments in the setting of a truncated, finite-dimensional problem. This refinement is needed in order to obtain close-to-optimal bounds that are uniform in the truncation parameter, following the approach of \cite{KNPSV}. Obtaining these bounds is much more involved than the analogue in \cite{KNPSV} since we are forced to work with a system of equations (see \cite{ChapoutoJDDE}), with the equation for the additional auxiliary quantity containing a larger number of different cubic, quintic and septic terms all of which have to be carefully estimated in the Bourgain-type spaces. \par

\item During the proof of the main result, it is necessary to understand the convergence of the flow of the truncated problem $u_N(t)$ to that of the original problem $u(t)$. In \cite{KNPSV}, this was done through energy estimates: one first showed that the sequence $u_N(t)$ was a Cauchy sequence in a relevant norm, and showed that the limit satisfied the original equation.
However, because we are working at lower regularities these energy estimates are not available to us. Instead of showing that the finite-dimensional flow is Cauchy in the cutoff parameter to construct the full flow, we choose to directly construct the non-truncated flow via a fixed-point argument. We then work with the difference equation, or system of equations, relating the non-truncated flow to the truncated flow and aim to estimate the difference of the two nonlinearities. To show the desired estimates, we devise an approach for understanding the convergence of the nonlinearities by decomposing them into two different parts, the first of which is estimated by exploiting the fact that there will always be at least one difference input while the second is estimated by proving that it always has at least one high frequency input. We then make use of our previously shown uniform-in-$N$ bounds on the solution in the relevant Bourgain spaces to show the desired convergence.

\item
Lastly, we present a simplified and, to our knowledge, new proof that the density $F_{R,-}$ defined in \eqref{Densities intro},
belongs to $L^q(d\mu)$ for any $1\le q<\infty$ and $R>0$ that only uses a standard Gagliardo-Nirenberg inequality (see Proposition \ref{GN proposition}),
\begin{align} \label{GN specific intro}
        \|u\|_{L^4(\T)} \le \|\langle \nabla \rangle^{1/4} u\|_{L^{8}(\T)}^{2/5} \|u\|_{L^2(\T)}^{3/5},
\end{align}
and hypercontractivity estimates for Gaussian random variables (see \eqref{hypercontractivity}). Note that our proof also extends to densities of the form
\begin{align}
    \chi_R(E_1(u)) \exp\left( \int_{\T} |u|^p\,dx\right)
    \end{align}
for any $2<p<6$.
\end{itemize}

\subsection{Notation}
\begin{itemize}
    \item We denote $B^{s,p}_R$ the ball of radius $R>0$ in $\F L^{s,p}(\T)$
    \item We allow the constants $C>0$ in inequalities such as 
    \begin{align}
        A\le C B
    \end{align}
    to change from line to line.
    \item To shorten the notations we shall write sometimes $\partial_{x} = \partial$.
    \item We shall write $\int f = \int_{\mathbb{T}} f dx$.
    \item We also write $\FL^{s,p} = \FL^{s,p}(\mathbb{T})$
    \item In the paper the functions are always $\mathbb{C}$-valued.
    \item For $x \in \mathbb{R}$ we will use the notation $\langle x \rangle := \sqrt{1 + x^{2}}$.
    \item We denote by $\Phi(t)$ and $\Phi_{N}(t)$ the flows associated with \eqref{mKdV 2Cauchy} and \eqref{mKdV2 truncated}.
    \item For any topological space $X$ we define $\mathcal{B}(X)$ the Borelian subsets of $X$.
\end{itemize}

\subsection*{Acknowledgements} 
ZL is partially supported by the NSF grants DMS-1840314, DMS-2052789 and DMS-2511517 through NP's grants. ZL also acknowledges support by The University of Texas at Austin through the Provost’s Graduate Excellence Fellowship.
NP is partially supported by the NSF grants DMS-1840314, DMS-2052789 and DMS-2511517. GS is partially supported by the NSF grant DMS-2052651 and the Simons Foundation
Collaboration Grant on Wave Turbulence.
This material is based upon work supported by the National Science Foundation under Grant No. DMW-2424139, while ZL, NP and GS were in residence at the Simons Laufer Mathematical Sciences Institute in Berkeley, California during the Fall 2025 semester. 
The authors would like to thank the Simons Laufer Mathematical Institute for their kind hospitality.

\section{Preliminaries}
In this section, we recall various notation and arguments used in \cite{ChapoutoJDDE} which will be useful in Section \ref{Section:FiniteDimApproximations}.
\subsection{Auxiliary spaces}
Let $\S(\R\times \T)$ be the space of functions $u\in C^\infty(\R\times \T)$ which satisfy
\begin{align}
    u(t,x+1)= u(t,x), \quad \sup_{(t,x)\in \R\times \T} |t^\alpha \partial_t^\beta \partial_x^\gamma u(t,x)|< \infty, \quad \forall \alpha, \beta, \gamma \in \Z.
    \end{align}
Bourgain \cite{BourgainGAFA} introduced the following $X^{s,b}$ spaces to study the KdV equation:
\begin{align*}
    \norm{u}_{X^{s,b}} = \norm{\langle n \rangle^s \langle \tau - n^3\rangle^b \wh{u}(\tau, n)}_{\ell ^2_n L^2_\tau}
\end{align*}
We will need the following generalization of these spaces:
\begin{defn}
 Let $s, b \in \mathbb{R}, 1 \leq p, q \leq \infty$. The space $X_{p, q}^{s, b}(\mathbb{R} \times \mathbb{T})$, abbreviated $X_{p, q}^{s, b}$, is defined as the completion of $\mathcal{S}(\mathbb{R} \times \mathbb{T})$ with respect to the norm
$$
\|u\|_{X_{p, q}^{s, b}}=\left\|\langle n\rangle^s\left\langle\tau-n^3\right\rangle^b \widehat{u}(\tau, n)\right\|_{\ell_n^p L_\tau^q}
$$
When $p=q=2$, the $X_{p, q}^{s, b}$-spaces defined above reduce to the standard $X^{s, b}$-spaces. 
\end{defn}

We now define time-localized version of these spaces.  
\begin{defn}\label{spaces}
Let $s, b \in \mathbb{R}, 1 \leq p, q<\infty$ and $I \subset \mathbb{R}$ an interval. We define the restriction space $X_{p, q}^{s, b}(I)$ of all functions $u$ which satisfy
$$
\|u\|_{X_{p, q}^{s, b}(I)}:=\inf \left\{\|v\|_{X_{p, q}^{s, b}}: v \in X_{p, q}^{s, b}(\mathbb{R} \times \mathbb{T}),\left.v\right|_{t \in I}=u\right\}<\infty
$$
with the infimum taken over all extensions $v$ of $u$. If $I=[-T, T]$, for some $0<T \leq 1$, we denote the spaces by $X_{p, q}^{s, b}(T)$. 
\end{defn}

We introduce the following quantities depending on $0<\delta\ll 1$, first introduced in \cite{ChapoutoJDDE}, which in turn depends on $2<p<\infty$. There is the restriction $\delta \lesssim  p^{-2}$ as $p\to\infty$. 
\begin{align}
b_0 &= 1 - 2\delta, & \qquad b_1 &= 1 - \delta, \\
q_0 &= \frac{1}{4\delta}, & \qquad q_1 &= \frac{1}{4.5\delta}, \\
\frac{1}{r_0} &= \frac{1}{2} + \delta, & \qquad \frac{1}{r_1} &= \frac{1}{2} + 2\delta, \\
&& \qquad \frac{1}{r_2} &= \frac{1}{2} + 3\delta.
\end{align}
Following \cite{ChapoutoJDDE}, we also define the following spaces:
\begin{align}
Y_0^s &= X^{s, \frac12}_{p,r_0}(\mathbb{R} \times \mathbb{T})  \label{Y_0^s} \\
Y_1^s &= X^{s, \frac12}_{p,r_1}(\mathbb{R} \times \mathbb{T}), \label{Y_1^s}\\
Z_0^s &= X^{s, b_0}_{p,q_0}(\mathbb{R} \times \mathbb{T}), \label{Z_0^s}\\
Z_1^s &= X^{s, b_1}_{p,q_0}(\mathbb{R} \times \mathbb{T}).\label{Z_1^s}
\end{align}
The spaces $Y_{j}^s(I), Z_{j}^{s}(I)$ for $j=0,1$ are defined similarly as in Definition \ref{spaces}.

Recall the following embedding: for any $1 \le p < \infty$,
\[
X^{s,b}_{p,q}(\mathbb{R} \times \mathbb{T}) \hookrightarrow C(\mathbb{R}; \mathcal{F} L^{s,p}(\mathbb{T})) 
\quad \text{for} \quad b > \frac{1}{q'} = 1 - \frac{1}{q}.
\]
In particular, it holds that $Z_0^s \subset Y_0^s \subset C(\mathbb{R}; \mathcal{F} L^{s, p}(\mathbb{T}))$.
We choose $\varphi:\R \to \R$ to be a smooth function such that 
\begin{align} \label{Definition of phi}
\varphi(t) = 
\begin{cases}
    1, \quad |t| \le 1 \\
    0, \quad |t| \ge 2. 
\end{cases}
\end{align}
For $T>0$, we also define
\begin{align} \label{Definition of rescaled phi}
    \varphi_T(t)=\varphi(t/T).
\end{align}

\subsection{Decomposition of nonlinearity}

The nonlinearity of \eqref{mKdV 2Cauchy} has the following spatial Fourier transform
\begin{align}
\sum_{\substack{n=n_1+n_2+n_3 \\ \Phi(\bar{n}_{123}) \neq 0}} 
i n_1 \widehat{u}(n_1)\widehat{\overline{u}}(n_2)\widehat{u}(n_3)
- i n |\widehat{u}(n)|^2 \widehat{u}(n),
\label{Fourier transform of nonlinearity}
\end{align}
where $\bar{n}_{123} = (n_1, n_2, n_3)$ and $\Phi$ denotes the resonance relation
\[
\Phi(\bar{n}_{123}) = n^3 - n_1^3 - n_2^3 - n_3^3
= 3(n_1+n_2)(n_1+n_3)(n_2+n_3),
\]
where the factorization holds if $n = n_1 + n_2 + n_3$. 

We consider the following operators:
\begin{align}
\mathcal{F}_x(\mathcal{N}\mathcal{R}_{\ge}(u_1,u_2,u_3))(n)
&= \sum_{\substack{n=n_1+n_2+n_3 \\ \Phi(\bar{n}_{123})\neq 0 \\ |n_2|\ge |n_3|}} 
i n_1 \widehat{u}_1(n_1)\widehat{\overline{u}}_2(n_2)\widehat{u}_3(n_3), \label{NR >=}
\\
\mathcal{F}_x(\mathcal{N}\mathcal{R}_{\le}(u_1,u_2,u_3))(n)
&= \sum_{\substack{n=n_1+n_2+n_3 \\ \Phi(\bar{n}_{123})\neq 0 \\ |n_2|\le |n_3|}} 
i n_1 \widehat{u}_1(n_1)\widehat{\overline{u}}_2(n_2)\widehat{u}_3(n_3),
\label{NR >}\\
\mathcal{F}_x(\mathcal{R}(u_1,u_2,u_3))(n)
&= - i n \widehat{u}_1(n)\widehat{\overline{u}}_2(n)\widehat{u}_3(n). \label{R}
\end{align}
As a result, we can write the nonlinearity \eqref{Fourier transform of nonlinearity} as a sum of a  resonant and non-resonant parts:
\begin{align} \label{Nonlinearity of mKdV2 physical space}
\mathcal{N}(u,\overline{u},u)
= \mathcal{N}\mathcal{R}_{\ge}(u,\overline{u},u)
+ \mathcal{N}\mathcal{R}_{\le}(u,u,\overline{u})
+ \mathcal{R}(u,u,u).
\end{align}
If $n_j$ denotes the spatial frequency corresponding to $u_j$, $j=1,2,3$, in
\eqref{NR >=} and \eqref{NR >}, then $|n_2|\ge |n_3|$. We therefore define the following subregions:
\begin{align*}
\mathbb{X}_A &= \{(n_1,n_2,n_3)\in\mathbb{Z}^3 : \Phi(\bar{n}_{123})\ne 0,~|n_2|\ll |n_1|\},\\
\mathbb{X}_B &= \{(n_1,n_2,n_3)\in\mathbb{Z}^3 : \Phi(\bar{n}_{123})\ne 0,~|n_3|\lesssim \min(|n|,|n_1|)\le \max(|n|,|n_1|)\sim|n_2|\},\\
\mathbb{X}_C &= \{(n_1,n_2,n_3)\in\mathbb{Z}^3 : \Phi(\bar{n}_{123})\ne 0,~|n|\lesssim |n_3|\ll |n_1|\},\\
\mathbb{X}_D &= \{(n_1,n_2,n_3)\in\mathbb{Z}^3 : \Phi(\bar{n}_{123})\ne 0,~|n_1|\lesssim |n_3|\}.
\end{align*}

For $\ast \in \{A,B,C,D\}$, we let $\mathcal{N}\mathcal{R}_{\ast,\ge}$ and $\mathcal{N}\mathcal{R}_{\ast,\le}$
denote the restrictions of the operators in \eqref{NR >=} and \eqref{NR >} to $\mathbb{X}_\ast$. 
We can thus write the non-resonant contributions as follows:
\[
\mathcal{N}\mathcal{R}_{\ge} = \mathcal{N}\mathcal{R}_{A,\ge} + \mathcal{N}\mathcal{R}_{B,\ge} 
+ \mathcal{N}\mathcal{R}_{C,\ge} + \mathcal{N}\mathcal{R}_{D,\ge},
\quad
\mathcal{N}\mathcal{R}_{\le} = \mathcal{N}\mathcal{R}_{A,\le} + \mathcal{N}\mathcal{R}_{B,\le} 
+ \mathcal{N}\mathcal{R}_{C,\le} + \mathcal{N}\mathcal{R}_{D,\le}.
\]
Define the Duhamel operator $D$, and its truncated version $\mathcal{D}$:
\begin{align}
DF(t,x) &= \int_0^t S(t - t')F(t',x)\, dt', \label{Non-truncated D} \\
\mathcal{D}F(t,x) &= \varphi(t) \cdot D(\varphi(t') \cdot F(t',x))(t) 
= \varphi(t) \int_0^t S(t - t')\varphi(t')F(t',x)\, dt'.
\end{align}
In studying the integral formulation of \eqref{mKdV 2Cauchy}, we will need to consider the truncated Duhamel operator applied to the various terms in the nonlinearity \eqref{Nonlinearity of mKdV2 physical space} (see Section \ref{Subsection def of solution}). Following \cite{ChapoutoJDDE}, we make a further decomposition of  non-resonant terms $\mathcal{D} \N \cR_{*, \ge}$ and $\mathcal{D} \N \cR_{*, >}$ into respectively $G_{*, \ge}, \mathbf{B}_{*, \ge}$ and $G_{*, >}, \mathbf{B}_{*, >}$ respectively for $*\in\{A,B\}$. The contributions in $G_{*, \ge},G_{*, >}$ will have enough smoothness to control the derivative in the nonlinearity. In order to execute this decomposition, we first introduce a Schwartz-class function $\eta$ satisfying
\begin{align}
    \hat{\eta}(-1)=0, \quad \mathcal{H}\hat{\eta}(-1)=-1, 
\end{align}
where $\mathcal{H}$ denotes the Hilbert transform.
We define the operators $G_{*, \ge},\mathbf{B}_{*, \ge}$ via their Fourier transforms:
\begin{align}
\mathcal{F}_x(G_{\ast,\ge}(u_1,u_2,u_3))(t,n) &= \varphi(t) 
   \sum_{\substack{n=n_1+n_2+n_3 \\ \bar{n}_{123} \in \mathbb{X}_\ast(n), \\ |n_2|\ge |n_3|}} 
   i n_1 \int_0^t e^{i(t-t')n^3} 
   \eta(\Phi(\bar{n}_{123})(t-t')) \\
   &\times 
   \varphi(t') 
   \prod_{j=1}^3 \widehat{u}_j(t',n_j)\, dt', \label{Definition of G_(*,>=)} \\ 
\mathcal{F}_x(\mathbf{B}_{\ast,\ge}(u_1,u_2,u_3))(t,n) &= \varphi(t) 
   \sum_{\substack{n=n_1+n_2+n_3 \\ \bar{n}_{123} \in \mathbb{X}_\ast(n), \\ |n_2|\ge |n_3|}} 
   i n_1 \int_0^t e^{i(t-t')n^3} 
   \bigl[1 - \eta(\Phi(\bar{n}_{123})(t-t'))\bigr] 
   \varphi(t') 
   \prod_{j=1}^3 \widehat{u}_j(t',n_j)\, dt'.
\end{align}
with similar definitions for $G_{*, >},\mathbf{B}_{*, >}$ but with the conditions $|n_2|>|n_3|$ replacing $|n_2|\ge |n_3|$ in the sums. 
\subsection{Multilinear estimates }
We state the following estimates concerning the operators $G_{*, >}, G_{*, \ge}$. 

\begin{lem} \label{Estimates for G}
The following estimates hold for $s\ge 1/2, 2\le p<\infty$,
    \begin{align}
\| G_{A, \#}(u_1,u_2,u_3) \|_{Y_1(T)} &\lesssim \|u_1\|_{Z_0(T)} \|u_2\|_{Y_0(T)} \|u_3\|_{Y_0(T)}, \\
\| G_{B, \#}(u_1,u_2,u_3) \|_{Y_1(T)} &\lesssim \|u_1\|_{Z_0(T)} \|u_2\|_{Z_0(T)} \|u_3\|_{Y_0(T)}.
\end{align}
for both $\#\in\{>, \ge\}$.
\end{lem}
\begin{proof}
This lemma follows from the following estimates given in Lemma 5.1 of Section 5 in \cite{ChapoutoJDDE}. 
\begin{align}
    \| G_A(\widetilde{u}_1,\widetilde{u}_2,\widetilde{u}_3) \|_{Y_1} 
&\lesssim \|\widetilde{u}_1\|_{Z_0} \, \|\widetilde{u}_2\|_{Y_0} \, \|\widetilde{u}_3\|_{Y_0}, \\
\| G_B(\widetilde{u}_1,\widetilde{u}_2,\widetilde{u}_3) \|_{Y_1} 
&\lesssim \|\widetilde{u}_1\|_{Z_0} \, \|\widetilde{u}_2\|_{Z_0} \, \|\widetilde{u}_3\|
\end{align}
applied to extensions $\widetilde{u_j}$ of $u_j$ that agree with $u_j$ on $[-T,T]$. Then, we may take the infimum over all such extension on both sides to conclude the lemma.
\end{proof}
Due to the simple estimate
\begin{align}
        \norm{\Pi_N v}_{X^{s,b}_{p,q}} \le \|v\|_{X^{s,b}_{p,q}}
\end{align}
which holds uniformly in $N$,  we also have the following result:
\begin{lem}
The following estimates holds 
    \begin{align}
\| \Pi_NG_A(\Pi_N u_1,\Pi_{N}u_2,\Pi_{N}u_3) \|_{Y_1} &\lesssim \|\Pi_{N}u_1\|_{Z_0} \|\Pi_{N}u_2\|_{Y_0} \|\Pi_{N}u_3\|_{Y_0}, \\
\| \Pi_{N}G_B(\Pi_{N}u_1,\Pi_{N}u_2,\Pi_{N}u_3) \|_{Y_1} &\lesssim \|\Pi_{N}u_1\|_{Z_0} \|\Pi_{N}u_2\|_{Z_0} \|\Pi_{N}u_3\|_{Y_0}.
\end{align}
\end{lem}
We will also need the following lemma from \cite{ChapoutoJDDE}. 
\begin{lem} \label{Lemma for F} 
Let $Y_j^s, Z_j^s, j\in\{0,1\}$ be as defined in \eqref{Y_0^s}, \eqref{Y_1^s}, \eqref{Z_0^s}, \eqref{Z_1^s}. Suppose that $f$ is a smooth function such that $f(0)=0$. Then, we have the following estimates
$$
\begin{aligned}
& \left\|\varphi_T \cdot f\right\|_{Y_0^s(T)} \lesssim T^\theta\|f\|_{Y_1^s(T)} \\
& \left\|\varphi_T \cdot f\right\|_{Z_0^s(T)} \lesssim T^\theta\|f\|_{Z_1^s(T)}
\end{aligned}
$$
for any $0<\theta \leq \frac{\delta}{2}$ and $0<T \leq 1$
\end{lem}
\begin{proof}
    This is proved using the similar estimate given in Lemma 2.4 of Section 2 of \cite{ChapoutoJDDE}:
    \begin{align}
        & \left\|\varphi_T \cdot f\right\|_{Y_0^s} \lesssim T^\theta\|f\|_{Y_1^s} \\
& \left\|\varphi_T \cdot f\right\|_{Z_0^s} \lesssim T^\theta\|f\|_{Z_1^s}
    \end{align}
    by applying it to extensions of $f$, and taken the infimum of both sides.
\end{proof}

\subsection{Definition of solution} \label{Subsection def of solution}

\begin{defn} \label{Definition of strong solution}

We say that 
\[
u \in C([-T, T]; \mathcal{F}L^{s,p}(\mathbb{T}))
\]
is a strong solution of \eqref{mKdV 2Cauchy} with initial data 
$u_0 \in \mathcal{F}L^{s,p}(\mathbb{T})$ on the interval $[-T,T]$ 
if it satisfies the following integral equation, the Duhamel formulation:
\begin{align} \label{mKdV2 Duhamel, Preliminaries}
u(t) = S(t)u_0 + D\mathcal{N}\mathcal{R}(u,\overline{u},u)(t) 
+ D\mathcal{R}(u,u,u)(t). 
\end{align}
for each $t\in[-T,T]$.
\end{defn}
When constructing solutions to the above integral equation in the case that $T$ is large, say $T>1$, we will always first construct solutions on small time intervals $[-\tau+t_0, \tau+t_0], t_0\in (-T,T), 0<\tau < 1$ and concatenate them together. By time-translation invariance, we can reduce to the case $t_0=0$. Therefore, we can always work with the above definition assuming that we are working on a small time interval $[-\tau, \tau]$. Due to the aforementioned remark, denoting our small interval by $[-\tau, \tau]$, we can consider an equivalent formulation of the Duhamel formulation \eqref{mKdV2 Duhamel, Preliminaries} when $T<1$:
\begin{align} \label{time-truncated mKdV2 Duhamel, Preliminaries}
u(t) = \varphi \cdot S(t)u_0 
+ \varphi_\tau \cdot D\mathcal{N}\mathcal{R}(u,\overline{u},u)(t) 
+ \varphi_\tau D\mathcal{R}(u,u,u)(t), \quad \forall t\in[-\tau, \tau]. 
\end{align}
Indeed, we see that if $u$ is a strong solution to \eqref{mKdV 2Cauchy} on $[-T,T]$ with $T<1$, it will also be a solution to \eqref{time-truncated mKdV2 Duhamel, Preliminaries} with $\tau=T$ and vice-versa. This is because for $t\in[-\tau,\tau], 0<\tau<1$, we have that $\varphi(t)=1, \varphi_\tau (t)=1$.  We introduce \eqref{time-truncated mKdV2 Duhamel, Preliminaries} following \cite{ChapoutoJDDE} who proves bounds specific to the time-truncated terms on the right hand side of \eqref{time-truncated mKdV2 Duhamel, Preliminaries}. For the nonlinear terms on the right hand side of \eqref{time-truncated mKdV2 Duhamel, Preliminaries}, we additionally truncate via the parameter $\tau$ to gain a suitable power of $\tau$ when performing the nonlinear estimates. It is not necessary to do this for the linear evolution.

\section{Finite-dimensional approximations} \label{Section:FiniteDimApproximations}

In this section, we introduce a finite-dimensional approximation of the Cauchy problem \eqref{mKdV 2Cauchy}, which we will then use in Section \ref{Section almost invariance} and Section \ref{Section: Proof of Main Theorem} to study the invariance of a particular measure under the flow of \eqref{mKdV 2Cauchy}. We also prove uniform bounds in $N$ in Fourier-Lebesgue spaces for such a  truncated flow on short time intervals depending  only on the norm of the initial data. To accomplish this, we are inspired by the approach taken in \cite{KNPSV}. However, we cannot  use energy estimates as is done in \cite{KNPSV} and instead rely on estimates given in \cite{ChapoutoJDDE}, where the well-posedness of \eqref{mKdV 2Cauchy} was studied in Fourier-Lebesgue spaces via $X^{s,b}_{p,q}$ estimates. To obtain better range of parameters in Fourier-Lebesgue spaces, we study the re-centered truncated flow. This is inspired by \cite{ChapoutoJDDE}, where the method was used for the full (non-truncated) flow.
We also show in Section 2.1 that under an assumption that the Fourier-Lebesgue norm of the truncated flow stays bounded on a large time-interval $[-T,T]$, there exists a solution of the full flow on this time interval, with the former converging to the latter in the topology of $C([-T,T], \F L^{s',p})$ with $1/2\le s'<s$.

\subsection{Bounds uniform in $N$ for the flow $\Phi_N(t)u_0$}
We now consider the following finite-dimensional approximation to the Cauchy problem \eqref{mKdV 2Cauchy}:
\begin{align}\label{mKdV2 truncated}
    \begin{cases}
            \partial_t u + \partial_x^3 u = 6 \Pi_N \,\N(\Pi_N u,\overline{\Pi_N u},\Pi_N u) \quad (t,x)\in \R\times \T \\
            u(0)=u_0,
        \end{cases}
\end{align}
where $\N$ is given as in \eqref{Nonlinearity of mKdV2 physical space} corresponds to the right hand side of \eqref{mKdV 2Cauchy} and  $\Pi_N$ is the Dirichlet projection
\begin{align} \label{Definition of Pi_N}
    \Pi_N \left( \sum_{n\in \Z } a_n e^{inx}\right) = \sum_{|n|\le N } a_n e^{inx}.
\end{align}
We define 
\begin{align}
    \Pi_{>N} = I- \Pi_N.
\end{align}
In what follows, we denote by $\Phi(t)$ the flow associated with \eqref{mKdV 2Cauchy} and by $\Phi_N(t)$ the flow associated with \eqref{mKdV2 truncated}. \par 
We note that for each fixed $N$, the Cauchy problem \eqref{mKdV2 truncated} is globally well-posed for initial data in $\F L^{s,p}$ for any $s \ge 0, 1\le p\le \infty$. Indeed, the solution splits as 
\begin{align}
    u = \Pi_N u + \Pi_{>N}u,
\end{align}
where $\Pi_N u$ evolves to a system of $2N+1$ nonlinear-ODEs (and hence has a local solution) while $\Pi_{>N}$ evolves according to the free KdV flow, and hence $\norm{\Pi_{>N} u(t)}_{\F L^{s,p}(\T)}=\norm{\Pi_{>N} u_0}_{\F L^{s,p}(\T)}$. To upgrade the local solution to a global solution, we note that
\begin{align*}
    \norm{\Pi_{N} u(t)}_{\F L^{s,p}(\T)} &\lesssim  N^s \norm{ \wh{\Pi_{N}u}(t,n)}_{\ell^p_n} \\
    &\le N^s C(N,p) \norm{\wh{\Pi_{N}u}(t,n)}_{\ell^2_n } \\
    &= N^s C(N,p) \norm{\wh{\Pi_{N}u_0}(n)}_{\ell^2_n } \\
    &\le N^{s} C(N,p) \norm{\wh{\Pi_{N}u_0}(n)}_{\ell^p_n } \\
    &\le N^{s}C(N,p)  \norm{\langle n\rangle^s \wh{u_0}(n)}_{\ell^p_n }
    \end{align*}
using the fact that $\norm{\Pi_{N}u(t,x)}_{L^2_x}=\norm{\Pi_{N}u_0(x)}_{L^2_x}$, which may be directly checked, Plancherel's theorem, and the equivalence of norms for the finite-dimensional normed spaces $\ell^p(\C^{2N+1})$ and $\ell^2(\C^{2N+1})$ for each $1\le p\le \infty$. Since we have a uniform-in-time bound, we can upgrade the local solution to a global one. We note that these bounds blow-up as $N\to\infty$ and thus we need a different argument to obtain bounds uniform in $N$. \par 

We introduce the main result of this section, the proof of which is inspired by \cite{KNPSV} and utilizes the multilinear estimates proved in \cite{ChapoutoJDDE}. 
\begin{prop}\label{S+1/S 2nd half}
Let $\Phi_N(t)u_0$ be the flow associated to \eqref{mKdV2 truncated} and let $(s, p)$ satisfy $s\ge 1/2$, $2 \leq p<\infty$. Then, there exists $\beta, c>0$ depending on $p$ such that for all $S>0$, 
\begin{align} \label{Uniform in N bound}
    \sup_{\substack{u_0\in B^{s,p}_S, N\in \mathbb{N} \\ |t|\le c \langle S\rangle^{-\beta}}} \norm{\Phi_N(t) u_0}_{\F L^{s,p}(\T)} \le S+1/S.
\end{align}
\end{prop}
\begin{rmk} \label{Remark on S+1/S}

As in \cite{KNPSV} the form of the bound $S+1/S$ here is vital to arguments in Section \ref{Section: Proof of Main Theorem}. Namely, it is important that the coefficient in front of $S$ is not larger than 1. In order to obtain such a bound, we need to employ additional arguments compared to the work  in \cite{KNPSV}. In particular, we first obtain bounds in $X^{s,b}_{p,q}$ spaces which give us a bound of the form $C(S+1/S)$, for $C$ possibly larger than 1. We then redo the analysis directly in the ``energy space" $C([-T, T], \F L^{s,p})$, for some $T=T(S,p)>0$ utilizing the embeddings $Y_0^s(T), Z_0^s(T) \hookrightarrow C([-T, T], \F L^{s,p})$, allowing us to obtain the desired $S+1/S$ bound.  
\end{rmk}
The rest of the subsection is devoted to the proof of Proposition \ref{S+1/S 2nd half}. \par

We first recall the main ideas from \cite{ChapoutoJDDE}. In that work, Chapouto constructs local in-time solutions to the renormalized mKdV2 \eqref{mKdV 2Cauchy} via a ``re-centering" method inspired by \cite{DNY21}. More precisely Chapouto turns the original fixed point equation into a system of equations involving $u$ and a new quantity $w$:
\begin{align}
    u&= w + F(u, w) \label{u[w]} \\ 
    w&= S(t)u_0 \pm D\N(u)- F(u, w) \label{w[u]}.
\end{align}
where $D\N(u)$ is defined via \eqref{Nonlinearity of mKdV2 physical space} and \eqref{Non-truncated D}. We will denote the flow related to $w(t)$ by $w(t)=\Psi(t) u_0$. In \cite{ChapoutoJDDE}, Chapouto constructs $u\in X^{s, 1/2}_{p, 2-}$ and $w\in X^{s, 1-}_{p, \infty-}$ locally in time. The expression $F(u, w)$ is a sum of four trilinear operators, with inputs taking values in $\{w, \overline{w}, u, \overline{u}\}$, each one representing the restriction of the non-resonant part of the nonlinearity to specific regions in frequency space. These terms are chosen so that the the smoothing in space they inherit by construction results in control of the derivative in the nonlinearity. Explicitly, we may write
\begin{align}\label{Section 2 forumula for F}
    F(u,w)= \left[
G_{A,\geq}(w,\overline{u},u) 
+ G_{A,>}(w,u,\overline{u}) 
+ G_{B,\geq}(w,\overline{w},u) 
+ G_{B,>}(w,w,\overline{u}) 
\right],
\end{align}
where $G_{\ast, \ge}, G_{\ast, >}, \ast\in\{A,B\}$ are defined in \eqref{Definition of G_(*,>=)}. With this choice of $F$, given $w$, one can use the bounds in Lemma \ref{Estimates for G} to construct a fixed-point argument solving for $u$ in \eqref{u[w]} in $Y_0^s$. However, the fixed-point argument for $w$ in \eqref{w[u]} does not close  due to failure of the appropriate trilinear estimates needed to control the right hand side of \eqref{w[u]}. To resolve this issue,  Chapouto first performs a \textit{partial second iteration}: in the problematic terms, she substitutes the equation \eqref{u[w]} for $u$, leading to quintilinear terms. Some of the new quintilinear terms can be controlled in the relevant norm. However, for others, another partial iteration must be performed and this results in new septilinear terms, all of which have the virtue of satisfying the necessary estimates in the $Z_0^s$ norm. The full equation for $w$, containing all the exact multilinear forms, on which the contraction mapping is performed is too long to include here.  
However, we can write down a compact version of it as follows:
\begin{align} \label{w in terms of C,Q,S}
     w =  S(t) u_0 
+  C(u, w) 
+  Q(u, w) 
+ \Sigma(u, w).
\end{align}
where $C, Q, \Sigma$ denote respectively the trilinear, quintilinear and septilinear terms in the final equation for $w$, with their arguments belonging to $\{u, w, \overline{u}, \overline{w}\}$. Note that here we absorb the defocusing/focusing $\pm$ sign into the definitions of $C, Q, \Sigma$. Since we will be concerned only with the local theory in this section, the sign is immaterial. 

 We can collect the relevant estimates pertaining to $C, Q, \Sigma$ proved in Section 6 of \cite{ChapoutoJDDE}:
\begin{lem}\label{Estimates for C, Q, S}
For the range $s\ge 1/2, 2\le p<\infty$, the following estimates hold true:
\begin{align} 
    \|C(u, w)\|_{Z_1^s} \lesssim \| u\|_{Y_0^s}^3 + \| w\|_{Z_0^s}^3 + \| w\|_{Z_0^s} \| u\|_{Y_0^s}^2 + \| w\|_{Z_0^s}^2 \| u\|_{Y_0^s} \\
     \|Q(u, w)\|_{Z_1} \lesssim \| u\|_{Y_0^s}^4 \| w\|_{Z_0^s}   + \| w\|_{Z_0^s}^2 \| u\|_{Y_0^s}^3 + \| w\|_{Z_0^s}^3 \| u\|_{Y_0^s}^2+ \| w\|_{Z_0^s}^4 \| u\|_{Y_0^s} \\
     \|\Sigma (u, w)\|_{Z_1} \lesssim \| u\|_{Y_0^s}^5 \| w\|_{Z_0^s}^2   + \| u\|_{Y_0^s}^4 \| w\|_{Z_0^s}^3  + \| u\|_{Y_0^s}^3\| w\|_{Z_0^s}^4 + \| u\|_{Y_0^s}^2\| w\|_{Z_0^s}^5 + \| u\|_{Y_0^s}\| w\|_{Z_0^s}^6.
\end{align}
\end{lem}
We are now ready to present the proof of Proposition \ref{S+1/S 2nd half}.
\begin{proof}[Proof of Proposition \ref{S+1/S 2nd half}]
For our purposes, we will consider the following decomposition, calling $u_N(t)$ the solution to the flow \eqref{mKdV2 truncated}, 
\begin{align} 
    u_N&= w_N + \Pi_N F(\Pi_N u_N, \Pi_N w_N) \label{u_N[w]} \\ 
    w_N&= S(t)u_0 +  \Pi_N C(\Pi_N u_N, \Pi_N w) 
+  \Pi_N Q(\Pi_N u_N , \Pi_N w_N) 
+ \Pi_N \Sigma(\Pi_N u_N , \Pi_N w_N )\label{w_N[u]}.
\end{align}
We denote the flow related to the above equations by $(w_N(t), u_N(t))= (\Psi_N(t) u_0, \Phi_N(t) u_0)$.

To enable us to use the estimates established in \cite{ChapoutoJDDE}, we recast our system of equations for $u_N, w_N$ as follows:
\begin{align} \label{equation for u_n with cutoffs}
      u_N &= w_N + \varphi_\tau\left(\Pi_N F(\Pi_N u_N, \Pi_N w_N)\right) \\ 
w_N &= \varphi S(t)u_0 + \varphi_\tau \left(\Pi_N C(\Pi_N u_N, \Pi_N w) 
+  \Pi_N Q(\Pi_N u_N , \Pi_N w_N) 
+ \Pi_N \Sigma(\Pi_N u_N , \Pi_N w_N )\right). \label{equation for w_n with cutoffs}
\end{align}
Solving \eqref{u_N[w]}, \eqref{w_N[u]} for $t\in[-\tau,\tau], 0<\tau<1$ is equivalent to solving the above pair of equations for $t\in[-\tau,\tau]$.

From Lemma \ref{Estimates for G}, we may deduce the following bound:
\begin{align} \label{bound for u_N Proposition 3.1}
    \|  u_N\|_{Y_0(\tau)}&\le C \| w_N\|_{Y_0(\tau)} + C\tau^{\delta/2}\left( \| \Pi_N u\|_{Y_0^s(\tau)}^2 \|\Pi_N w\|_{Z_0^s(\tau)} + \| \Pi_N u\|_{Y_0^s(\tau)} \| \Pi_N w\|_{Z_0^s(\tau)}^2\right) \\
    &\le  C \| w_N\|_{Z_0^s(\tau)} + C\tau^{\delta/2}\left( \| \Pi_N u\|_{Y_0^s(\tau)}^2 \| \Pi_N w\|_{Z_0^s(\tau)} + \| \Pi_N u\|_{Y_0^s(\tau)} \|\Pi_N w\|_{Z_0^s(\tau)}^2\right),
\end{align}
where we used the continuous embedding $Z_0^s \hookrightarrow Y_0^s$. 
The heart of the proof will consist of proving the control
\begin{align} \label{Final Z_0s bound for w_N}
    \sup_{\substack{u_0\in B^{1/2,p}_S, N\in \mathbb{N} }} \norm{w_N}_{Z_0^s(I)} \le 2C'S .
\end{align}
for some interval $I$. We note that such a bound combined with a standard bootstrap argument applied to \eqref{bound for u_N Proposition 3.1} yields
\begin{align} \label{Bounds on u given bounds on w}
    \sup_{\substack{w\in B^{Z_0}_S, N\in \mathbb{N} }} \norm{u_N(t)}_{Y_0^s(J_S\cap I)} < 4CC'S
\end{align}
where
\begin{align} \label{definition for J_S}
    J_S\equiv &\bigg\{T: |T|^{\delta/2} \big((2C'S)^2 (4CC'S) 
+ (2C'S) (4CC'S)^2 
\big) \le S \bigg\}.
\end{align}

We now focus on proving \eqref{Final Z_0s bound for w_N}. The equation for $w$ given the form of $F(u, w)$ takes the form \eqref{equation for w_n with cutoffs}. \par 
By a direct calculation, and remembering the definition of $Z_0^s$ in \eqref{Z_0^s},  
\begin{align}
    \| \varphi S(t) u_0\|_{Z_0^s(\tau)} &\le  \| \varphi S(t) u_0\|_{Z_0^s(\R)} \\
    &= \| \langle n\rangle^s \langle \lambda -n^3 \rangle^{1-\delta} \widehat{\varphi S(t) u)}(n,\lambda )  \|_{\ell^p_n L^{1/4\delta}_\lambda} \\
    &=\|  \langle n\rangle^s\langle \lambda -n^3 \rangle^{1-\delta} \widehat{\varphi} * \delta(\cdot-n^3) \widehat{u_0}(n)(n,\lambda )  \|_{\ell^p_n L^{1/4\delta}_\lambda} \\
    &= \|  \langle n\rangle^s\langle \lambda -n^3 \rangle^{1-\delta} \widehat{\varphi}(\lambda-n^3)\widehat{u_0}(n)(n,\lambda )  \|_{\ell^p_n L^{1/4\delta}_\lambda} \\
    &= \|   \langle \lambda \rangle^{1-\delta} \widehat{\varphi}(\lambda)  \|_{ L^{1/4\delta}_\lambda} \|u_0\|_{\F L^{s,p}} \\
    &\le  \|\langle \lambda \rangle \widehat{\varphi}(\lambda)  \|_{ L^{1/4\delta}_\lambda} \|u_0\|_{\F L^{s,p}} \\
\end{align}
By interpolation and Young's inequality for products, assuming as we may that $0<\delta\ll 1$,
\begin{align}
    \|\langle \lambda \rangle \widehat{\varphi}(\lambda)  \|_{ L^{1/4\delta}_\lambda} &\le \|\langle \lambda \rangle \widehat{\varphi}(\lambda)  \|_{ L^{1}_\lambda}^{4\delta} \|\langle \lambda \rangle \widehat{\varphi}(\lambda)  \|_{ L^{\infty}_\lambda}^{1-4\delta}  \\
    &\le 4\delta \|\langle \lambda \rangle \widehat{\varphi}(\lambda)  \|_{ L^{1}_\lambda} + (1-4\delta) \|\langle \lambda \rangle \widehat{\varphi}(\lambda)  \|_{ L^{\infty}_\lambda} \\
    &\le  \|\langle \lambda \rangle \widehat{\varphi}(\lambda)  \|_{ L^{1}_\lambda} + \|\langle \lambda \rangle \widehat{\varphi}(\lambda)  \|_{ L^{\infty}_\lambda} \\
    &= \| \langle \lambda \rangle \widehat{\varphi}(\lambda)\|_{L^1_\lambda \cap L^\infty_\lambda}.
\end{align}
Therefore, there exists $C>0$ depending on $\varphi$ such that 
\begin{align}
    \| \varphi S(t) u_0\|_{Z_0^s(\tau)} \le C \|u_0\|_{\F L^{s,p}}.
\end{align}
Next, we have that
\begin{align} 
\begin{aligned}
   \| w_N\|_{Z_0^s(\tau)} &\le  C'\|u_{0}\|_{\mathcal{F}L^{s,p}} 
+ C'\tau^{\delta/2} \left(
\| \Pi_N u_N\|_{Y_0^s(\tau)}^3
+ \| \Pi_N w_N\|_{Z_0^s(\tau)}^3 
+ \| \Pi_N w_N\|_{Z_0^s(\tau)}^2 \| \Pi_N u_N\|_{Y_0^s(\tau)} \right.\\
&\quad
+ \| \Pi_N w_N\|_{Z_0^s(\tau)} \| \Pi_N u_N\|_{Y_0^s(\tau)}^2 
+ \| \Pi_N u_N\|_{Y_0^s(\tau)}^4 \| \Pi_N w_N\|_{Z_0^s(\tau)} 
+ \| \Pi_N w_N\|_{Z_0^s(\tau)}^2 \| \Pi_N u_N\|_{Y_0^s(\tau)}^3 \\
&\quad
+ \| \Pi_N w_N\|_{Z_0^s(\tau)}^3 \| \Pi_N u_N\|_{Y_0^s(\tau)}^2 
+ \| \Pi_N w_N\|_{Z_0^s(\tau)}^4 \| \Pi_N u_N\|_{Y_0^s(\tau)} 
+ \| \Pi_N u_N\|_{Y_0^s(\tau)}^5 \| \Pi_N w_N\|_{Z_0^s(\tau)}^2 \\
&\quad
+ \| \Pi_N u_N\|_{Y_0^s(\tau)}^4 \| \Pi_N w_N\|_{Z_0^s(\tau)}^3 
+ \| \Pi_N u_N\|_{Y_0^s(\tau)}^3 \| \Pi_N w_N\|_{Z_0^s(\tau)}^4 
+ \| \Pi_N u_N\|_{Y_0^s(\tau)}^2 \| \Pi_N w_N\|_{Z_0^s(\tau)}^5 
 \\
 &+ \left.\| \Pi_N u_N\|_{Y_0^s(\tau)} \| \Pi_N w_N\|_{Z_0^s(\tau)}^6 
\right).
\end{aligned} \label{bound on w_N}
\end{align}
We now show \eqref{Final Z_0s bound for w_N}.
We  introduce the set
\begin{align}
    I_S\equiv &\bigg\{T: |T|^{\delta/2} \big( 
(4CC'S)^3 
+ (2C'S)^3 
+ (2C'S)^2 (4CC'S) 
+ (2C'S) (4CC'S)^2 \\
&+ (4CC'S)^4 (2C'S) 
+ (2C'S)^2 (4CC'S)^3 
+ (2C'S)^3 (4CC'S)^2 \\
&+ (2C'S)^4 (4CC'S) 
+ (4CC'S)^5 (2C'S)^2 
+ (4CC'S)^4 (2C'S)^3 \\
&+ (4CC'S)^3 (2C'S)^4 
+ (4CC'S)^2 (2C'S)^5 
+ (4CC'S) (2C'S)^6 
\big) < S \bigg\}, 
\end{align}
and note that $I_S\subseteq J_S$. We claim that
\begin{align}
    \|  w_N \|_{Z_0^s(I_S)} < 2C'S, \quad \forall u_0 \text{ such that } \| \Pi_N u_0\|_{\F L^{s,p}}\le S.
\end{align}
By contradiction (and here we use the continuity of the map $ t\mapsto \| w_N\|_{Z_0^s([-t,t])}$ since $w_N$ is smooth) assume that there exists $u_0\in \F L^{s,p}$ such that  $\| \Pi_N u_0\|_{\F L^{s,p}}\le S$ and $t_0\in I_S$ such that
\begin{align}\label{6}
     \| w_N\|_{Z_0^s([-t_0, t_0])}=2C'S.
\end{align}
Again, a standard bootstrap argument yields that by the above equation and \eqref{bound for u_N Proposition 3.1}.
\begin{align}
    \|  u_N\|_{Y_0^s(J_S\cap [-t_0,t_0])} \le 4CC'S,
\end{align}
where $J_S$ is defined in \eqref{definition for J_S}.
Next, we remark that
\begin{align}
    [-t_0,t_0]= I_S \cap [-t_0, t_0] \subseteq J_S\cap [-t_0, t_0]
\end{align}
and therefore, 
\begin{align}
     \|  u_N\|_{Y_0^s([-t_0, t_0])} \le 4CC'S.
\end{align}

Combining this estimate with the \eqref{bound on w_N} and the definition of $I_S$, we deduce
\begin{align}
    \| w_N\|_{Z_0^s([-t_0, t_0])} < C'S + C'S=2C'S,
\end{align}
contradicting \eqref{6}. 
Next, we  notice that for $c=c_0^{1/\delta}$ for small enough $c_0$  and $\beta=12/\delta$, we have the inclusion
\begin{align}
\{t:|t|\le c \langle S\rangle^{-\beta}\}\subseteq I_S
\end{align}
 Then, we have the following chain of implications
\begin{align}
u_0\in B_S^{s,p} &\implies \|w_N\|_{Z(c\langle S\rangle^{-\beta})}\le 2CS 
\implies \|u_N \|_{Y_0^s(J_S \cap \{|t|\le c\langle S\rangle^{-\beta}\})}\le 4CC'S \\
&\implies \|u_N \|_{Y_0^s(c\langle S\rangle^{-\beta})}\le 4CC'S,
\end{align}
where in the last step, we used the fact that for our specific choice of $c, \beta$ 
\begin{align}
     \{|t|\le c\langle S\rangle^{-\beta}\} \subseteq J_S.
\end{align}
Therefore, by the embedding $Y_0^s(T) \hookrightarrow C([-T,T],\F L^{s,p})$ ,
\begin{align} \label{Y_0s final bound on u_N}
    \sup_{\substack{u_0\in B^{s,p}_S, N\in \mathbb{N} \\ |t|\le c\langle S\rangle^{-\beta}}} \norm{  u_N(t)}_{\F L^{s,p}} \le 4CC'S.
\end{align} 
Unfortunately, this bound is not tight enough for our purpose. To obtain a better one, we take the $C([-T,T], \F L^{s,p})$ norm, with $T\le c \langle S\rangle^{-\beta}$, of the equation \eqref{equation for u_n with cutoffs} linking $u_N$ and $w_N$
\begin{align}
    \| u_N \|_{C([-T,T], \F L^{s,p})} &\le  \|w_N\|_{C([-T,T], \F L^{s,p})} + \|\varphi_T\left(\Pi_N F(\Pi_N u_N, \Pi_N w_N)\right) \|_{C([-T,T], \F L^{s,p})}.
\end{align}
We now use the embedding $ Y_0^s(T) \hookrightarrow C([-T,T], \F L^{s,p})$, the form of $F$ in \eqref{Section 2 forumula for F} and the bounds for estimating $F$ contained in Lemma \ref{Estimates for G} to conclude that 
\begin{align}
     \| u_N \|_{C([-T,T], \F L^{s,p})} &\le  \|w_N\|_{C([-T,T], \F L^{s,p})} + CT^{\delta/2}\left( \| \Pi_N u\|_{Y_0^s(T)}^2 \|\Pi_N w\|_{Z_0^s(T)} + \| \Pi_N u\|_{Y_0^s(T)} \| \Pi_N w\|_{Z_0^s(T)}^2\right).
\end{align}
Using similar arguments on the equation \eqref{equation for w_n with cutoffs} and the unitarity of the free evolution $S(t)=e^{t\partial_x^3}$ in Fourier-Lebesgue spaces, 
\begin{align}
    \|S(t)u_0\|_{\FL^{s,p}} = \|u_0\|_{\FL^{s,p}},
\end{align}
we show that
\begin{align}
    \|w_N\|_{C([-T,T], \F L^{s,p})} &\le \|u_0\|_{\F L^{s,p}} +C'T^{\delta/2} \bigg(\| \Pi_N u\|_{Y_0^s}^3+ \| \Pi_N w\|_{Z_0^s}^3 
+ \| \Pi_N w\|_{Z_0^s}^2 \| \Pi_N u\|_{Y_0^s} 
+ \| \Pi_N w\|_{Z_0^s} \| \Pi_N u\|_{Y_0^s}^2 \\
&+ \| \Pi_N u\|_{Y_0^s}^4 \| \Pi_N w\|_{Z_0^s} 
+ \| \Pi_N w\|_{Z_0^s}^2 \| \Pi_N u\|_{Y_0^s}^3 
+ \| \Pi_N w\|_{Z_0^s}^3 \| \Pi_N u\|_{Y_0^s}^2 
+ \| \Pi_N w\|_{Z_0^s}^4 \| \Pi_N u\|_{Y_0^s} \\
&+ \| \Pi_N u\|_{Y_0^s}^5 \| \Pi_N w\|_{Z_0^s}^2 
+ \| \Pi_N u\|_{Y_0^s}^4 \| \Pi_N w\|_{Z_0^s}^3 
+ \| \Pi_N u\|_{Y_0^s}^3 \| \Pi_N w\|_{Z_0^s}^4 
+ \| \Pi_N u\|_{Y_0^s}^2 \| \Pi_N w\|_{Z_0^s}^5 \\
&+ \| \Pi_N u\|_{Y_0^s} \| \Pi_N w\|_{Z_0^s}^6 \bigg)
\end{align}
and combining the two above bounds, we conclude that
\begin{align}
    \| u_N \|_{C([-T,T], \F L^{s,p})} &\le \|u_0\|_{\F L^{s,p}} + CT^{\delta/2}\left( \| \Pi_N u\|_{Y_0^s(T)}^2 \|\Pi_N w\|_{Z_0^s(T)} + \| \Pi_N u\|_{Y_0^s(T)} \Pi_N w\|_{Z_0^s(T)}^2\right) \\
    &+C'T^{\delta/2} \bigg(\| \Pi_N u\|_{Y_0^s}^3 
+ \| \Pi_N w\|_{Z_0^s}^3 
+ \| \Pi_N w\|_{Z_0^s}^2 \| \Pi_N u\|_{Y_0^s} 
+ \| \Pi_N w\|_{Z_0^s} \| \Pi_N u\|_{Y_0^s}^2 \\
&+\| \Pi_N u\|_{Y_0^s}^4 \| \Pi_N w\|_{Z_0^s} 
+ \| \Pi_N w\|_{Z_0^s}^2 \| \Pi_N u\|_{Y_0^s}^3 
+ \| \Pi_N w\|_{Z_0^s}^3 \| \Pi_N u\|_{Y_0^s}^2 
+ \| \Pi_N w\|_{Z_0^s}^4 \| \Pi_N u\|_{Y_0^s} \\
&+ \| \Pi_N u\|_{Y_0^s}^5 \| \Pi_N w\|_{Z_0^s}^2 
+ \| \Pi_N u\|_{Y_0^s}^4 \| \Pi_N w\|_{Z_0^s}^3 
+ \| \Pi_N u\|_{Y_0^s}^3 \| \Pi_N w\|_{Z_0^s}^4 
+ \| \Pi_N u\|_{Y_0^s}^2 \| \Pi_N w\|_{Z_0^s}^5 \\
&+ \| \Pi_N u\|_{Y_0^s} \| \Pi_N w\|_{Z_0^s}^6 \bigg).
\end{align}
Using the fact that $T\le c \langle S\rangle^{-\beta}$ and our above estimates, we can show that 
\begin{align}
    \| u_N \|_{C([-T,T], \F L^{s,p})} &\le \|u_0\|_{\F L^{s,p}} + CT^{\delta/2}\big((2C'S)^2 (4CC'S) 
+ (2C'S) (4CC'S)^2 
\big) + C' T^{\delta/2}\big( 
(4CC'S)^3 
+ (2C'S)^3 \\
&+ (2C'S)^2 (4CC'S) 
+ (2C'S) (4CC'S)^2 
+ (4CC'S)^4 (2C'S) 
+ (2C'S)^2 (4CC'S)^3 \\
&+ (2C'S)^3 (4CC'S)^2 
+ (2C'S)^4 (4CC'S) 
+ (4CC'S)^5 (2C'S)^2 
+ (4CC'S)^4 (2C'S)^3 \\
&+ (4CC'S)^3 (2C'S)^4 
+ (4CC'S)^2 (2C'S)^5 
+ (4CC'S) (2C'S)^6 
\big).
\end{align}
Now, if we take $T$ even smaller such that $T\le c_1 \langle S \rangle^{-\beta'}$ where $\beta'=18/\delta$ for small enough $c_1(\delta)>0$, we may conclude that 
\begin{align}
     \| u_N \|_{C([-T,T], \F L^{s,p})} &\le \|u_0\|_{\F L^{s,p}} + \frac{1}{S}.
\end{align}
Therefore, 
\begin{align}
    \sup_{\substack{u_0\in B^{s,p}_S, N\in \mathbb{N} \\ |t|\le c_1\langle S\rangle^{-\beta'}}} \norm{  u_N(t)}_{\F L^{s,p}} \le S+1/S.
\end{align}
By abuse of notation, we relabel $\beta=\beta'$ and $c=c_1$. This concludes the proof of Proposition \ref{S+1/S 2nd half}
\end{proof}
\subsection{Long-time convergence of $\Phi_{N}$}
We now present a proposition that is useful in the proof of Theorem \ref{Main Result} that is presented in Section \ref{Section: Proof of Main Theorem}, where we will apply it along a subsequence $\{N_k,k\}_{k\in \mathbb{N}}\subseteq \{N,k\}_{N\in \mathbb{N}, k\in \mathbb{N}}$, where $N_k\to\infty$.
\begin{prop} \label{Long-time convergence proposition}
Let $1/2\le s <s', 2\le p<\infty$ be given and $\{u_{k,0}\}$ be a sequence in $\mathcal{F}L^{s,p}$ such that
\[
u_{k,0} \xrightarrow[k \to \infty]{} u_0 \quad \text{in } \mathcal{F}L^{s,p}.
\]
Assume moreover that for some $T > 0$ we have
\begin{equation}\label{Section 2.1 uniform bound on truncated flow}
\sup_{k, N \in \mathbb{N}} \sup_{t \in [-T,T]} \|\Phi_{N}(t)u_{k,0}\|_{\mathcal{F}L^{s,p}} \le M< \infty,
\end{equation}
for some fixed $M>0$, where $\Phi_N(t)$ is the flow associated to \eqref{mKdV2 truncated}. Then there exists $u \in C([-T,T], \mathcal{F}L^{s,p})$ such that
\begin{equation}\label{2.48}
\sup_{t \in [-T,T]} \| \Phi_{N}(t)u_{k,0} -u(t)\|_{\mathcal{F}L^{s',p}}
\xrightarrow[k,N \to \infty]{} 0
\end{equation}
and $u$ is the unique strong solution to \eqref{mKdV 2Cauchy} in the sense of Definition \ref{Definition of strong solution}. Moreover
\begin{equation}\label{2.49}
\sup_{t \in [-T,T]} \|u(t)\|_{\mathcal{F}L^{s,p}} \leq 
\sup_{k,N \in \mathbb{N}} \sup_{t \in [-T,T]} \|\Phi_{N}(t)u_{k,0}\|_{\mathcal{F}L^{s,p}}.
\end{equation}
\end{prop}

The rest of this subsection is devoted to the proof of the above proposition. To prepare for the proof, we first remark on how to construct a local solution to the full flow $\Phi(t)u_0$ and show how to obtain bounds for the truncated flow $(\Phi_N(t) u_{k,0}, \Psi_N(t) u_{k,0})$, where $\Psi_N(t)$ is the flow associated to \eqref{w_N[u]}, on small intervals of length $2c\langle M \rangle^{-\beta}$, and the we iterate   to cover $[-T,T]$.

In the proof of Proposition \ref{S+1/S 2nd half} in the previous subsection, we demonstrated how to construct a solution to the flow $\Phi_{N}(t)u_{0,k}$ on the interval $\{|t|\le \tau\}$, where $\tau=c \langle M \rangle^{-\beta}$, for initial data $u_{k,0}\in B^{s,p}_M$ with bounds uniform in $N,k$ by solving the following system of equations 
\begin{align} \label{Truncated flow in Section 2.1 eq for u}
    u_{N}&= w_{N} + \varphi_\tau \Pi_{N} F(\Pi_{N} u_{N}, \Pi_{N} w_{N}) \\ \label{Truncated flow in Section 2.1 eq for w}
    w_{N}&= S(t)u_{0,k} + \varphi_\tau \left(\Pi_N C(\Pi_N u_N, \Pi_N w) 
+  \Pi_N Q(\Pi_N u_N , \Pi_N w_N) 
+ \Pi_N \Sigma(\Pi_N u_N , \Pi_N w_N )\right).
    \end{align}
    
By repeating this argument for the full flow $\Phi(t)u_0$, we can also construct a local solution on the same time interval, satisfying \eqref{u[w]}, \eqref{w[u]} by recasting the equations \eqref{u[w]}, \eqref{w in terms of C,Q,S} as follows:
\begin{align}
    u&= w + \varphi_\tau F(u, w) \label{u[w] with phi} \\ 
    w&= S(t)u_0 + \varphi_\tau\left( C(u, w) 
+  Q(u, w) 
+ \Sigma(u, w) \right) \label{w[u] with phi}
\end{align}

By hypothesis, we have the bound
\begin{align} \label{Long-time FL bound }
     \sup_{t\in[-T,T]} \| \Phi_N(t) u_{k,0}\|_{\F L^{s,p}}  \le M
\end{align}
uniformly in $N,k$. From this bound at $t=0$ and our previous analysis leading to \eqref{Final Z_0s bound for w_N}, \eqref{Y_0s final bound on u_N}, we may deduce 
\begin{align}
    \| \Phi_N(t) u_{k,0}\|_{Y^s_0(c \langle M \rangle^{-\beta})} +  \| \Psi_N(t) u_{k,0}\|_{Z^s_0(c \langle M \rangle^{-\beta})} \le (4CC'+2C')M \equiv C''M.
\end{align}
Iterating this process at shifted times  $\pm c\langle M \rangle^{-\beta}h$ for $ h \in \Z \cap \{ l:|l|\le \lceil Tc^{-1}\langle M \rangle^{\beta} \rceil\}$ and by using the hypothesis \eqref{Section 2.1 uniform bound on truncated flow}, we can repeat this argument $\lceil \frac{T}{c\langle M \rangle^{-\beta}} \rceil$ times to cover the entire interval $[-T,T]$ and we will have shown that 
\begin{align}
 \sup_{h\in \Z \cap \{ l:|l|\le \lceil Tc^{-1}\langle M \rangle^{\beta} \rceil\}} &\| \Phi_N(t) u_{k,0}\|_{Y^s_0([(h-1)c \langle M \rangle^{-\beta}, (h+1)c \langle M \rangle^{-\beta}])}  \\
 &+  \| \Psi_N(t) u_{k,0}\|_{Z^s_0([(h-1)c \langle M \rangle^{-\beta}, (h+1)c \langle M \rangle^{-\beta}])} \le  C''M. \label{bounds for truncated flow 2.1}
\end{align}
\textit{Proof of Proposition \ref{Long-time convergence proposition}}
We will prove Proposition \ref{Long-time convergence proposition} via the following sequence of arguments.
\begin{enumerate}

\item In the main and first step we show that taking $k,N\to\infty$ implies that $\Psi_N(t) u_{k,0}\to \Psi(t) u_0$ in $Z_0^{s'}([-c \langle M \rangle^{-\beta}, c \langle M \rangle^{-\beta}])$ and  $\Phi_N(t) u_{k,0}\to \Phi(t) u_0$ in $Y_0^{s'}[-c \langle M \rangle^{-\beta}, c \langle M \rangle^{-\beta}]$ for $c=c(\delta)>0$ small enough, $1/2\le s'<s$, where $\Phi(t) u_0,\Psi(t)u_0$ are the unique solutions to \eqref{u[w] with phi}, \eqref{w[u] with phi} (and therefore \eqref{u[w]}, \eqref{w in terms of C,Q,S}) \label{item 1 2.1}
 \item Then, by the embedding $Y_0^{s'}(T_0)\hookrightarrow C([-c \langle M \rangle^{-\beta},c \langle M \rangle^{-\beta}],\F L^{s',p})$, 
\begin{align}
    \lim_{N,k \to\infty }\sup_{t\in (-c \langle M \rangle^{-\beta}, c \langle M \rangle^{-\beta})} \| \Phi_{N}(t) u_{k,0}- \Phi(t) u_0\|_{\F L^{s',p}} =0
\end{align}
and therefore by \eqref{Section 2.1 uniform bound on truncated flow}, 
\begin{align}
    \|\Phi(\pm c \langle M \rangle^{-\beta})u_0\|_{\FL^{s',p}} =\lim_{k,N\to\infty} \|\Phi_{N}(\pm c \langle M \rangle^{-\beta})u_{k,0}\|_{\FL^{s',p}} \le M.
\end{align}
Now, by a contraction argument, on the intervals $[0, 2c \langle M \rangle^{-\beta}],[- 2c \langle M \rangle^{-\beta},0]$, we can construct a solution $(w(t), u(t))=(\Psi(t) u_0,\Phi(t)u_0)$ to \eqref{u[w]}, \eqref{w[u]}, with bounds as in \eqref{bounds for truncated flow 2.1} with $h=1$ using a recasting similar to \eqref{u[w] with phi}, \eqref{w[u] with phi} \label{item 2 2.1}
\item By the translation invariance of the truncated and full flows and noting the initial data convergence $\Phi_N(\pm c \langle M \rangle^{-\beta})u_{0,k}\to \Phi(\pm c \langle M \rangle^{-\beta})u_{0}$ due to item (\ref{item 2 2.1}), we deduce thanks to item (\ref{item 1 2.1}) that
\begin{align}
    \Psi_{N}(t)u_{k,0}\to \Psi(t)u_0 \text{ in } Z_0^{s'}[0, 2c \langle M \rangle^{-\beta}], \quad \Phi_{N}(t)u_{k,0}\to \Phi(t)u_0 \text{ in } Y_0^{s'}[0, 2c \langle M \rangle^{-\beta}].
\end{align}

and similarly for the interval $[- 2c \langle M \rangle^{-\beta},0]$.
\item By iterating this argument until the entire interval $[-T,T]$ is covered, we will have constructed a unique strong solution $(u,w)=(\Phi(t)u_0, \Psi(t) u_0)$ to \eqref{u[w]}, \eqref{w[u]} in $C([-T,T],\F L^{s',p})$ such that
\begin{align}
    \lim_{N,k\to\infty} \sup_{t\in [-T,T]} \| \Phi_{N}(t) u_{k,0}- \Phi(t) u_0\|_{\F L^{s',p}} = 0.
\end{align}
\item By the weak lower-semi-continuity of $\F L^{s,p}$ combined with the fact that $\Phi_{N}(t)u_{k,0} \rightharpoonup u(t)$ weakly in $\F L^{s,p}$ as $N,k\to\infty$, we conclude that
\begin{align} \label{bound in item (5) Section 3}
    \sup_{t\in[-T,T]} \|\Phi(t)u_0\|_{\FL^{s,p}} \le \sup_{t\in[-T,T]} \liminf_{k,N\to\infty} \|\Phi_{N}(t)u_{k,0}\|_{\FL^{s,p}} \le M,
\end{align}
which in turn implies \eqref{2.49}.
\item Finally, we show that $u\in C([-T,T], \FL^{s,p})$. \label{item 6}
\end{enumerate}

We will start by proving item (1). Then items (2), (3), (4) and (5) follow as explained above. The subsection will conclude with a proof of item (6).

In order to prove item (\ref{item 1 2.1}), we will first state the following two claims.
\textbf{Claim} A: 
    For $|\tau| \le c \langle M \rangle^{-\beta}$ for $c,\beta$ as in Proposition \ref{S+1/S 2nd half}, we have that
\begin{align} \label{Claim 1}
    \| u-\Phi_N(t)u_{k,0}\|_{Y^{s'}_0(\tau)} &\le \| w-\Psi_N(t)u_{k,0}\|_{Z^{s'}_0(\tau)} + 4C(M)\tau^{\delta/2} \| w-\Psi_N(t)u_{k,0}\|_{Z^{s'}_0(\tau)}\\ &+ 4C(M)\tau^{\delta/2} \| u-\Phi_N(t)u_{k,0}\|_{Y^{s'}_0(\tau)} 
    + 4\tau^{\delta/2} C(M) N^{-(s-s')}
\end{align}
with $C(M)$ a polynomial of degree at most 2 in $M$.

\textbf{Claim B} 
    For $|\tau| \le c \langle M \rangle^{-\beta}$ for $c,\beta$ as in Proposition \ref{S+1/S 2nd half}, we have that
\begin{align} \label{Claim 2}
    \|w-\Psi_N(t)u_{k,0}\|_{Z_0^{s'}(\tau)} &\le C \|u_0-u_{0,k}\|_{\F L^{s',p}} + C(M)\tau^{\delta/2} \| w-\Psi_N(t)u_{k,0}\|_{Z^{s'}_0(\tau)} \\& + C(M)\tau^{\delta/2} \| u-\Phi_N(t)u_{k,0}\|_{Y^{s'}_0(\tau)} 
    + \tau^{\delta/2} C(M) N^{-(s-s')}
\end{align}
with $C(M)$ a polynomial of degree at most 6 in $M$. 

Let us show how item (\ref{item 1 2.1}) follows from the two preceding claims.
We first take 
\begin{align} \label{eq for tau}
    \tau\le  c_0(\delta) \langle M \rangle^{-\beta}, \quad \beta=12/\delta
\end{align}
where $c_0$ is small enough and independent of $M$ such that 
\begin{align} \label{eq for c_0}
    c_0^{\delta/2} 4 \sup_{M\ge 0}C(M)\langle M\rangle^{-6}\le 1/2.
\end{align} We then deduce that, due to \textbf{Claim} A, 
\begin{align}
    \| u-\Phi_N(t)u_{k,0}\|_{Y^{s'}_0(\tau)} &\le \| w-\Psi_N(t)u_{k,0}\|_{Z^{s'}_0(\tau)} + 4C(M)\tau^{\delta/2} \| w-\Psi_N(t)u_{k,0}\|_{Z^{s'}_0(\tau)} \\
    &+ \frac{1}{2} \| u-\Phi_N(t)u_{k,0}\|_{Y^{s'}_0(\tau)} + 4\tau^{\delta/2} C(M) N^{-(s-s')}
\end{align}
Absorbing the term $\frac{1}{2} \| u-\Phi_N(t)u_{k,0}\|_{Y^{s'}_0(\tau)}$ into the left hand side yields
\begin{align} \label{Estimate for u-u_N_k 2.1}
      \frac{1}{2}\| u-\Phi_N(t)u_{k,0}\|_{Y^{s'}_0(\tau)}&\le \| w-\Psi_N(t)u_{k,0}\|_{Z^{s'}_0(\tau)} + 4C(M)\tau^{\delta/2} \| w-\Psi_N(t)u_{k,0}\|_{Z^{s'}_0(\tau)} +  4\tau^{\delta/2} C(M) N^{-(s-s')} \\
      &=(1+4C(M)\tau^{\delta/2})\| w-\Psi_N(t)u_{k,0}\|_{Z^{s'}_0(\tau)} +  4\tau^{\delta/2} C(M) N^{-(s-s')}
\end{align}

We now  insert the above estimate for $u-\Phi_N(t)u_{k,0}$ into the right hand side of the estimate in \textbf{Claim B}:
\begin{align}
    \|w-\Psi_N(t)u_{k,0}\|_{Z_0^{s'}(\tau)} &\le 
 C \|u_0-u_{0,k}\|_{\F L^{s',p}} + C(M)\tau^{\delta/2}(3+8C(M)\tau^{\delta/2})\| w-\Psi_N(t)u_{k,0}\|_{Z^{s'}_0(\tau)} \\
    &+ 8(C(M)\tau^{\delta/2})^2 N^{-(s-s')}
    + \tau^{\delta/2} C(M) N^{-(s-s')},
\end{align}
for some $C(M)>0$ a polynomial of degree at most 6. Thanks to \eqref{eq for tau} and \eqref{eq for c_0}, we can absorb $w-\Psi_N(t)u_{k,0}$ on the right hand side to the left hand side, after which we conclude that
\begin{align}
    \frac{1}{2}\|w-\Psi_N(t)u_{k,0}\|_{Z_0^{s'}(\tau)} &\le C \|u_0-u_{0,k}\|_{\F L^{s',p}} + \tau^{\delta/2} C(M) N^{-(s-s')}.
\end{align}
Letting $k,N\to\infty$, we conclude that $ \Psi_N(t)u_{k,0}\to w$ in $Z_0^{s'}(\tau)$. Letting $k,N\to\infty$ in \eqref{Estimate for u-u_N_k 2.1}, we conclude that $\Phi_N(t)u_{k,0} \to u$ in $Y_0^{s'}(\tau)$ if we take $\tau= c(\delta) \langle M \rangle^{-\beta} $ for small enough $c(\delta)>0$. This finishes the proof of item \eqref{item 1 2.1}. 

We now present the proofs of the two claims. We will denote $w_{N,k}= \Psi_N(t)u_{k,0}, \quad u_{N,k}=\Phi_N(t)u_{k,0}$.

\begin{proof}[Proof of \textbf{Claim} A]

We note that thanks to \eqref{u[w] with phi} and \eqref{Truncated flow in Section 2.1 eq for u},
\begin{align} \label{Equation for u-u_N_k}
    u-u_{N,k} = w-w_{N,k} + \varphi_\tau \left[F(u, w) - \Pi_{N}F(\Pi_{N} u_{N,k}, \Pi_{N} w_{N,k})\right].
\end{align}
Now, by recalling \eqref{Section 2 forumula for F}, $F(u,w)$ consists of a sum of four trilinear operators. One of them is $ \,G_{A, \ge}(w, \overline{u},u)$, where $G_{A, \ge}$ is defined in \eqref{Definition of G_(*,>=)}. We will first demonstrate how to bound
\begin{align}
   \varphi_\tau \left(G_{A, \ge}(w, \overline{u},u) -\Pi_{N}G_{A, \ge}(\Pi_{N}w_{N,k}, \overline{\Pi_{N}u_{N,k}},\Pi_{N}u_{N,k}\right).
\end{align}
Note that we may write 
\begin{align}
    &G_{A, \ge}(w, \overline{u},u) -\Pi_{N}G_{A, \ge}(\Pi_{N}w_{N,k}, \overline{\Pi_{N}u_{N,k}},\Pi_{N}u_{N,k}) \\
    &=\left[G_{A, \ge}(w, \overline{u},u) - G_{A, \ge}(w_{N,k}, \overline{u_{N,k}},u_{N,k})\right] -\left(\Pi_{N}G_{A, \ge}(\Pi_{N}w_{N,k}, \overline{\Pi_{N}u_{N,k}},\Pi_{N}u_{N,k}) -  G_{A, \ge}(w_{N,k}, \overline{u_{N,k}},u_{N,k}) \right) \\
    &\equiv  \left[D_1\right]- (D_2).
\end{align}
We will show that 
\begin{align} \label{E_1-E_2}
    \| \varphi_\tau D_1\|_{Y_0^s(\tau)}  &\le C(M)\tau^{\delta/2} \| w-w_{N,k}\|_{Z^{s'}_0(\tau)} + C(M)\tau^{\delta/2} \| u-u_{N,k}\|_{Y^{s'}_0(\tau)}
\end{align}
and that 
\begin{align}
    \|\varphi_\tau D_2\|_{Y_0^s(\tau)} \le \tau^{\delta/2} C(M) N^{-(s-s')},
\end{align}
for $C(M)$ a quadratic polynomial in $M$ and $s>s'\ge 1/2$.
In turn, by the triangle inequality, this will give us a bound on 
\begin{align}
    \left\| G_{A, \ge}(w, \overline{u},u) -\Pi_{N}G_{A, \ge}(\Pi_{N}w_{N,k}, \overline{\Pi_{N}u_{N,k}},\Pi_{N}u_{N,k})\right\|_{Y_0^s}. 
\end{align}

This reduction is useful because it isolates on one hand the convergence of $(u_{N,k}, w_{N,k})\to (u,w)$, which is the main mechanism behind the bounds for $E_1$ and on the other hand, the convergence of the projections $\Pi_{N} \to \operatorname{Id}$ as $N\to\infty$ when applied to $u_{N,k}, w_{N,k}$ in the $\FL^{s',p}$ norm, where $\operatorname{Id}$ is the identity, which is the main mechanism driving the estimate for $E_2$.

Now, we will expand the term $G_{A, \ge}(w, \overline{u},u)$ into eight sub-term, with the first seven containing at least one difference $w-w_{N,k}$ or $u-u_{N,k}$ in one of its arguments:
\begin{align}
    G_{A, \ge}(w, \overline{u},u)  
    &= G_{A, \ge}(w-w_{N,k} + w_{N,k}, \overline{u-u_{N,k}} + \overline{u_{N,k}},u-u_{N,k} +u_{N,k} )  \\
    &=G_{A, \ge}(w-w_{N,k}, \overline{u-u_{N,k}} ,u-u_{N,k} ) \\
    &+G_{A, \ge}(w-w_{N,k}, \overline{u-u_{N,k}} ,u_{N,k} ) \\
    &+ G_{A, \ge}(w-w_{N,k}, \overline{u_{N,k}},u-u_{N,k}  ) \\
    &+G_{A, \ge}(w-w_{N,k}, \overline{u_{N,k}},u_{N,k} ) \\
    &+ G_{A, \ge}(w_{N,k}, \overline{u-u_{N,k}} ,u-u_{N,k} )\\
    &+G_{A, \ge}(w_{N,k}, \overline{u-u_{N,k}} ,u_{N,k} ) \\
    &+ G_{A, \ge}(w_{N,k}, \overline{u_{N,k}},u-u_{N,k} )\\
    &+G_{A, \ge}(w_{N,k}, \overline{u_{N,k}}, u_{N,k} ).
\end{align}
Therefore,
\begin{align}\label{Form of D_1 eq 1}
       D_1&=G_{A, \ge}(w, \overline{u},u)- G_{A, \ge}(w_{N,k}, \overline{u_{N,k}},u_{N,k}) \\
       &=G_{A, \ge}(w-w_{N,k}, \overline{u-u_{N,k}} ,u-u_{N,k} ) 
    +G_{A, \ge}(w-w_{N,k}, \overline{u-u_{N,k}} ,u_{N,k} ) \\
    &+ G_{A, \ge}(w-w_{N,k}, \overline{u_{N,k}},u-u_{N,k}  ) 
    +G_{A, \ge}(w-w_{N,k}, \overline{u_{N,k}},u_{N,k} ) \\
    &+ G_{A, \ge}(w_{N,k}, \overline{u-u_{N,k}} ,u-u_{N,k} )
    +G_{A, \ge}(w_{N,k}, \overline{u-u_{N,k}} ,u_{N,k} ) \\
    &+ G_{A, \ge}(w_{N,k}, \overline{u_{N,k}},u-u_{N,k} ).
\end{align}

Taking the $Y_0^{s'}$ norm and utilizing Lemma \ref{Lemma for F}, we find that
\begin{align}
    \|  \varphi_\tau D_1\|_{Y^{s'}_0(\tau)} &\lesssim \tau^{\delta/2} \| D_1\|_{Y^{s'}_1(\tau)}\\
    &\lesssim \tau^{\delta/2}\sum_{\substack{x\in  \{w_{N,k}, w-w_{N,k}\} \\ (y, z)\in \{ u_{N,k}, u- u_{N,k}\}^2 \\ (x,y,z)\neq (w_{N,k}, u_{N,k}, u_{N,k}) }} \| x\|_{Z_0^{s'}} \| y\|_{Y_0^{s'}}\| z\|_{Y_0^{s'}(\tau)} \label{Form of D_1 eq 2}.
\end{align}
Note that in \eqref{Form of D_1 eq 2}, there is always at least one factor of either $u-u_{N,k}$ or $w-w_{N,k}$ in each of the term. We will utilize this remark in showing that $w-w_{N,k}\to 0$ in $Z_0^s$ if $u_0-u_{0,k}\to 0$ in $Y_0^s$. 
\par 

We deduce that
\begin{align}\| \varphi_\tau D_1\|_{Y^{s'}_0(\tau)} 
&\lesssim \tau^{\delta/2} \bigg(\| w-w_{N,k}\|_{Z^{s'}_0(\tau)} \, \|u-u_{N,k}\|_{Y^{s'}_0(\tau)}^{2} \\
&+ \| w-w_{N,k}\|_{Z^{s'}_0(\tau)} \, \|u-u_{N,k}\|_{Y^{s'}_0(\tau)} \, \|u_{N,k}\|_{Y^{s'}_0(\tau)}  
+ \| w-w_{N,k}\|_{Z^{s'}_0(\tau)} \, \|u_{N,k}\|_{Y^{s'}_0(\tau)}^{2} \\
& + \|w_{N,k}\|_{Z^{s'}_0(\tau)} \, \|u-u_{N,k}\|_{Y^{s'}_0(\tau)}^{2} 
 + \|w_{N,k}\|_{Z^{s'}_0(\tau)} \, \|u_{N,k}\|_{Y^{s'}_0(\tau)} \, \|u-u_{N,k}\|_{Y^{s'}_0(\tau)}\bigg) .
\end{align}
Now, if we take $\tau$ such that $\tau\le c \langle M \rangle^{-\beta}$ for small enough $c>0, \beta=12/\delta$, then by the estimates shown during the proof of Proposition \ref{S+1/S 2nd half}, we have that
\begin{align}
     \|w_{N,k}\|_{Z^{s}_0(\tau)} \le 2C' M \\
     \|u_{N,k}\|_{Y^{s}_0(\tau)} \le 4C C'M \\
     \|u \|_{Y^{s}_0(\tau)} \le 4CC' M.
\end{align}
Therefore, by the embeddings $Y^{s}_0 \hookrightarrow Y^{s'}_0, Z^{s}_0 \hookrightarrow Z^{s'}_0$, it follows that, by the triangle inequality, 
\begin{align} \label{E_1}
    \| \varphi_\tau D_1\|_{Y^{s'}_0(\tau)} &\le C(M)\tau^{\delta/2} \| w-w_{N,k}\|_{Z^{s'}_0(\tau)} + C(M)\tau^{\delta/2} \| u-u_{N,k}\|_{Y^{s'}_0(\tau)} 
\end{align}
for some $C(M)>0$ a quadratic polynomial in $M$.
The same bounds hold true for $ G_{A, >}(w, \overline{u},u)- G_{A, >}(w_{N,k}, \overline{u_{N,k}},u_{N,k})$.

Let us now turn to estimating
\begin{align}
    D_2=\Pi_{N}G_{A, \ge}(\Pi_{N}w_{N,k}, \overline{\Pi_{N}u_{N,k}},\Pi_{N}u_{N,k}) -  G_{A, \ge}(w_{N,k},\overline{u_{N,k}},u_{N,k}).
\end{align}
We may write
\begin{align}
    &D_2=\Pi_{N}G_{A, \ge}(\Pi_{N}w_{N,k}, \overline{\Pi_{N}u_{N,k}},\Pi_{N}u_{N,k}) -  G_{A, \ge}(w_{N,k},\overline{u_{N,k}},u_{N,k}) \\
    &= \Pi_{N}\left(G_{A, \ge}(\Pi_{N}w_{N,k}, \overline{\Pi_{N}u_{N,k}},\Pi_{N}u_{N,k}) -  G_{A, \ge}(w_{N,k},\overline{u_{N,k}},u_{N,k})\right) - \Pi_{>N}G_{A, \ge}(w_{N,k},\overline{u_{N,k}},u_{N,k}) \\
    &\equiv D_{2,1}- D_{2,2}
\end{align}
We can write $E_{2,2}$ as
\begin{align}
    D_{2,2} &=\Pi_{>N}G_{A, \ge}(w_{N,k},\overline{u_{N,k}},u_{N,k}) \\
    &= \Pi_{>N}G_{A, \ge}(\Pi_{N/4} w_{N,k} + \Pi_{>N/4} w_{N,k}, \overline{\Pi_{N/4} u_{N,k}} + \overline{\Pi_{>N/4} u_{N,k}},
 \Pi_{N/4} u_{N,k} +\Pi_{>N/4} u_{N,k}).
\end{align}
One can see that after expanding out fully the above expression, all terms of $\Pi_{>N}G_{A, \ge}$ which have at least one argument at high frequency survive. Furthermore, the only term which does not share this property vanishes. Indeed, 
\begin{align}
    \Pi_{>N}G_{A, \ge}(\Pi_{N/4} w_{N,k},\overline{\Pi_{N/4}u_{N,k}},\Pi_{N/4}u_{N,k})=0
\end{align}
by utilizing disjoint supports. Therefore, by a similar argument we employed above, using the norms bounds for $G_{A, \ge}$ emanating from Lemma \ref{Estimates for G}, we deduce that
\begin{align}\label{9}
 \| \varphi_\tau D_{2,2}\|_{Y_0^{s'}(\tau)}
&\lesssim \tau^{\delta/2} \|  D_{2,2}\|_{Y_1^{s'}(\tau)} \\
&\lesssim \tau^{\delta/2} \bigg(\|\Pi_{N/4} w_{N,k}\|_{Z_0^{s'}(\tau)} \,\|\Pi_{N/4} u_{N,k}\|_{Y_0^{s'}(\tau)} \,\|\Pi_{>N/4} u_{N,k}\|_{Y_0^{s'}(\tau)} \\
&+ \|\Pi_{N/4} w_{N,k}\|_{Z_0^{s'}(\tau)} \,\|\Pi_{>N/4} u_{N,k}\|_{Y_0^{s'}(\tau)} \,\|\Pi_{N/4} u_{N,k}\|_{Y_0^{s'}(\tau)} \\
&+ \|\Pi_{N/4} w_{N,k}\|_{Z_0^{s'}(\tau)} \,\|\Pi_{>N/4} u_{N,k}\|_{Y_0^{s'}(\tau)}^{2} \\
&+ \|\Pi_{>N/4} w_{N,k}\|_{Z_0^{s'}(\tau)} \,\|\Pi_{N/4} u_{N,k}\|_{Y_0^{s'}(\tau)}^{2} \\
&+ \|\Pi_{>N/4} w_{N,k}\|_{Z_0^{s'}(\tau)} \,\|\Pi_{N/4} u_{N,k}\|_{Y_0^{s'}(\tau)} \,\|\Pi_{>N/4} u_{N,k}\|_{Y_0^{s'}(\tau)} \\
&+ \|\Pi_{>N/4} w_{N,k}\|_{Z_0^{s'}(\tau)} \,\|\Pi_{>N/4} u_{N,k}\|_{Y_0^{s'}(\tau)}^{2}\bigg).
\end{align}
Now, we note that, by the definition of $Y_0^s, Z_0^s$, 
\begin{align}
     \|\Pi_{>N/4} w_{N,k}\|_{Z_0^{s'}(\tau)} &\lesssim  N^{-(s-s')} \| w_{N,k}\|_{Z_0^{s}(\tau)} \label{gain small, negative power of N w eq 1} \\
     \|\Pi_{>N/4} u_{N,k}\|_{Y_0^{s'}(\tau)} &\lesssim  N^{-(s-s')} \| u_{N,k}\|_{Y_0^{s}}(\tau) \label{gain small, negative power of N u eq 1}.
\end{align}
If  we furthermore take $\tau\le c \langle M \rangle^{-\beta}$ , where $c, \beta>0$ are as in the Proof of Proposition \ref{S+1/S 2nd half}, we may deduce that
\begin{align}
    \| w_{N,k}\|_{Z_0^{s}(\tau)} &\le \| w_{N,k}\|_{Z_0^{s}( c \langle M \rangle^{-\beta})} \le 2C'M \\
    \| u_{N,k}\|_{Y_0^{s}(\tau)} &\le \| u_{N,k}\|_{Y_0^{s}( c \langle M \rangle^{-\beta})} \le 4C'CM,
\end{align}
where $C, C'>0$ are as in the proof of Proposition \ref{S+1/S 2nd half}. Therefore, we deduce that, combining the above inequalities, 
\begin{align}
    \|\Pi_{>N/4} w_{N,k}\|_{Z_0^{s'}(\tau)} &\le N^{-(s-s')} C(M) \label{gain small, negative power of N w eq 2} \\
     \|\Pi_{>N/4} u_{N,k}\|_{Y_0^{s'}(\tau)} &\le N^{-(s-s')} C(M), \label{gain small, negative power of N u eq 2}
\end{align}
which, combined with \eqref{9}, leads to  
\begin{align} \label{E_(2,2)}
     \| \varphi_\tau D_{2,2} \|_{Y_0^{s'}(\tau)} \le  \tau^{\delta/2} C(M) N^{-(s-s')}.
\end{align}
We now turn to estimating $E_{2,1}$: by Lemma \ref{Lemma for F} and the form of $E_{2,1}$, 
\begin{align}
&\| \varphi_\tau D_{2,1} \|_{Y_0^{s'}(\tau)} \\
&\lesssim \tau^{\delta/2} \|  D_{2,1} \|_{Y_1^{s'}(\tau)} \\
    &\lesssim \tau^{\delta/2}\| \Pi_{N}\left(G_{A, \ge}(\Pi_{N}w_{N,k}, \overline{\Pi_{N}u_{N,k}},\Pi_{N}u_{N,k}) -  G_{A, \ge}(w_{N,k},\overline{u_{N,k}},u_{N,k})\right)\|_{Y^{s'}_1(\tau)} \\ &\lesssim \tau^{\delta/2} \| \left(G_{A, \ge}(\Pi_{N}w_{N,k}, \overline{\Pi_{N}u_{N,k}},\Pi_{N}u_{N,k}) -  G_{A, \ge}(w_{N,k},\overline{u_{N,k}},u_{N,k})\right)\|_{Y^{s'}_1(\tau)} \\
    &\lesssim \tau^{\delta/2}\sum_{\substack{x\in  \{\Pi_N w_{N,k}, \Pi_{>N} w_{N,k}\} \\ (y, z)\in \{ \Pi_N u_{N,k}, \Pi_{>N} u_{N,k}\}^2 \\ (x,y,z)\neq (\Pi_N w_{N,k}, \Pi_Nu_{N,k},\Pi_N u_{N,k}) }} \| x\|_{Z_0^{s'}(\tau)} \| y\|_{Y_0^{s'}(\tau)}\| z\|_{Y_0^{s'}(\tau)}.
\end{align}
Using the bounds \eqref{gain small, negative power of N w eq 2}, \eqref{gain small, negative power of N u eq 2} (with $N$ replacing $N/4$) we may deduce that for $\tau\le c \langle M \rangle^{-\beta}$,  
\begin{align} \label{E_(2,1)}
    \| \varphi_\tau D_{2,1} \|_{Y_0^{s'}(\tau)} \le \tau^{\delta/2} C(M) N^{-(s-s')}.
\end{align}
Combining the estimates \eqref{E_1}, \eqref{E_(2,2)}, \eqref{E_(2,1)} with \eqref{E_1-E_2}, we conclude that
\begin{align} \label{G_(A,>=)}
     &\| \varphi_\tau \left(G_{A, \ge}(w, \overline{u},u) -\Pi_{N}G_{A, \ge}(\Pi_{N}w_{N,k}, \overline{\Pi_{N}u_{N,k}},\Pi_{N}u_{N,k})\right)\|_{Y^{s'}_0(\tau)} \\ &\le \|\varphi_\tau D_1 \|_{Y^{s'}_0(\tau)}  +  \|\varphi_\tau D_{2,1} \|_{Y^{s'}_0(\tau)} + \|\varphi_\tau D_{2,2} \|_{Y^{s'}_0(\tau)} \\
    &\le C(M)\tau^{\delta/2} \| w-w_{N,k}\|_{Z^{s'}_0(\tau)} + C(M)\tau^{\delta/2} \| u-u_{N,k}\|_{Y^{s'}_0(\tau)} \\
    &+ \tau^{\delta/2} C(M) N^{-(s-s')}.
\end{align}Using similar argument, we may in fact show that 
\begin{align} \label{G_(A,>)}
    &\| \varphi_\tau \left(G_{A,> }(w, \overline{u},u) -\Pi_{N}G_{A,>}(\Pi_{N}w_{N,k}, \overline{\Pi_{N}u_{N,k}},\Pi_{N}u_{N,k})\right)\|_{Y^{s'}_0(\tau)} \\
    &\le C(M)\tau^{\delta/2} \| w-w_{N,k}\|_{Z^{s'}_0(\tau)} + C(M)\tau^{\delta/2} \| u-u_{N,k}\|_{Y^{s'}_0(\tau)} \\
    &+ \tau^{\delta/2} C(M) N^{-(s-s')}
\end{align}
and
\begin{align} \label{G_B}
     &\| \varphi_\tau \left(G_{B,\#}(w, \overline{w},u) -\Pi_{N}G_{B,\#}(\Pi_{N}w_{N,k}, \overline{\Pi_{N}w_{N,k}},\Pi_{N}w_{N,k})\right)\|_{Y^{s'}_0(\tau)} \\
     &\le C(M)\tau^{\delta/2} \| w-w_{N,k}\|_{Z^{s'}_0(\tau)} + C(M)\tau^{\delta/2} \| u-u_{N,k}\|_{Y^{s'}_0(\tau)} \\
    &+ \tau^{\delta/2} C(M) N^{-(s-s')}
\end{align}
For any  $\# \in\{>, \ge \}$ and some possibly larger $C(M)>0$. 

We now conclude that by combining, \eqref{G_(A,>=)}, \eqref{G_(A,>)} and \eqref{G_B} as well as the form of $F$
\begin{align} \label{Section 2.2 estimate for F}
    \| \varphi_\tau \left[F(u, w) - \Pi_{N}F(\Pi_{N} u_{N,k}, \Pi_{N} w_{N,k})\right]\|_{Y_0^{s'}(\tau)} &\le  4C(M)\tau^{\delta/2} \| w-w_{N,k}\|_{Z^{s'}_0(\tau)} + 4C(M)\tau^{\delta/2} \| u-u_{N,k}\|_{Y^{s'}_0(\tau)} \\
    &+ 4\tau^{\delta/2} C(M) N^{-(s-s')}.
\end{align}

By the triangle inequality applied to \eqref{Equation for u-u_N_k}, the embedding $Z_0^{s'} \hookrightarrow Y_0^{s'}$, and the estimate \eqref{Section 2.2 estimate for F}, we obtain 
\begin{align}\label{16}
    \| u-u_{N,k}\|_{Y^{s'}_0(\tau)} &\le \| w-w_{N,k}\|_{Z^{s'}_0(\tau)} + 4C(M)\tau^{\delta/2} \| w-w_{N,k}\|_{Z^{s'}_0(\tau)} + 4C(M)\tau^{\delta/2} \| u-u_{N,k}\|_{Y^{s'}_0(\tau)} \\
    &+ 4\tau^{\delta/2} C(M) N^{-(s-s')}.
\end{align}
This concludes the proof of \textbf{Claim} A.
\end{proof}
We now prove \textbf{Claim} B.

\begin{proof}[Proof of \textbf{Claim} B]

 Using \eqref{w[u] with phi} and \eqref{Truncated flow in Section 2.1 eq for w}, we can show that 
\begin{align}\label{15}
     w-w_{N,k} &= \varphi S(t)\,(u_0-u_{0,k}) 
+ \varphi_\tau \left( C(u,w)-\,\Pi_{N}C(\Pi_{N} u, \Pi_{N} w)\right)  \\
&\quad+ \varphi_\tau \left( Q(u,w)-\,\Pi_{N} Q(\Pi_{N} u, \Pi_{N} w)\right) \\
&\quad+ \varphi_\tau \left( \Sigma(u,w)-\,\Pi_{N} \Sigma(\Pi_{N} u, \Pi_{N} w)\right).
\end{align}
Using an argument similar to the one deployed in showing
\eqref{Section 2.2 estimate for F} and the estimate for $C$ in Lemma \ref{Estimates for C, Q, S}, we may show that
\begin{align} \label{convergence for C}
    \| \varphi_\tau \left[C(u, w) - \Pi_{N}C(\Pi_{N} u_{N,k}, \Pi_{N} w_{N,k})\right]\|_{Z_0^{s'}(\tau)} &\le  C(M)\tau^{\delta/2} \| w-w_{N,k}\|_{Z^{s'}_0(\tau)} + C(M)\tau^{\delta/2} \| u-u_{N,k}\|_{Y^{s'}_0(\tau)} \\
    &+ \tau^{\delta/2} C(M) N^{-(s-s')}
\end{align}
for some possibly larger $C(M)>0$ a quadratic polynomial in $M$. The key idea is to again write
\begin{align} \label{Section 2.2: difference of cubic}
    &C(u, w) - \Pi_{N}C(\Pi_{N} u_{N,k}, \Pi_{N} w_{N,k}) \\
    &= C(u,w)-C(w_{N,k}, u_{N,k}) - \left[ C(\Pi_{N} u_{N,k}, \Pi_{N} w_{N,k}) -C(w_{N,k}, u_{N,k}) \right]
\end{align}
and then to notice that on the right hand side of \eqref{Section 2.2: difference of cubic}, $C(u,w)-C(w_{N,k}, u_{N,k})$ is a sum of trilinear operators with at least one input belonging to $\{u-u_{N,k},w-w_{N,k}\}$; similarly, $C(\Pi_{N} u_{N,k}, \Pi_{N} w_{N,k}) -C(w_{N,k}, u_{N,k})$ may be written as a sum of trilinear operators with at least one input taking values in $\{\Pi_{>N}u_{N,k}, \Pi_{>N}w_{N,k}, \Pi_{>N/4}u_{N,k}, \Pi_{>N/4}w_{N,k}\}$. Having executed this decomposition, we show \eqref{convergence for C} by using the trilinear estimates for $C$ given in Lemma  \ref{Estimates for C, Q, S} and the structure of the argument used to handle the convergence for $F(u,w)$. Similarly, we may show that
\begin{align}
    \| \varphi_\tau \left[Q(u, w) - \Pi_{N}Q(\Pi_{N} u_{N,k}, \Pi_{N} w_{N,k})\right]\|_{Z_0^{s'}(\tau)} &\le  C(M)\tau^{\delta/2} \| w-w_{N,k}\|_{Z^{s'}_0(\tau)} + C(M)\tau^{\delta/2} \| u-u_{N,k}\|_{Y^{s'}_0(\tau)} \\
    &+ \tau^{\delta/2} C(M) N^{-(s-s')}
\end{align}
for $C(M)>0$ a quartic polynomial in $M$ using the estimates for $Q$ in Lemma \ref{Estimates for C, Q, S} and
\begin{align}
    \| \varphi_\tau \left[S(u, w) - \Pi_{N}S(\Pi_{N} u_{N,k}, \Pi_{N} w_{N,k})\right]\|_{Z_0^{s'}(\tau)} &\le  C(M)\tau^{\delta/2} \| w-w_{N,k}\|_{Z^{s'}_0(\tau)} + C(M)\tau^{\delta/2} \| u-u_{N,k}\|_{Y^{s'}_0(\tau)} \\
    &+ \tau^{\delta/2} C(M) N^{-(s-s')}
\end{align}
for $C(M)>0$ a polynomial of degree 6 in $M$ using the estimates for $\Sigma$ in Lemma \ref{Estimates for C, Q, S}. 

\noindent 
Combining the above estimates with \eqref{15} and applying the triangle inequality, we deduce that
\begin{align}
    \|w-w_{N,k}\|_{Z_0^{s'}(\tau)} &\le C \|u_0-u_{0,k}\|_{\F L^{s',p}} + C(M)\tau^{\delta/2} \| w-w_{N,k}\|_{Z^{s'}_0(\tau)} + C(M)\tau^{\delta/2} \| u-u_{N,k}\|_{Y^{s'}_0(\tau)} \\
    &+ \tau^{\delta/2} C(M) N^{-(s-s')} 
\end{align}
for a larger $C(M)>0$ of degree 6 in $M$. This concludes the proof of \textbf{Claim} B.
\end{proof}
 
We have therefore proved item \eqref{item 1 2.1}. In order to finish the proof of Proposition \ref{Long-time convergence proposition} we now show the proof of item \eqref{item 6},  
which we divide into the following claims, the proofs of which are inspired by Claim 1, Claim 2 and Claim 3  in Section 2 of \cite{KNPSV}. However, we adapt those proofs to the context of Fourier-Lebesgue spaces and their accompanied analysis via $X^{s,b}_{p,q}$ spaces instead of using energy estimates which are not available for us. 

Let now $u(t)\in C([-T,T], \FL^{s',p}), 1/2\le s' < s$ be the strong solution to \eqref{mKdV 2Cauchy} we constructed above.

\textbf{Claim 1}: We have $u \in C_w([-T,T], \mathcal{F}L^{s,p})$, that is: if $t_n \to \bar{t}$ then 
$u(t_n) \rightharpoonup u(\bar{t})$ weakly in $\mathcal{F}L^{s,p}$.
\begin{proof}
We know that 
\[
u \in C([-T,T], \mathcal{F}L^{s',p}),
\]
for every $s'$ with $\tfrac{1}{2} \leq s' < s$.  
Hence, whenever $t_n \to \bar{t}$, we have 
\[
u(t_n) \to u(\bar{t}) \quad \text{strongly in } \mathcal{F}L^{s',p}.
\]
In addition, due to \eqref{bound in item (5) Section 3}, we have the uniform bound
\[
\sup_n \|u(t_n)\|_{\mathcal{F}L^{s,p}} < \infty.
\]

Since $1<p<\infty$, we have that $\F L^{s,p}$ is weakly sequentially compact. 

It follows that there exists a subsequence $t_{n_k}$ and $\bar{u}\in \F L^{s,p}$ such that 
\[
u(t_{n_k}) \rightharpoonup \bar{u} \quad \text{weakly in } \mathcal{F}L^{s,p}.
\]
On the other hand, since $u(t_{n_k}) \to u(t)$ strongly in $\mathcal{F}L^{s',p}$, we deduce that $\bar{u} = u$ in the sense of distributions. We conclude that 
\[
u(t_n) \rightharpoonup u(t) \quad \text{in } \mathcal{F}L^{s,p}.
\]
This proves the claim.
\end{proof}
\textbf{Claim 2}: $u(t)$ is continuous at $t=0$ in the strong $\F L^{s,p}$ topology. 
\begin{proof}
    To establish the claim, let $t_n \to 0$ as $n \to \infty$. 
Our goal is to prove that
\begin{align} \label{Section 2.2 convergence of norms Claim 2}
\|u(t_n)\|_{\mathcal{F}L^{s,p}} \to\|u(0)\|_{\mathcal{F}L^{s,p}} 
\quad \text{as } n \to \infty.
\end{align}
The above convergence, combined with Claim 1 is enough to prove the strong continuity at $t=0$.

Now we proceed by proving \eqref{Section 2.2 convergence of norms Claim 2}.
Using Claim 1 together with the weak lower-semicontinuity of the 
$\mathcal{F}L^{s,p}$-norm, we deduce that
\begin{align} \label{Section 2.2 Claim 2 liminf}
\|u(0)\|_{\mathcal{F}L^{s,p}} \;\leq\; 
\liminf_{n\to\infty} \|u(t_n)\|_{\mathcal{F}L^{s,p}}.
\end{align}
We now proceed to show the bound
\begin{align} \label{Section 2.2 Claim 2 limsup}
    \limsup_{n\to\infty}\|u(t_n)\|_{\mathcal{F}L^{s,p}} \le \|u(0)\|_{\mathcal{F}L^{s,p}}
\end{align}
First, we have that, by \eqref{equation for u_n with cutoffs} and  \eqref{equation for w_n with cutoffs} of $u_N, w_N$ and the triangle inequality, 
\begin{align}
    \|u_N\|_{C([-t_n, t_n],\F L^{s,p})} &\le \| w_N\|_{C([-t_n, t_n],\F L^{s,p}} + \|\varphi_{t_n}F(u_N, w_N)\|_{C([-t_n, t_n],\F L^{s,p})} \\
    \|w_N\|_{C([-t_n, t_n],\F L^{s,p})} &\le \|u(0)\|_{\F L^{s,p}} + \|\varphi_{t_n} C(u_N, w_N)\|_{C([-t_n, t_n],\F L^{s,p})} +  \|\varphi_{t_n} Q(u_N, w_N)\|_{C([-t_n, t_n],\F L^{s,p})} \\
    &+  \|\varphi_{t_n}\Sigma(u_N, w_N)\|_{C([-t_n, t_n],\F L^{s,p})}.
\end{align}
Now, we use the embedding $ Z_0^s(t_n) \hookrightarrow C([-t_n, t_n], \F L^{s,p})$, Lemma \ref{Lemma for F} in combination with the definition of $F(u,w)$ to deduce that 
\begin{align}
    \|\varphi_{t_n}F(u_N, w_N)\|_{C([-t_n, t_n],\F L^{s,p})} \lesssim \|\varphi_{t_n}F(u_N, w_N)\|_{Y_0^s} \lesssim |t_n|^{\delta/2}\|F(u_N, w_N)\|_{Y_1^s}.
\end{align}
Doing the same but this time with the embedding $Z_0^s(t_n) \hookrightarrow C([-t_n, t_n], \F L^{s,p})$ and utilizing Lemma \ref{Lemma for F}, we conclude the following estimates:
\begin{align}
    \|\varphi_{t_n}C(u_N, w_N)\|_{C([-t_n, t_n],\F L^{s,p})} &\lesssim \|\varphi_{t_n}C(u_N, w_N)\|_{Z_0^s} \lesssim |t_n|^{\delta/2}\|C(u_N, w_N)\|_{Z_1^s} \\
    \|\varphi_{t_n}Q(u_N, w_N)\|_{C([-t_n, t_n],\F L^{s,p})} &\lesssim \|\varphi_{t_n}Q(u_N, w_N)\|_{Z_0^s} \lesssim |t_n|^{\delta/2}\|Q(u_N, w_N)\|_{Z_1^s} \\
    \|\varphi_{t_n}\Sigma (u_N, w_N)\|_{C([-t_n, t_n],\F L^{s,p})} &\lesssim \|\varphi_{t_n}\Sigma(u_N, w_N)\|_{Z_0^s} \lesssim |t_n|^{\delta/2}\|\Sigma(u_N, w_N)\|_{Z_1^s}.
\end{align}

From the above estimates, we deduce that 
\begin{align} \label{Section 2.2 inequality for Fourier-Lebesge norm of u_N}
     \|u_N\|_{C([-t_n, t_n],\F L^{s,p})} &\le \|u(0)\|_{\F L^{s,p}} + |t_n|^{\delta/2}\big(\| F(u_N, w_N)\|_{Y_1^s(t_n)} +\| C(u_N, w_N)\|_{Z_1^s(t_n)} + \| Q(u_N, w_N)\|_{Z_1^s(t_n)} \\
     &+ \| \Sigma(u_N, w_N)\|_{Z_1^s(t_n)}\big).
\end{align}
Take $n$ large enough such that $|t_n|\le c \langle M \rangle^{-\beta}$. Then, using the estimates given in Lemma \ref{Estimates for C, Q, S}, Lemma \ref{Estimates for G} and the fact that (see \eqref{bounds for truncated flow 2.1})
\begin{align}
 \| u_N\|_{Y_0^s(c \langle M \rangle^{-\beta})} +   \| w_N\|_{Z_0^s(c \langle M \rangle^{-\beta})} \le C'' M,
\end{align}
together with \eqref{Section 2.2 inequality for Fourier-Lebesge norm of u_N} we obtain
\begin{align}
    \|u_N(t_n) \|_{ \F L^{s,p}}\le \|u_N\|_{C([-t_n, t_n],\F L^{s,p})} &\le \|u(0)\|_{\F L^{s,p}} + C(M)|t_n|^{\delta/2},
\end{align}
where $C(M)$ is a constant depending only on $M$. 
Note that since $u_{N}(t)$ converges strongly to $u(t)$ in $\mathcal{F}L^{s',p}$ for every $1/2\le s' < s$, it follows that 
$u_{N}(t)$ converges weakly to $u(t)$ in $\mathcal{F}L^{s,p}$. Consequently, we obtain
\[
\|u(t_n)\|_{\mathcal{F}L^{s,p}} 
\leq \liminf_{N\to\infty} \|u_N(t_n)\|_{\mathcal{F}L^{s,p}} 
\leq \liminf_{N\to\infty}\Big(C|t_n|^{\tfrac{\delta}{2}} + \|u(0)\|_{\mathcal{F}L^{s,p}}\Big) 
= C|t_n|^{\tfrac{\delta}{2}} + \|u(0)\|_{\mathcal{F}L^{s,p}}.
\]
From this inequality it follows that
\[
\limsup_{n\to\infty} \|u(t_n)\|_{\mathcal{F}L^{s,p}} \leq \|u(0)\|_{\mathcal{F}L^{s,p}}.
\]
Therefore, by combining \eqref{Section 2.2 Claim 2 liminf} and \eqref{Section 2.2 Claim 2 limsup}, we have that
\begin{align}
    \limsup_{n\to\infty} \|u(t_n)\|_{\mathcal{F}L^{s,p}} \le \|u(0)\|_{\mathcal{F}L^{s,p}} \le \liminf_{n\to\infty} \|u(t_n)\|_{\mathcal{F}L^{s,p}},
\end{align}
which implies \eqref{Section 2.2 convergence of norms Claim 2}.
\end{proof}
\textbf{Claim 3.} $u(t)$ is continuous in the $\mathcal{F}L^{s,p}$ topology for every $t \in [-T,T]$.  
\begin{proof}
 To establish this, we modify the truncated initial value problem  by prescribing the data 
$u_{N}(\bar{t}) = u(\bar{t})$ at some fixed time $\bar{t}$. Denote by $v_{N}$ the solution of this new Cauchy problem.  
By applying at time $\bar{t}$ the same reasoning that was used for $u_{N}$ and $u$ at $t=0$, 
we deduce continuity of $u(t)$ in $\mathcal{F}L^{s,p}$ for all $\bar{t} \in [-T,T]$, which completes the proof.

\end{proof}

This concludes the proof of item \eqref{item 6}, and thus the proof of Proposition \ref{Long-time convergence proposition}. 

 \section{Weighted Gaussian Measures, Pairing and Multilinear Gaussian estimates} 
 \label{Section Gaussian Measures} 

In this section, we introduce and study the weighted and unweighted Gaussian measures and prove various bounds they obey, which in turn will be used in the globalization argument. We are inspired by the similar results proved in \cite{KNPSV}, though we must adapt them to the more general setting of $\F L^{s,p}$ for $p\neq 2$ ($p=2$ corresponds to $H^s$) and to the measure $\mu$ we study in this work. Since the support of $\mu$ includes unbounded functions, unlike that of $\mu_n, n\ge 2$ studied in \cite{KNPSV}, our density in the weighted Gaussian measure will no longer be bounded in the focusing case, and instead only be \textit{integrable}. That our density is integrable was shown in \cite{bourgain1994}, and we make use of the accompanying bounds shown in \cite{bourgain1994}. We note that we do not use Egoroff's theorem to prove Proposition \ref{Prop: Properties of weighted measures} below as was the case in \cite{KNPSV} but instead provide a proof relying on more elementary arguments, albeit it is consequently longer.

\subsection{Large Deviation Estimates}

 In this subsection we follow \cite{KNPSV}. Let us fix $n \in \mathbb{N}$ and $i_k \in\{ \pm 1\}$ for $k=1, \ldots, n$. For any vector $\left(k_1, \ldots, k_n\right) \in \mathbb{Z}^n$ and for every set $\mathcal{I} \subset\{1, \ldots, n\}$ we denote by $\left(k_1, \ldots, k_n\right)_{\mathcal{I}} \in \mathbb{Z}^{n-\# \mathcal{I}}$ the vector obtained removing from $\left(k_1, \ldots, k_n\right)$ the entries $k_j$ with $j \in \mathcal{I}$. Also we denote $\mathbb{N}_{\leq n}=\mathbb{N} \cap\{1, \ldots, n\}$. We give the following definition.

\begin{defn}\label{pairings}
 Given $\left(k_1, \ldots, k_n\right) \in \mathbb{Z}^n$ we say that we have: 
 \begin{enumerate}
\item  0 -pairing when $k_l \neq k_m$ for every $l, m \in\{1, \ldots, n\}$ with $i_l \neq i_m$;
\item 1-pairing, and we write $\left(k_l, k_m\right)$, when $k_l=k_m$ for some $l, m \in\{1, \ldots, n\}$ with $i_l \neq i_m$ and $k_m \neq k_h$ for every $h, m \in\{1, \ldots, n\} \backslash\{l, m\}$ such that $i_h \neq i_m$.
\item  $r$-pairing for $r \in\left\{1, \ldots,\left[\frac{n}{2}\right]\right\}$ provided that there exist $X=\left(l_1, \ldots, l_r\right) \in \mathbb{N}_{\leq n}^r, Y=\left(j_1, \ldots, j_r\right) \in$ $\mathbb{N}_{\leq n}^r$, with $\left\{l_1, \ldots, l_r\right\} \cap\left\{j_1, \ldots, j_r\right\}=\emptyset$ and $\#\left\{l_1, \ldots, l_r\right\}=\#\left\{j_1, \ldots, j_r\right\}=r$, such that $k_{l_m}=k_{j_m}$ for every $m=1, \ldots, r$ and $i_{j_m} \neq i_{l_m}$, and moreover $k_m \neq k_h$ for every $h, m \in$ $\{1, \ldots, n\} \backslash\left\{l_1, \ldots, l_r, j_1, \ldots, j_r\right\}$ such that $i_h \neq i_m$. In this case we shall write that $\left(k_1, \ldots, k_n\right)$ are $(X, Y)$-pairing.
\end{enumerate}
\end{defn}
Next we consider linear combinations of multilinear Gaussians $g_{\vec{k}}(\omega)=\prod_{j=1}^n g_{k_j}^{i_j}(\omega)$ where $g_l^1=g_l$ and $g_l^{-1}=\bar{g}_l$. Let $a_{k_1, \ldots, k_n} \in \mathbb{C}$ then we have the following bound, which is a special case of the more general large deviation results proved in \cite{DengNahmodYue2024Gibbs2D}, \cite{DengNahmodYue2022RandomTensors}. 

\begin{prop}\label {large deviations}
 We have the following bounds for a suitable constant $C>0$ :
$$
\left\|\sum_{\vec{k} \in \mathbb{Z}^n} a_{\vec{k}} g_{\vec{k}}(\omega)\right\|_{L_\omega^2}^2 \leq C \sum_{\substack{r \in\left\{1, \ldots,\left[\frac{n}{2}\right]\right\} \\ X=\left(l_1, \ldots, l_r\right) \in \mathbb{N}^r \leq n, Y=\left(j_1, \ldots, j_r\right) \in \mathbb{N}^r \leq n \\\left\{l_1, \ldots, l_r\right\} \cap\left\{j_1, \ldots, j_r\right\}=\emptyset=\\ \#\left\{l_1, \ldots, l_r\right\}=\#\left\{j_1, \ldots, j_r\right\}=r}} \sum_{\vec{h} \in \mathbb{Z}^{n-2 r}}\left(\sum_{\substack{\vec{k} \in \mathbb{Z}^n \\ \vec{k} \text{ is } (X,Y)- \text{pairing} \\ \vec{k}_{(l_1, \ldots, l_r, j_1, \ldots, j_r)}=\vec{h}}}\left|a_{\vec{k}}\right|\right)^2 .
$$
\end{prop}
\subsection{Gaussian Measures and basic properties}
The following lemma is inspired by the results proved in Section 3.1 of \cite{KNPSV}, which we adapted to the context of $\F L^{s,p}$ spaces.
\begin{lem} We denote by $\mu$ as the Gaussian measure associated with the law of the random variable \eqref{random vector}. We note the following properties regarding $\mu$:

\begin{enumerate}
    \item $\mu(\F L^{s,p}(\T))=1, \quad \forall s<1-1/p$

    \item For each $s<1-1/p, p<\infty$, there exists $k, K>0$ depending on $s,p$ such that $$\mu\{u\in\F L^{s,p}(\T): \norm{u}_{\F L^{s,p}(\T)} >M \}\le K e^{-k M^2}$$.
\end{enumerate}
\end{lem}
\begin{proof}
For (1), we note that since $(1-s)p>1$
\begin{align}
    \mathbf{E} \left(\norm{\sum_{j\in \Z} \frac{g_j(\omega)}{\langle j \rangle } e^{ijx}}_{\F L^{s,p}}^p \right) &= \sum_{j\in \Z} \mathbf{E}(|g_j(\omega)|^p)\langle j\rangle ^{-(1-s)p } \\
    &\le C_p \sum_{j\in \Z} \langle j\rangle ^{-(1-s)p }<\infty
\end{align}

where $C_p:= \mathbf{E}(|g_1(\omega)|^p)$.

    For (2), we have, by the Minkowski inequality for $q\ge 4$, and thus $q\ge p$ since $2<p<4$, 
    \begin{align}
        \left(\mathbf{E} \left(\norm{\sum_{j\in \Z} \frac{g_j(\omega)}{\langle j \rangle } e^{ijx}}_{\F L^{s,p}}^q \right)\right)^{1/q} & = \left(\mathbf{E} \left(\norm{\frac{g_j(\omega)}{\langle j \rangle^{1-s} } }_{l^{p}}^q \right)\right)^{1/q} \\
      &= \norm{\norm{\frac{g_j(\omega)}{\langle j \rangle^{1-s} } }_{\ell^p_j}}_{L^q(d\omega)} \\
      &\le \norm{\norm{\frac{g_j(\omega)}{\langle j \rangle^{1-s} } }_{L^q(d\omega)}}_{\ell^p_j}\\
        &= \norm{\left(\mathbf{E}\left|\frac{g_j(\omega)}{\langle j \rangle^{1-s} }\right|^q\right)^{1/q}}_{\ell^p_j} \\
        &\lesssim q^{1/2} \norm{\frac{1}{\langle j \rangle^{1-s} }}_{\ell^{p}_j} \\
        &\lesssim q^{1/2} \left(\sum_{j\in \Z}\frac{1}{\langle j \rangle^{p(1-s)} }\right)^{1/p} \\
        &\lesssim  q^{1/2} C(s,p),
    \end{align}
    where we used the condition $s<1-1/p$ and the fact that for any centered, normalized Gaussian random variable $g_j$, 
    \begin{align}
        \left(\mathbf{E}(|g_j(\omega)|^q)\right)^{1/q} \lesssim q^{1/2}.
    \end{align}

    We conclude the Gaussian bound (2) by the following Proposition (see \cite{TV-ASENS-13} for the proof).
\begin{prop}
 \label{Chebyshev}
       Let $F:(\Omega, \mathcal{A}, \mu) \rightarrow \mathbf{C}$ be measurable and $C>0$ be such that
$$
\|F\|_{L^q(d\mu)} \leq C q^{1/2}, \quad  \forall q \in[4, \infty).
$$

Then
$$
\mu\left\{\omega \in \Omega:|F(\omega)|>\lambda\right\} \leq e^{-\frac{1}{2e}\left(\frac{\lambda}{C}\right)^{2}}, \quad \forall \lambda\ge C(4e)^{1/2}.
$$
\end{prop}
\end{proof}    
\subsection{Weighted Gaussian Measures}

    Next, we introduce the family of measures
\begin{align} \label{Truncated rho definiiton}
    d\rho_{R,N,\pm} :=  Z_{R, N, \pm }^{-1} F_{R,N, \pm}(u) d\mu,
\end{align}
where 
\begin{align}
    F_{R,N, \pm }(u) := \chi_{R^2}(E_1(\Pi_N u)) \exp\left(\mp \int_\T |\Pi_N u|^4\,dx\right)
\end{align}
and
\begin{align}
     Z_{R, N, \pm } := \int \chi_{R}(E_1(\Pi_N u)) \exp\left(\mp \int_\T |\Pi_N u|^4\,dx\right) d\mu,
\end{align}
where we recall that, due to \eqref{def of E_1}, 
\begin{align}
    E_1(u)= \int_\T |u|^2\,dx.
\end{align}
and $\Pi_N$ was introduced in \eqref{Definition of Pi_N}.
We also define
\begin{align} \label{untruncated rho definition}
d\rho_{R,\pm} : =Z_{R, \pm }^{-1} F_{R, \pm} d\mu
\end{align}
where
\begin{align}
    F_{R, \pm}:= \chi_{R}(E_1(u)) \exp\left(\mp \int_\T |u|^4\,dx\right)
\end{align}
and
\begin{align}
    Z_{R, \pm} := \int \chi_{R}(E_1(u)) \exp\left(\mp \int_\T |u|^4\,dx\right) d\mu.
\end{align} 

We first show that the measures $\rho_{R,\pm}$ are well-defined, which is equivalent to showing that
\begin{align} \label{well-definedness of rho}
    \chi_{R}(E_1(u)) \exp\left( \pm \int |u|^4\right) \in L^1(d\mu)
\end{align}
and furthermore in the defocusing case, that
\begin{align} \label{Z non-zero in defocusing case}
    Z_{R, -} \neq 0.
\end{align}
The result \eqref{well-definedness of rho} for the focusing case follows from the following lemma we will prove:
\begin{lem}\label{Non-bourgain integrability}
    For any $R>0$ the function 
    \begin{align}
        \exp\left(\left\|  \sum_n \frac{g_n(\omega) e^{inx}}{\langle n\rangle}\right\|_{L^4_x}^4\right) \chi_{R}\left(\sum_{n} \frac{|g_n(\omega)|^2}{\langle n\rangle^2} \right)
    \end{align}
    lies in $L^q(d\mu)$ for any $1\le q <\infty$.
\end{lem}
    The proof we present of Lemma \ref{Non-bourgain integrability} uses a Gagliardo-Nirenberg-type inequality (see Proposition \ref{GN}), the hypercontractivity of Gaussian random variables (see \eqref{hypercontractivity}), and a Chebychev-type inequality (see Proposition \ref{Chebyshev}). It also extends to densities of the form
    \begin{align}
        \exp\left(\left\|  \sum_n \frac{g_n(\omega) e^{inx}}{\langle n\rangle}\right\|_{L^p_x}^p\right) \chi_{R}\left(\sum_{n} \frac{|g_n(\omega)|^2}{\langle n\rangle^2} \right)
    \end{align}
    for any $2<p<6$. The first two ingredients replace the argument introduced by Bourgain in \cite{bourgain1994}.
\begin{prop} \label{GN proposition}
The following estimate holds,
    \begin{align} \label{GN}
        \|u\|_{L^4(\T)} \le \|\langle \nabla \rangle^\sigma u\|_{L^{\tilde{q}}(\T)}^{\theta} \|u\|_{L^2(\T)}^{1-\theta},
    \end{align}
    where $\sigma\ge 0, 1\le \tilde{q}\le \infty$ are such that 
    \begin{align} \label{eq for theta}
        \frac{1}{4}= \theta\left(\frac{1}{\tilde{q}}-\sigma\right)+\frac{1-\theta}{2}
    \end{align} and $\langle \nabla \rangle^s$ is the Fourier multiplier given on the Fourier side by $n \mapsto (1+n^2)^{s/2}, n\in \Z$.
    \end{prop}
    \begin{proof}
    We first define
    \begin{align}
        H^{s,q}(\T) := \{u\in \mathcal{S}'(\T): \| \langle \nabla \rangle^s u \|_{L^q(\T)}<\infty \}
    \end{align}
    The proof of \eqref{GN} easily follows by the Sobolev embedding (see Corollary 1.2 of \cite{MR3119672}) $H^{\sigma_\theta, q_\theta}(\T) \hookrightarrow L^4(\T)$ for 
        \begin{align}
            \frac{1}{4}= \frac{1}{\tilde{q}_\theta}-\sigma_\theta
        \end{align}
and $\tilde{q}_\theta< 4$ and the interpolation estimate (See Theorem 2 in Section 3.6.1 of \cite{ST}, recalling that $H^{\sigma, q}(\T)= F^s_{q,2}(\T)$, where  $F^s_{q,2}(\T)$ is Triebel-Lizorkin space)
\begin{align}
    \|u\|_{H^{\sigma_\theta, \tilde{q}_\theta}(\T)} &\le \|u\|_{H^{\sigma,\tilde{q}}(\T)}^\theta \|u\|_{H^{0,2}(\T)}^{1-\theta} \\
    &=\|u\|_{H^{\sigma ,\tilde{q}}(\T)}^\theta \|u\|_{L^2(\T)}^{1-\theta}
\end{align}
for
\begin{align}
    \frac{1}{\tilde{q}_\theta} =\frac{\theta}{\tilde{q}} + \frac{1-\theta}{2}, \quad \sigma_\theta&= \sigma \theta.
\end{align}
Therefore, we have that that \eqref{GN} holds provided \eqref{eq for theta}.
\end{proof}
Letting $\theta=2/5, \sigma=1/4, \tilde{q}=8, \sigma_\theta=(2/5)(1/4)=1/10, \tilde{q}_\theta=20/7<4$, we have that 
\begin{align} \label{GN specific}
        \|u\|_{L^4(\T)} \le \|\langle \nabla \rangle^{1/4} u\|_{L^{8}(\T)}^{2/5} \|u\|_{L^2(\T)}^{3/5}.
    \end{align}

    We will also use the following hyper-contractivity estimate for Gaussian variables (see \cite{TT-2010})
\begin{align} \label{hypercontractivity}
    \left\| \sum_{j\in \Z} c_j g_j(\omega) \right\|_{L^r(d\mu)} \le C \sqrt{r} \left( \sum_{n\in \Z} |c_n|^2 \right)^{1/2}
\end{align}
valid for any $r\ge 2$ and Proposition \ref{Chebyshev}.
\begin{proof}[Proof of Lemma \ref{Non-bourgain integrability}]
In the following, we will take
\begin{align}
    u(x,\omega)= \sum_{j\in \Z} \frac{g_j(\omega) e^{ijx}}{\langle j\rangle}.
\end{align}
On the set $\{\omega: \|u\|_{L^2}\le R^{1/2}\}$, the inequality \eqref{GN specific} gives that
    \begin{align}
        \|\langle \nabla \rangle^{1/4} u\|_{L^{8}(\T)} \ge \frac{\|u\|_{L^4(\T)}^{5/2}}{R^{3/4}}.
    \end{align}
    Therefore, we have that
\begin{align} \label{Lemma 4.5 eq 1}
    \mu \{ \omega: \|u\|_{L^4(\T)} > \lambda, \|u\|_{L^2(\T)}\le R^{1/2} \} \le \mu \left\{ \omega: \|\langle \nabla \rangle^{1/4} u\|_{L^{8}(\T)} > \frac{\lambda^{5/2}}{R^{3/4}} \right\}.
\end{align}
By the argument in Proposition 4.4 of \cite{TV-ASENS-13}, we have that 
\begin{align}
    \mu\{\omega: \|\langle \nabla \rangle^\sigma u\|_{L^{\tilde{q}}(\T)}> M\} \le C e^{-c M^2}.
\end{align}
Taking $M=\frac{\lambda^{5/2}}{R^{3/4}}$ and using \eqref{Lemma 4.5 eq 1}, we have that 
\begin{align} \label{Lemma 4.5 eq 2}
    \mu \{ \omega: \|u\|_{L^4(\T)} > \lambda, \|u\|_{L^2(\T)}\le \sqrt{R} \} \le C \exp \left(-c \lambda^5 R^{-3/2}\right).
\end{align}
 We now compute, recalling that $\chi_{R}(\cdot) \le 1_{\{ |\cdot|\le 2R\}}$, 
\begin{align}
    &\left\| \exp\left(\left\|  \sum_n \frac{g_n(\omega) e^{inx}}{\langle n\rangle}\right\|_{L^4_x}^4\right) \chi_{R}\left(\sum_{n} \frac{|g_n(\omega)|^2}{\langle n\rangle^2} \right) \right\|_{L^q(d\mu)}^q \\
    &\le \int_{\{\|u\|_{L^2(\T)} \le \sqrt{2R}\}} \exp\left(q\left\|  \sum_n \frac{g_n(\omega) e^{inx}}{\langle n\rangle}\right\|_{L^4_x}^4\right)d\mu \\
    &= \int_{\{\|u\|_{L^2(\T)} \le \sqrt{2R}\}} \left[\exp\left(q\left\|  \sum_n \frac{g_n(\omega) e^{inx}}{\langle n\rangle}\right\|_{L^4_x}^4\right)-1\right]d\mu + \mu\{ \|u\|_{L^2(\T)}\le \sqrt{2R} \}  \\
    &=\int_0^\infty 4qt^3 e^{qt^4} \mu \{ \omega: \|u\|_{L^4(\T)} > t, \|u\|_{L^2(\T)}\le \sqrt{2R} \} \,dt +\mu\{ \|u\|_{L^2(\T)}\le \sqrt{2R} \} \\
    &\le \int_0^\infty C4qt^3 \exp\left(qt^4-ct^5 (2R)^{-3/2}\right) \,dt +\mu\{ \|u\|_{L^2(\T)}\le \sqrt{2R} \}\label{final Lq bound}
\end{align}
This concludes the proof as the above integral is finite for every $q\ge 1, R>0$.
\end{proof}
A nearly identical proof to that of Lemma \ref{Non-bourgain integrability} yields 
\begin{lem} \label{Non-bourgain density is uniformly L^q}
      For any $R>0$, the function 
    \begin{align}
        \exp\left(\left\|  \sum_{|n|\le N} \frac{g_n(\omega) e^{inx}}{\langle n\rangle}\right\|_{L^4_x}^4\right) \chi_{R}\left(\sum_{|n|\le N} \frac{|g_n(\omega)|^2}{\langle n\rangle^2} \right)
    \end{align}
    lies in $L^q(d\mu)$ for any $1\le q <\infty$ with $L^q(d\mu)$ norm bounds independent of $N$. 
\end{lem}
For the defocusing case, \eqref{well-definedness of rho} follows  due to the property
$$\exp\left(-\int_\T |u|^4\,dx\right)\le 1$$
for any  $u$. The condition \eqref{Z non-zero in defocusing case} will follow from \eqref{Lemma 4.5 eq 2}.
First, we note that \eqref{Lemma 4.5 eq 2} implies that 
\begin{align}
    \mu\left\{\left\|  \sum \frac{g_n(\omega) e^{inx}}{\langle n\rangle}\right\|_{L^4_x}=+\infty, \sum \frac{|g_n(\omega)|^2}{1+n^2} \le R \right\}=0.
\end{align}
Next, we have that
\begin{align} \label{Integral inequality in proving that Z_R,- >0}
    \int \chi_{R}(E_1(u)) \exp\left(-\int_{\T} |u|^4\right) d\mu &\ge \int_{\{E_1(u)\le R\}} \exp\left(-\int_{\T} |u|^4\right) d\mu.
\end{align}
We  would like to conclude that the right hand side of \eqref{Integral inequality in proving that Z_R,- >0} is positive. We know that on  $\{\|u\|_{L^2(\T)}\le \sqrt{R}\}$, the quantity $\| u\|_{L^4(\T)}$ is $\mu$-a.s. finite. What we must show then is that that $\mu\{E_1(u)\le R\}>0$. This follows immediately from the facts that $\mu\{E_1(u)\le R\}=1-\mu\{E_1(u)> R\}$ and $\mu(L^2(\T))=1$, so that for large enough $R>0$, we have $\mu\{E_1(u)> R\}<1/2$, implying that $\mu\{E_1(u)\le R\}>1/2$. Therefore, \eqref{Z non-zero in defocusing case} holds.

We collect various properties of the measure $\rho_{R,\pm}$ and $\rho_{R, N,\pm}$ in the following Proposition:

\begin{prop} \label{Prop: Properties of weighted measures}

Fix $R>0$ and $N\in \mathbb{N}$. Then, the following bounds hold,
\begin{align} \label{Prop 4.4 eq 1}
    \sup_{N\in \mathbb{N}}\| F_{N, R, \pm} \|_{L^q(d\mu)}<\infty, \quad \| F_{ R, \pm} \|_{L^q(d\mu)}<\infty,
    \end{align}
    for any $1\le q<\infty$. In the defocusing case, $q=\infty$ may be taken. We have the following convergence, 
    \begin{align}\label{truncated measure convergence}
    \sup_{\substack{A\in\mathcal{B}(\F L^{s,p}) \\s<1-1/p}} \left|\rho_{R,N,\pm}(A)- \rho_{R, \pm }(A)\right| \overset{N\to\infty}{\longrightarrow} 0.
\end{align}
Lastly, we have: for each $(s,p)$ such that $s<1-1/p$, there exists $k, K>0$ depending on $s,p$ such that
\begin{align} \label{sub-gaussian bound on rho_N}
    \rho_{N,R,\pm } \{u\in\F L^{s,p}(\T): \|u\|_{\F L^{s,p}(\T)} >M \} \le K e^{-kM^2}
\end{align}
\end{prop}

\begin{proof}

Note that \eqref{Prop 4.4 eq 1} follows immediately from Lemma \ref{Non-bourgain integrability} and Lemma \ref{Non-bourgain density is uniformly L^q}.
We now show \eqref{truncated measure convergence}.
As in \cite{KNPSV}, we will make of use Egoroff's theorem.

Since $\mu(H^{1/4})=1$, we have that $\rho_{R, N, \pm}(H^{1/4})=\rho_{R, \pm}(H^{1/4})=1$ and hence
\begin{align}
    \sup_{\substack{A\in\mathcal{B}(\F L^{s,p}) \\s<1-1/p}}\left|\rho_{R,N,\pm}(A)- \rho_{R, \pm }(A)\right| &= \sup_{\substack{A\in\mathcal{B}(\F L^{s,p}) \\s<1-1/p}} \left|\rho_{R,N,\pm}(A\cap H^{1/4})- \rho_{R, \pm }(A\cap H^{1/4})\right|.
\end{align}
Note that 
\begin{align}
    \Pi_N u \overset{N\to\infty}{\longrightarrow} u \quad \forall u\in H^{1/4}(\T).
\end{align} 
Additionally, we have that by the Sobolev embedding theorem
\begin{align}
    \| u-\Pi_N u\|_{L^4(\T)}\lesssim \| u-\Pi_N u\|_{H^{1/4}(\T)}\to 0
\end{align}
for $u\in A \cap H^{1/4}$. Therefore, $\Pi_N u\to u$ in $L^4(\T)$ as $N\to\infty$. We also have that, by similar arguments, $\Pi_N u\to u$ in $L^2(\T)$ for $u\in A \cap H^{1/4}$. Therefore, 
\begin{align}
    F_{R,\pm}(\Pi_N u) \to F_{R, N, \pm}(u)
\end{align}
for each $u\in A \cap H^{1/4}$. Next, we apply Egoroff's theorem to $H^{1/4}$  which states that for every $\e>0$ there exists a set $\mathcal{R}_\varepsilon\subseteq H^{1/4}$ such that $\mu(\mathcal{R}_\e)<\e$ and the convergence of $F_{R, N, \pm}(\Pi_N) \to F_{R, N, \pm}(u)$ is uniform on $ H^{1/4}\setminus \mathcal{R}_\e$. Therefore, we have that, by the triangle inequality and the property $\mu(C)= \mu(C\cap \mathcal{R}_\e) + \mu(C \cap \mathcal{R}_\e^c)$,
\begin{align}
    \sup_{\substack{A\in\mathcal{B}(\FL^{s,p})}} \left|\rho_{R,N,\pm}(A\cap H^{1/4})- \rho_{R, \pm }(A\cap H^{1/4})\right|  &\le \sup_{\substack{A\in\mathcal{B}(\FL^{s,p})}} \left|\rho_{R,N,\pm}(A\cap H^{1/4}\cap  \mathcal{R}_\e)- \rho_{R, \pm }(A \cap H^{1/4}\cap \mathcal{R}_\e)\right| \label{Egoroff 1}\\
    &+\sup_{\substack{A\in\mathcal{B}(\FL^{s,p})}} \left|\rho_{R,N,\pm}(A\cap H^{1/4}\cap  \mathcal{R}_\e^c)- \rho_{R, \pm }(A \cap H^{1/4}\cap \mathcal{R}_\e^c)\right|.
\end{align}
We estimate the first term  on the right hand side of \eqref{Egoroff 1} as follows, by Hölder,
\begin{align}
    \left|\rho_{R,N,\pm}(A\cap H^{1/4}\cap  \mathcal{R}_\e)- \rho_{R, \pm }(A \cap H^{1/4}\cap \mathcal{R}_\e)\right| &\le \sup_{N\in \mathbb{N}} \rho_{R,N,\pm}(R_\e) + \rho_{R,\pm}(R_\e) \\
    &\le \mu(R_\e)^{1/2} \sup_{N\in \mathbb{N}} \|F_{R, \pm}(\Pi_N u)\|_{L^2(d\mu)} + \mu(R_\e)^{1/2} \|F_{R, \pm}( u)\|_{L^2(d\mu)} \\
    &\le C \e^{1/2},
\end{align}
where we used \eqref{Prop 4.4 eq 1},
while the second term may be made smaller than $\e$ if $N\ge N_\e$ is taken large enough (it may also be zero if $A\cap \mathcal{R}_\e^c $ is empty) due to the uniform convergence $F_{R, \pm}(\Pi_N) \to F_{R,  \pm}(u)$.  Hence, 
\begin{align}
    \sup_{\substack{A\in\mathcal{B}(\FL^{s,p})}} \left|\rho_{R,N,\pm}(A \cap H^{1/4})- \rho_{R, \pm }(A \cap H^{1/4})\right| \le C\sqrt{\e} + \e 
\end{align}
for all  $N\ge N_\e$ and we have shown \eqref{truncated measure convergence}.

From the above bound on $F$, we also have that by the Cauchy-Schwarz inequality
\begin{align}
    &\rho_{ N,R,\pm } \{u\in\F L^{s,p}(\T): \|u\|_{\F L^{s,p}(\T)} >M \}\\
    &=\int_{ \{u\in\F L^{s,p}(\T): \|u\|_{\F L^{s,p}(\T)} >M \}} Z_{R,N, \pm}^{-1}\, \chi_{R}(E_1(\Pi_N u)) \exp\left(\mp \int_{\T} |u|^4\right) d\mu \\
     &\le \mu_{1} \{u\in\F L^{s,p}(\T): \|\Pi_Nu\|_{\F L^{s,p}(\T)} >M \}^{1/2} \left\| Z_{R,N, \pm}^{-1}\, \chi_{R}(E_1(\Pi_Nu)) \exp\left(\mp \int_{\T} |\Pi_Nu|^4\right) \right\|_{L^2(d\mu)} \\
    &\le  K e^{-k M^2}
\end{align}
with constants $K, k$ independent of $N$. 
\end{proof}\par 
The following proposition is useful in the argument given in Section \ref{Section: Proof of Main Theorem}. Its proof is inspired by the similar proposition given in \cite{KNPSV}. 
\begin{prop} \label{Prop for complements of Borel Sets}
    Let $A_j\in \F L^{s,p}$ be Borel subsets with $s<1-1/p$ such that $\rho_{j, \pm}(\F L^{s,p} \setminus A_j)=0$ for every $j\in \mathbb{N}$. Then necessarily, $\mu(\F L^{s,p} \setminus \bigcup_{j\in \mathbb{N}} A_j)=0$.
\end{prop}
\begin{proof}
    Note that since $\mu(L^2)=1$, we may replace $\F L^{s,p}$ by $L^2 \cap \F L^{s,p}$. We introduce the sets
    \begin{align}
        \Omega_j = \{u\in L^2 \cap \F L^{s,p} : \chi(E_1(u)/j)>0 \}.
    \end{align}
    Then, we claim that for each $j\ge 1$. 
    \begin{align}
        \mu(L^2 \cap \F L^{s,p} \setminus A_j \cap \Omega_j)=0
    \end{align}
   To see why this holds, note that by our hypothesis,
   \begin{align}
       0= \int_{L^2 \cap \F L^{s,p} \setminus A_j} F_{j, \pm} d\mu = \int_{(L^2 \cap \F L^{s,p} \setminus A_j) \cap \Omega_j} F_{j, \pm} d\mu,
   \end{align}
   and the claim follows since $F_{j, \pm}>0$ on $\Omega_j$. 
    
    Next, we will show that
    \begin{align} \label{limit of measure Section 3}
        \mu_{1,\pm}(\Omega_j)\to 1 \text{ as } j\to\infty.
    \end{align}
To see why, we note that for $u\in L^2 \cap \F L^{s,p}$ then $E_1(u)<\infty$ by definition, so there exists $l$ such that if $j\ge l$, then $u\in \Omega_j$. This implies that $1_{\Omega_j}(u)\to 1$ as $j\to\infty$ for each $u\in L^2 \cap \F L^{s,p}$ and \eqref{limit of measure Section 3} follows by the Dominated Convergence  
  Theorem since $\mu_{1, \pm}$ are probability measures. Next, we deduce that for every $l\in \mathbb{N}$, 
\begin{align}
    \mu_{1, \pm} \Biggl( \Bigl( \mathcal{F}L^{s,p} \cap L^2 
   \setminus \bigcup_{j \in \mathbb{N}} A_j \Bigr) 
   \cap \Omega_l \Biggr) 
&= \mu_{1, \pm} \Biggl( \bigcap_{j \in \mathbb{N}} 
   \Bigl( \mathcal{F}L^{s,p} \cap L^2 \setminus A_j \Bigr) 
   \cap \Omega_l \Biggr) \\
&\leq \mu_{1, \pm} \Biggl( 
   \Bigl( \mathcal{F}L^{s,p} \cap L^2 \setminus A_j \Bigr) 
   \cap \Omega_l\Biggr) \\
&= 0
\end{align}
 We complete the proof by taking the limit $l\to\infty$ in the above inequality and using \eqref{limit of measure Section 3}. 
\end{proof}

\section{Almost-Invariance of \texorpdfstring{$\rho_{N, R, \pm}$}{truncated measure} along the flow \texorpdfstring{$\Phi_N(t)$}{PhiN(t)}} \label{Section almost invariance}

 In this section, we prove that the weighted measures $\rho_{R,N,\pm}$ are almost-invariant under the flow $\Phi_N(t)$. We adapt the methodology of \cite{KNPSV}, which is in part itself inspired by \cite{TV-ASENS-13, TV-IMRN-13, TV-JMPA-15}, to our different weighted measure. We utilize the bounds proved in Section 3.

\begin{prop}\label{almost invariance prop}
    Let $R,T>0$ be given. Then, 
    \begin{align}
         \sup_{\substack{t\in[0,T] \\ A\in \mathcal{B}(\F L^{s,p}), s<1-1/p}} \big|\rho_{R,N, \pm} (\Phi_N(t)A) - \rho_{R,N, \pm} (A)\big| \overset{N\to\infty}{\longrightarrow} 0
    \end{align}
\end{prop}
We recall the definition of the  energies 
\begin{align}
    E_1(u)=\int |u|^2, \quad E_3(u)=\int |\partial u|^2 \pm |u|^4
\end{align}
and the measure
\begin{align}
    d\rho_{R,N, \pm} =\chi_{R}(E_1(\Pi_N)u)) \exp\left(\mp\int |\Pi_N u|^4\right) d\mu,
\end{align}
where $\chi_R\in C_c^\infty([0,\infty))$ is supported in $B(0,2R)$ and equals $1$ on $B(0,R)$ and $\mu$ is the law of the random vector
\begin{align} \label{random vector}
    \frac{1}{\sqrt{2\pi}} \sum_{j\in \Z}
 g_j(\omega) \frac{e^{ijx}}{\langle j\rangle},
\end{align}
where $g_j$ is a sequence of normalized, centered i.i.d complex Gaussian random  variables. The measure $\mu$ is supported on $\F L^{s,p}(\T)$ for all $s<1-1/p$.  \par Let us introduce the functional
\begin{align}
    E_{3, N}^{*}:= \left(\frac{d}{dt} E_3(\Pi_N \Phi_N(t)u)\right)_{t=0}.
\end{align}

In order to prove Proposition \ref{almost invariance prop}, we shall need the following result
\begin{prop}\label{Probabilistic estimate}
For every $R>0$ 
\begin{align}
    \norm{\chi_{R}(E_1(\Pi_N u))^{1/2} \times E_{3, N}^{*}}_{L^2(d\mu)} \overset{N\to\infty}{\longrightarrow} 0 
\end{align}
\end{prop}
We first show how Proposition \ref{Probabilistic estimate} implies Proposition \ref{almost invariance prop}.\newline
\textit{Prop. \ref{Probabilistic estimate} $\implies$ Prop. \ref{almost invariance prop}}:
We shall prove for any $R>0$ the following estimate
\begin{align}\label{d/dt rho}
    \sup_{\substack{t\in \R \\ A\in \mathcal{B}(\F L^{s,p}), s<1-1/p}}\bigg| \frac{d}{dt} \rho_{R,N, \pm} (\Phi_N(t)A) \bigg| \overset{N\to\infty}{\longrightarrow} 0.
\end{align}
Hence, once we establish \eqref{d/dt rho}, then Proposition \ref{almost invariance prop} will follow by fundamental theorem of calculus.\par
We will now use the splitting
\begin{align}
    d\mu= \gamma_N \exp\left(-\norm{\Pi_N u}_{L^2}^2 -\norm{\Pi_N \partial u}_{L^2}^2\right) d\mathcal{L}_N \otimes d\mu_{1,N}^{\perp},
\end{align}
where $\mathcal{L}=\prod_{j=-N}^Nu_j $ is the classical Lebesgue measure on $\C^{2N+1}$, $\gamma_N$ is a normalization constant such that $\gamma_N \exp\left(-\norm{\Pi_N u}_{L^2}^2 -\norm{\Pi_N \partial u}_{L^2}^2\right)$ is a probability measure on $\C^{2N+1}$ and $\mu^\perp$ is the pushforward of the high frequency projection of the random vector \eqref{random vector}, 
\begin{align}
    \frac{1}{\sqrt{2\pi}} \sum_{j\in \Z_{>N}}
 g_j(\omega) \frac{e^{ijx}}{\langle j\rangle}. 
\end{align}
Let us now compute, with $R_3(u):=\int |u|^4$, 
\begin{align}
    \frac{1}{\gamma_N} \frac{d}{dt}\left( \rho_{N,R, \pm} (\Phi_N(t)A\right)_{t=0} &= \frac{d}{dt}\left( \int_{\Phi_N(t)A} \chi_{R}(E_1(u)) \exp\left(-E_3(u)\right) d\mathcal{L}_N \otimes d\mu_{1,N}^{\perp}\right)_{t=0} \\
    &=\frac{d}{dt}\left( \int_{A} \chi_{R}(E_1(\Phi_N(t)u)) \exp\left(-E_3(\Phi_N(t)u)\right) d\mathcal{L}_N \otimes d\mu_{1,N}^{\perp}\right)_{t=0} \\
    &= - \int_{A} \chi_{R}(E_1(\Pi_N u)) E_{3,N}^* \exp\left(-E_3(\Pi_N u)\right) d\mathcal{L}_N \otimes d\mu_{1,N}^{\perp} \\
    &= - \gamma_N^{-1} \int_{A} \chi_{R}(E_1(\Pi_N u)) E_{3,N}^* \exp\left(\mp R_3(\Pi_N u)\right) d\mu
\end{align}
where we have used the change of variable $A\mapsto \Phi_N(t)A$, Liouville's theorem (as the flow $\Pi_N \Phi_N(t)$ is Hamiltonian) guaranteeing the invariance of the measure under this transformation, and the fact that $E_1(\Phi_N(t)u)$ is conserved. We further estimate, by Cauchy-Schwarz 
\begin{align}
    &\sup_{A\in \mathcal{B}(\F L^{s,p}), s<1-1/p} \abs{\int_{A} \chi_{R}(E_1(\Pi_N u)) E_{3,N}^* \exp\left(\mp R _3(\Pi_N u)\right) d\mu} \\
    &\le  \left\| \chi_{R}(E_1(\Pi_N u))^{1/2} \exp\left(\mp \int |\Pi_N u|^4\right)\right\|_{L^2(d\mu)}\|\chi_{R}(E_1(\Pi_N u))^{1/2} E_{3,N}^*\|_{L^2(\mu)} \\
    &\le  C(R) \norm{ \chi_{R}(E_1(\Pi_N u))^{1/2} E_{3,N}^*}_{L^2(\mu)}\overset{N\to\infty}{\longrightarrow} 0
\end{align}
 
by Lemma \ref{Non-bourgain density is uniformly L^q} for $q=2$ and Proposition \ref{Probabilistic estimate}.  Therefore, we have shown that
\begin{align}
    \sup_{ A\in \mathcal{B}(\FL^{s,p}), s<1-1/p}\bigg| \frac{d}{dt}\left( \rho_{R,N, \pm} (\Phi_N(t)A)\right)_{t=0} \bigg| \overset{N\to\infty}{\longrightarrow} 0
\end{align}
Now, using the fact that $\Phi_N(t)$ is a flow, 
\begin{align}
    \frac{d}{dt}\left( \rho_{R,N, \pm} (\Phi_N(t)A)\right)_{t=t_0} = \frac{d}{dt}\left(\int_{\Phi_N(t)(\Phi_N(t_0)A)} d\rho_{N,R, \pm}\right)_{t=0}
\end{align}
and Proposition \ref{almost invariance prop} follows provided we switch to the Borel set $\Phi_N(t_0)A$, which is admissible as our estimate is uniform in $A$.

\subsection{Proof of Prop. 1.2}

In this subsection, we follow the methodology used in Section 4 of \cite{KNPSV}.
From the expression for $E_3(u)$, we have that
\begin{align}
    E_{3,N}^* = 2\operatorname{Re}\int \left(\partial \partial_t (\Pi_N u_N) \pm   2\partial_t (\Pi_N u_N )\overline{ \Pi_N u_N} |\Pi_N u_N|^2\right)_{t=0},
\end{align}
where we recall $u_N(t)=\Phi_N(t)u$.
Note that $\Pi_N u_N(t)$ solves the equation 
\begin{align}\label{equation for u_N}
    \partial_t  \Pi_N u_N + \partial_x^3  \Pi_N u_N &= \pm 6 \Pi_N \mathcal{N}( \Pi_N u_N, \overline{ \Pi_N u_N}, \Pi_N u_N) \\
    &=    \pm 6  \mathcal{N}( \Pi_N  u_N, \overline{ \Pi_N u_N}, u_N) \mp 6 \Pi_{>N} \mathcal{N}( \Pi_N u_N, \overline{\Pi_N  u_N}, \Pi_N u_N) \\
    &= \mp 6  \mathcal{N}( \Pi_N u_N, \overline{\Pi_N u_N}_, \Pi_N u_N) \mp  6 \Pi_{>N} (|\Pi_N u_N|^2 \partial_x \Pi_N u_N),
\end{align}
where $\mathcal{N}$ is defined in \eqref{Nonlinearity of mKdV2 physical space} and we used that
\begin{align}
    \Pi_{>N} \mathcal{N}( \Pi_N u_N, \overline{\Pi_N  u_N}, \Pi_N u_N) &=\Pi_{>N}  \left(|\Pi_N u_N|^2 \partial_x \Pi_N u-M( \Pi_N u_N )\partial_x \Pi_N  u_N -iP(\Pi_N u_N)\Pi_N u_N\right) \\
    &= \Pi_{>N} \left(|\Pi_N u_N|^2 \partial_x \Pi_N u_N\right)
\end{align}
since by orthogonality and the fact that $M(\Pi_Nu_N), P( \Pi_N u_N )$ are scalars,
\begin{align}
    M( \Pi_N u_N )\Pi_{>N} \partial_x \Pi_N u_N & =0 \\
     P( \Pi_N u_N )\Pi_{>N} \Pi_N u_N &=0.
\end{align}

Recall that for smooth solutions (in particular for $u_N(t)$), the flows of the mKdV2 \eqref{mKdV 2Cauchy} and the mKdV \eqref{mKdV} are related by two gauge transformations. Since these gauge transformations leave $E_3(u)$ unchanged and $E_3(u)$ is conserved along the flow of the mKdV \eqref{mKdV}, we have that $E_3(u)$ is conserved along the flow of the mKdV2 \eqref{mKdV 2Cauchy}. Therefore we can replace $\partial_t \Pi_N u_N$ by the additional forcing term on the right-hand side of \eqref{equation for u_N}, obtaining 
\begin{align}
    E_{3,N}^* &=  12\operatorname{Re } \int\left[ \mp \partial  \left(\Pi_{>N} (|\Pi_N u_N|^2 \partial_x \Pi_N u_N)\right)) \partial (\overline{\Pi_N u_N}) - 2 \Pi_{>N} (|\Pi_N  u_N|^2 \partial_x \Pi_N  u_N) \overline{\Pi_N u_N}|\Pi_N u_N|^2\right]_{t=0}  \\
    &= - 24 \re \int \left[\Pi_{>N} (|\Pi_N u_N|^2 \partial_x \Pi_N u_N) \overline{\Pi_N u_N } |\Pi_N u_N|^2\right]_{t=0}, 
\end{align}
using an orthogonality argument in the last step:
\begin{align}
    \int \Pi_{>N} f\, \Pi_N g=0
\end{align}
for appropriate $f,g$. Recalling the definition of $u_N$, we have that
\begin{align} \label{E_(3,N)}
    E_{3,N}^*= - 24 \re \int\Pi_{>N} (|\Pi_N u |^2 \partial \Pi_N u  ) \Pi_N \overline{u} |\Pi_N u|^2
\end{align}
We now replace the function $u$ in \eqref{E_(3,N)} by the random vector \eqref{random vector}. We shall show that
\begin{align}\label{Probabilistic estimate 2.0}
    \norm{\im \sum_{\substack {j_1, j_2, j_3, j_4, j_4, j_6\in \Z_{\le N} \\
    j_1-j_2+j_3-j_4+j_5-j_6=0 \\
    |j_1-j_2+j_3|>N
    }} \frac{j_3}{\langle j_1\rangle \langle j_2\rangle 
    \langle j_3\rangle
    \langle j_4\rangle
    \langle j_5\rangle
    \langle j_6\rangle}g_{\vec{j}}(\omega)}_{L^2_\omega} \overset{N\to\infty}{\longrightarrow} 0,
\end{align}
where 
\begin{align}
    \Z_{\le N}= \Z\cap [-N,N], \quad g_{\vec{j}}(\omega) = g_{j_1}(\omega) \overline{g}_{j_2}(\omega) g_{j_3}(\omega)\overline{g}_{j_4}(\omega) g_{j_5}(\omega)\overline{g}_{j_6}(\omega).
\end{align}
Note that we pass from taking the real part in \eqref{E_(3,N)} to the imaginary part in \eqref{Probabilistic estimate 2.0} since taking the Fourier transform of a derivative brings down a factor of $-i$. The imaginary part is not crucial to the argument, but gives an additional way to show that one of the contributions to \eqref{Probabilistic estimate 2.0} vanishes (see Lemma \ref{tilde}) so we retain it.  We introduce some notation to streamline the presentation:
\begin{align}
    \I_N=\{\j\in \Z_{\le N}: \L(\j)=0, |\P(\j)|> N\}
\end{align}
where
\begin{align}
    \mathcal{L}(\j) = j_1-j_2+j_3-j_4+j_5-j_6, \quad \P(\j) = j_1-j_2+j_3.
\end{align}
Next we split $\I_N$ into a number of subsets that will allow us to effectively apply Proposition 1.1. We use Definition 3.6 for $n=6$, $i_1=i_3=i_5=1$ and $i_2=i_4=i_6$ and we partition the set of indices according to whether we have a 1, 2, or 3-pairing.
\begin{rmk}
    If we have $\j\in \I_N$ with a 2 or 3-pairing then its contribution to \eqref{Probabilistic estimate 2.0} vanishes. Indeed, if we have a 2-pairing, then by the condition $\L(\j)=0$, we must have a 3-pairing, in which case the associated Gaussian $g_{\j}(\omega)$ is real-valued, and thus its imaginary part is zero. We therefore restrict our discussion to the 0 or 1-pairings.
\end{rmk}
Now, we consider the following sets:
\begin{align}
    \I^0_N = \{\j\in \I_n: \j \text{ 0-pairing }\}
\end{align}
and for any $k\in\{1,3,5\}$ and $l\in\{2,4,6\}$,
\begin{align}
    \I_N^{j_k, j_l}=\{\j\in \I_N: (j_k,j_l) \text{ 1-pairing }\}.
\end{align}
We also define the splitting
\begin{align}
    \I_N^{j_k, j_l} = \tilde{\I}_N^{j_k, j_l} \cup \hat{\I}_N^{j_k, j_l},
\end{align}
where
\begin{align}
\tilde{\I}_N^{j_k, j_l} = \{\j\in \I_N: (k,l) \text{ is a 1-pairing and } \#\{j_1, j_2, j_3, j_4, j_5, j_6\}=5\}, \quad \hat{\I}_N^{j_k, j_l} = \I_N^{j_k, j_l} \setminus \tilde{\I}_N^{j_k, j_l}.
\end{align}
We introduce the definition
\begin{align}
    a(\j) = \frac{j_3}{\langle j_1\rangle \langle j_2\rangle 
    \langle j_3\rangle
    \langle j_4\rangle
    \langle j_5\rangle
    \langle j_6\rangle}.
\end{align} 
\par 
We treat the $0$-pairing case in Lemma \ref{0-pairing}. We split the cases of 1-pairing in two sub-cases: either the pairing happens between indices that lie in the same triplet $\{1,2,3\}$ or $\{4,5,6\}$, or they lie in different triplets. All the pairings of the first type are treated in Lemma \ref{same triplet lemma} and the pairings of the second type are treated in Lemma \ref{different triplet lemma} except the pairings $\I^{j_3, j_4}_N$ and $\I^{j_3, j_6}_N$. One can check that these pairings give divergent contributions if we apply the the proofs of Lemma \ref{different triplet lemma} to them. For this reason, we split
\begin{align}
    \I^{j_3, j_4}_N = \tilde{\I}^{j_3, j_4}_N \cup {\hat\I}^{j_3, j_4}_N, \quad \I^{j_3, j_4}_N = \tilde{\I}^{j_3, j_6}_N \cup {\hat\I}^{j_3, j_6}_N\quad 
\end{align}
and treat the contributions of $\tilde{\I}^{j_3, j_4}_N, \tilde{\I}^{j_3, j_6}_N$ in Lemma \ref{tilde} and the contributions of $\tilde{\I}^{j_3, j_4}_N, \tilde{\I}^{j_3, j_6}_N$ in Lemma \ref{hat}.\par
We first treat the case of $0$-pairing.
\begin{lem}\label{0-pairing}
The following holds:
\begin{align}
    \im \sum_{\j\in \Z_{\le N}} a(\j) \mathbf{1}_{\tilde{\I}^0_N} g_\j(\omega)\overset{N\to\infty}{\longrightarrow} 0.
\end{align}
\end{lem}
\begin{proof}
From the constraint $|\P(\j)|>N$, we have that $\max\{|j_4|, |j_5|, |j_6|\} \ge N/3$, and from the identity $\L(\j)=0$. We can reduce the sum to only five indices $j_1, j_2, j_4, j_5, j_6$. Using the bound, 
\begin{align}
    a(\j)\le \frac{1}{\langle j_1\rangle \langle j_2\rangle 
    \langle j_4\rangle
    \langle  j_5\rangle
    \langle j_6\rangle},
\end{align}
we conclude from
\begin{align}
    \sum_{\substack{ j_1, j_{2}, j_{4},j_{5, j_6} \in \Z_{\le N}    \\
    \max\{|j_{4}|, |j_{5}|, |j_{6}|\}\ge N/3}} \frac{1}{\langle j_1\rangle^2 \langle j_{2}\rangle^2 \langle j_{4}\rangle^2 \langle j_{5}\rangle^2\langle j_{6}\rangle^2 } = O(N^{-1}).
\end{align}
\end{proof}
We deal first with the cases where $(k,l)$ belong to the same triplet
\begin{lem}\label{same triplet lemma}
We have the following for every $\mathcal{Z}_N = \I_N^{j_5, j_6},\I_N^{j_4, j_5},\I_N^{j_2, j_3}, \I_N^{j_1, j_2}$.
\begin{align}
        \im \sum_{\j\in \Z_{\le N}} a(\j) \mathbf{1}_{\mathcal{Z}_N} g_\j(\omega)=0
    \end{align}    
\end{lem}
\begin{proof}
Let us consider the case $\mathcal{Z}_N=\I_N^{j_5, j_6}$ as the other cases follow by essentially the same argument. In this case, notice we have from $|\P(\j)|>N$ that $|j_4|>N$. However, this set is disjoint from $|j_4|\le N$, and thus the sum collapses and this case does not contribute. 
\end{proof}
We now deal with the cases where $(k,l)$ belong to different triplets.
\begin{lem} \label{different triplet lemma}
    We have the following limit for every $\mathcal{Z}_N = \I_N^{j_1, j_4},\I_N^{j_1, j_6},\I_N^{j_2, j_5}$
    \begin{align}\label{same triplet eq}
        \im \sum_{\j\in \Z_{\le N}} a(\j) \mathbf{1}_{\mathcal{Z}_N} g_\j(\omega)\overset{N\to\infty}{\longrightarrow} 0.
    \end{align}
\end{lem}
\begin{proof}
   First, notice that if we show \eqref{same triplet eq} for $\mathcal{Z}_N=\I_N^{j_1, j_4}$, then \eqref{same triplet eq} automatically holds for $\mathcal{Z}_N=\I_N^{j_1, j_6}$ due to the $j_4 \leftrightarrow j_6$ permutation symmetry. We will therefore prove the statement only for the former. We must have, with $j=j_4=j_1$, $\max\{|j|, |j_{5}|, |j_{6}|\}\ge N/3$ and the condition $\L(\j)=0$ means that we need only sum over indices $j, j_2, j_5, j_6$:
    \begin{align}
        \sum_{\substack{  j_{2}, j_{3},j_{5}, j_6 \in \Z_{\le N}    \\
    \max\{|j|, |j_{5}|, |j_{6}|\}\ge N/3}} \left(\sum_{j\in\Z_{\le N}} \frac{1}{\langle  j\rangle^2 \langle j_{2}\rangle \langle j_{5}\rangle \langle j_{6}\rangle }\right)^2 =O(N^{-1}).
    \end{align}
    Lastly, we deal with the case $\mathcal{Z}_N=\I_N^{j_2, j_5}$, in which case, with $j=j_2=j_5$, we have $\max\{|j_4|, |j|, |j_{6}|\}\ge N/3$ and we only need to sum over indices $j_1, j, j_4, j_6$:
    \begin{align}
        \sum_{\substack{  j_1,j, j_{4}, j_6 \in \Z_{\le N}    \\
    \max\{|j_4|, |j|, |j_{6}|\}\ge N/3}} \left(\sum_{j\in\Z_{\le N}} \frac{1}{\langle  j\rangle^2 \langle j_{1}\rangle \langle j_{4}\rangle \langle j_{6}\rangle }\right)^2 =O(N^{-1}).
    \end{align}
\end{proof}
We now treat the cases $\tilde{\I}_N^{j_3, j_4}, \tilde{\I}_N^{j_3, j_6}$.
 We give two ways of treating these terms. The first is based on the estimate
\begin{align}
    \sum_{j\in \Z_{\le N}} \frac{1}{\langle j \rangle \langle N-j \rangle} \lesssim \frac{\log(N)}{N},
\end{align}
which may be proved by separately considering the regions $\{0\le |j|<N/2\}$, $\{N/2<|j|\le N\}$, and showing  the contributions vanish \textit{after} taking $N\to \infty$. The second method  is based on a cancellation property which states that the contributions of these terms  vanish for any fixed $N$. 
\begin{lem}
    The following limit holds:
    \begin{align}
        \im \sum_{\j\in \Z_{\le N}} a(\j) \mathbf{1}_{\tilde{\I}_N^{j_3, j_4} \cup \tilde{\I}_N^{j_3, j_6}}g_\j(\omega) \overset{N\to\infty}{\longrightarrow} 0
    \end{align}
\end{lem}
\begin{proof}
    We notice that $\tilde{\I}_N^{j_3, j_4}$ and $\tilde{\I}_N^{j_3, j_6}$ are disjoint so we treat them separately. We focus on dealing with the former, as the latter can be treated similarly. We note that
    \begin{align}
        |j_1-j_2| &= |j_1-j_2+j_3-j_3|\ge |j_1-j_2+j_3|-|j_3|\ge N- |j_3|.
    \end{align}
    Therefore, we must have
    \begin{align}
        \max\{|j_1|, |j_2|\} \ge \frac{N-|j_3|}{2}.
    \end{align}
    Assume that $|j_1|\ge |j_2|$, the other case being treated similarly. We then estimate
    \begin{align}
        \sum_{\substack{  j_1,j_2,, j_{4}, j_5\in \Z_{\le N}    \\
    |j_1-j_2+j_3|>N \\ |j_1|\ge j_2 }} \left(\sum_{j_3\in\Z_{\le N}} \frac{1}{\langle  j_1\rangle \langle j_{2}\rangle \langle j_{3}\rangle \langle j_{4}\rangle \langle j_{5}\rangle }\right)^2 
    &\lesssim  \sum_{\substack{  j_1,j_2,, j_{4}, j_5\in \Z_{\le N}    \\
    |j_1-j_2+j_3|>N}} \frac{1}{ \langle j_{2}^2\rangle \langle j_{4}\rangle^2 \langle j_{5}\rangle^2 } \left(\sum_{j_3\in\Z_{\le N}}\frac{1}{\langle j_3 \rangle \langle N-j_3 \rangle}  \right)^2 \\
    &\lesssim \sum_{\substack{  j_1,j_2,, j_{4}, j_5\in \Z_{\le N}}} \frac{1}{ \langle j_{2}^2\rangle \langle j_{4}\rangle^2 \langle j_{5}\rangle^2 }   \frac{\log(N)^2}{N^2}.
    \end{align}
    Summing over $j_1$ produces a factor of $N$, giving a final bound of
    \begin{align}
        \frac{\log(N)^2}{N} \overset{N\to\infty}{\longrightarrow} 0
    \end{align}
    
\end{proof}
Now, we present the second method.
\begin{lem} \label{tilde}
    We have the following identity:
    \begin{align}
        \im \sum_{\j\in \Z_{\le N}} a(\j) \mathbf{1}_{\tilde{\I}_N^{j_3, j_4} \cup \tilde{\I}_N^{j_3, j_6}}g_\j(\omega)=0
    \end{align}
\end{lem}
\begin{proof}
    We notice that $\tilde{\I}_N^{j_3, j_4}$ and $\tilde{\I}_N^{j_3, j_6}$ are disjoint so we treat them separately. We focus on dealing with the former, as the latter can be treated similarly. We first notice that both vectors 
    \begin{align}\label{permuted vectors}
        (j_1, j_2, j, j, j_5, j_6) \quad (j_6, j_5, j, j, j_2, j_1)
    \end{align}
belong to $\tilde{\I}_N^{j_3, j_4}$. Moreover, both the above vectors are different since by assumption $\#\{j, j_1, j_2, j_5, j_6\}=5$, and any permutation different from the identity (as done in \eqref{permuted vectors}) produces a genuinely new vector. The first vector in \eqref{permuted vectors} contributes the following to the overall estimate \eqref{Probabilistic estimate 2.0}:
\begin{align}
   \im \frac{j}{\langle j_1\rangle \langle j_2\rangle 
    \langle j\rangle^2
    \langle  j_5\rangle
    \langle j_6\rangle} g_{j_1}\overline{g}_{j_2} |g_{j}|^2 g_{j_5}\overline{g}_{j_6},
\end{align}
while the second vector in \eqref{permuted vectors} contributes
\begin{align}
    \im \frac{j}{\langle j_1\rangle \langle j_2\rangle 
    \langle j\rangle^2
    \langle  j_5\rangle
    \langle j_6\rangle} g_{j_6}\overline{g}_{j_5} |g_{j}|^2 g_{j_2}\overline{g}_{j_1}.
\end{align}
Hence, adding up their contribution gives
\begin{align}
    \frac{j}{\langle j_1\rangle \langle j_2\rangle 
    \langle j\rangle^2
    \langle  j_5\rangle
    \langle j_6\rangle} \im\left(g_{j_1}\overline{g}_{j_2} |g_{j}|^2 g_{j_5}\overline{g}_{j_6} + g_{j_6}\overline{g}_{j_5} |g_{j}|^2 g_{j_2}\overline{g}_{j_1}\right)=0.
\end{align}
\end{proof}
Next, we treat the cases $\hat{\I}_N^{j_3, j_4}, \hat{\I}_N^{j_3, j_6}$.
\begin{lem}\label{hat}
    We have the following result:
    \begin{align}
        \im \sum_{\j\in \Z_{\le N}} a(\j) \mathbf{1}_{\hat{\I}_N^{j_3, j_4} \cup \hat{\I}_N^{j_3, j_6}}g_\j(\omega) \overset{N\to\infty}{\longrightarrow} 0.
    \end{align}
\end{lem}
\begin{proof}
    We will again focus on the case $\hat{\I}_N^{j_3, j_4}$, the other one being similar. We consider vectors $\j$ such that $j_3=j_4=j_l=j$, where $l\in\{1,2,5,6\}$. Next, notice that by the property $\P(\j)|>N$, we have that
\begin{align}
        \max\{|j|, |j_{l_1}|, |j_{l_2}|, |j_{l_3}|\}\ge N/3.
    \end{align}
where $\{l_1, l_2, l_3\}= \{1, 2, 5, 6\}\setminus\{l\}$ We also have the bound
\begin{align}
    a(\j) \le \frac{1}{\langle j\rangle^2 \langle j_{l_1}\rangle \langle j_{l_2}\rangle \langle j_{l_3}\rangle }
\end{align}
Thus, the desired estimate is reduced to
\begin{align}
\sum_{\substack{  j_{l_1}, j_{l_2},j_{l_3} \in \Z_{\le N}    \\
    \max\{|j|, |j_{l_1}|, |j_{l_2}|, |j_{l_3}|\}\ge N/3}} \left(\sum_{j\in\Z_{\le N}} \frac{1}{\langle  j\rangle^2 \langle j_{l_1}\rangle \langle j_{l_2}\rangle \langle j_{l_3}\rangle }\right)^2 = O(N^{-1}). 
\end{align}
\end{proof}
This concludes the proof of Proposition \ref{Probabilistic estimate}.

\section{Proof of Main Theorem}\label{Section: Proof of Main Theorem}

The proof that we present here is inspired by the proof of Theorem 1.5 in \cite{KNPSV}. However, there are changes in the function spaces used and for that purpose we use the material obtained in Section \ref{Section Gaussian Measures} and Section \ref{Section almost invariance}.

\subsection{Getting ready for Proof of Theorem \ref{Main Result}}

We note that the flow $\Phi_N(t)$ is defined in Section \ref{Section:FiniteDimApproximations}. For every $i,j\in \mathbb{N}$ and $D>0$, define the set
\begin{align}
    B^{i,j}_{s,p,D}:=\{u\in \F L^{s,p}: \| u \|_{\F L^{s,p}}\le D\sqrt{i+j} \}.
\end{align}
as well as
\begin{align}
    \Sigma^{i,j}_{N, D, s, p} := \bigcap_{h\in \mathbb{Z} \cap \{t: |t|\le 2^j c^{-1} \langle D\sqrt{i+j}\rangle ^{\beta}\}} \Phi_N(hc \langle D\sqrt{i+j}\rangle ^{-\beta})  B^{i,j}_{D, s,p}.
\end{align}
\begin{rmk}
    The elements in $\Sigma^{i,j}_{N,s,D}$ are exactly the initial data in 
$B^{i,j}_{s,D}$ which are mapped back to $B^{i,j}_{s,D}$ by $\Phi_N(t)$ along times 
equidistributed at distance $c \langle D\sqrt{i+j} \rangle^{-\beta}$ in the 
time interval $[-2^j,2^j]$
\end{rmk}
We state the following lemma:

\begin{lem} \label{Lemma 5.2 equivalent}
     There exists $D_0>0$ such that for $D>D_0$, we have the bound
    \begin{align}
        \sup_{|\tau|\le 2^j } \| \Phi_N(\tau)u\|_{\F L^{s,p}} \le  D\sqrt{i+j+1} \quad \forall u\in \Sigma^{i,j}_{s,p,D}, \quad \forall N\in \mathbb{N}.
    \end{align}

\end{lem}
\begin{proof}
    Splitting the interval $[-2^j, 2^j]$ into intervals of width $c^{} \langle D\sqrt{i+j}\rangle ^{-\beta}$, it is enough to show that
    \begin{align} 
        \sup_{\tau\in[hc \langle D\sqrt{i+j}\rangle ^{-\beta}, (h+1)c \langle D\sqrt{i+j}\rangle ^{-\beta}]}\| \Phi_N(\tau)u\|_{\F L^{s,p}} \le  D\sqrt{i+j+1}
    \end{align}
    uniformly in
    \begin{align}
        h\in \mathbb{Z}\cap [-2^j  c^{-1} \langle D\sqrt{i+j}\rangle ^{\beta}, 2^j  c^{-1} \langle D\sqrt{i+j}\rangle ^{\beta}], \quad u \in \Sigma^{i,j}_{D, s,p}.
    \end{align}
    Now we proceed to prove the above claim. We choose $D_0$ large enough such that
    \begin{align}
        D\sqrt{i+j} +(D\sqrt{i+j})^{-1}\le D\sqrt{i+j+1}, \quad \forall i,j\in\mathbb{N}
    \end{align}
    holds for all $D>D_0$. We briefly sketch why this is possible. The above inequality is equivalent to
    \begin{align}
        D(\sqrt{i+j+1}- \sqrt{i+j})\ge  \frac{1}{\sqrt{i+j}} D^{-1}.
    \end{align}
    Using the difference of square identity, we may rewrite the left hand side as
\begin{align}
        \frac{D}{\sqrt{i+j+1} + \sqrt{i+j}} \ge   \frac{1}{\sqrt{i+j}} D^{-1}.
    \end{align}
Some rearranging shows the above is equivalent to
\begin{align}
    D^2 \ge 1+ \sqrt{1+\frac{1}{i+j}}, \quad  \forall i, j \in \mathbb{N}.
\end{align}
    Since the right hand side is decreasing in both $i,j$, it is enough to check it holds for $(i,j)=(1,1)$:
    \begin{align}
        D^2 \ge 1+ \sqrt{3/2},
    \end{align}
    and therefore we may choose
    \begin{align}
        D_0=\sqrt{1+\sqrt{3/2}}.
    \end{align}
    
    Next, we claim that
    \begin{align} \label{D+1/D bound}
\sup_{\tau\in[hc \langle D\sqrt{i+j}\rangle ^{-\beta}, (h+1)c \langle D\sqrt{i+j}\rangle ^{-\beta}]}\| \Phi_N(\tau)u\|_{\F L^{s,p}} \le  D\sqrt{i+j} +(D\sqrt{i+j})^{-1} \quad \forall u \in \Sigma^{i,j}_{D, s,p}.   
    \end{align}
    In order to prove \eqref{D+1/D bound}, we start by using the group property of $\Phi_N$ as follows,
    \begin{align}
        \sup_{\tau\in[hc \langle D\sqrt{i+j}\rangle ^{-\beta}, (h+1)c \langle D\sqrt{i+j}\rangle ^{-\beta}]}\| \Phi_N(\tau)u\|_{\F L^{s,p}} &= \sup_{0\le t\le     c \langle D\sqrt{i+j}\rangle ^{-\beta}}\| \Phi_N(t) \Phi_N(hc \langle D\sqrt{i+j}\rangle ^{-\beta}) u\|_{\F L^{s,p}} \nonumber\\
        &\le \sup_{|t| \le     c \langle D\sqrt{i+j}\rangle ^{-\beta}}\| \Phi_N(t) \Phi_N(hc \langle D\sqrt{i+j}\rangle ^{-\beta}) u\|_{\F L^{s,p}} \label{eq: group property}
    \end{align}
        Now, we note that since $u\in \Sigma^{i,j}_{D,s,p}$, for any $h$, there exists $u_0\in B^{i,j}_{D,s,p}$ such that $u= \Phi_N(-hc \langle D\sqrt{i+j}\rangle ^{-\beta})u_0$ and therefore, by using the group property of $\Phi_N(t)$ again, 
        \begin{align}
            \Phi_N(hc \langle D\sqrt{i+j}\rangle ^{-\beta}) u = \Phi_N(hc \langle D\sqrt{i+j}\rangle ^{-\beta})\Phi_N(-hc \langle D\sqrt{i+j}\rangle ^{-\beta})u_0=u_0\in B^{i,j}_{D,s,p}.
        \end{align}
        Now we are in a position to use Proposition \ref{S+1/S 2nd half} as follows 
        \begin{equation} \label{eq: group product}
            \sup_{|t| \le     c \langle D\sqrt{i+j}\rangle ^{-\beta}}\| \Phi_N(t) \Phi_N(hc \langle D\sqrt{i+j}\rangle ^{-\beta}) u\|_{\F L^{s,p}} \le  D\sqrt{i+j}+  (D\sqrt{i+j})^{-1} \le D\sqrt{i+j+1}.
        \end{equation}
        Finally, we combine \eqref{eq: group property} and \eqref{eq: group product} to obtain 
        \begin{align}
            \sup_{\tau\in[hc \langle D\sqrt{i+j}\rangle ^{-\beta}, (h+1)c \langle D\sqrt{i+j}\rangle ^{-\beta}]}\| \Phi_N(\tau)u\|_{\F L^{s,p}} \le D\sqrt{i+j+1},
        \end{align}
        uniformly in $h\in \mathbb{Z}\cap [-2^j  c^{-1} \langle D\sqrt{i+j}\rangle ^{\beta}, 2^j  c^{-1} \langle D\sqrt{i+j}\rangle ^{\beta}]$, 
        which finishes the proof. 
        \end{proof}
\par
We start with the following proposition:
\begin{prop} \label{Prop 5.3}
    For every $i\in \mathbb{N}, R>0$ and $(s,p)$ as in Proposition \ref{S+1/S 2nd half}, there exists $\Sigma^{s,p}_{i,R} \subset \F L^{s,p}$ and a constant $D>0$ such that 
    \begin{align}\label{2^(-i) bound}
        \Sigma^{s,p}_{i,R} \text{ is closed in } \F L^{s,p}; \\
        \rho_{R, \pm}(\F L^{s,p} \setminus  \Sigma^{s,p}_{i,R}) \le 2^{-i},
        \end{align}
        where $\rho_{R,\pm}$ are as defined in \eqref{untruncated rho definition}.  
     
      Additionally, for any $u \in \Sigma^{s,p}_{i,R}$,  there exists $j_k, L_k \to\infty$ and  $u_k\in \F L^{s,p}$ such that 
        \begin{align} \label{Section 5 property of Sigma-s-p-i-R}
        \| u_k-u\|_{\F L^{s,p}} \to 0, \quad \sup_{|\tau|< 2^{j'}} \| \Phi_{L_k}(\tau) u_k\|_{\F L^{s,p}} \le D \sqrt{i+j+1}, \quad \forall j'\in \{1, \cdots, j_k\}.
    \end{align}
\end{prop} 
        \begin{proof}
            We argue that we can choose $D$ and a sequence $N_j\to\infty$ such that
            \begin{align} \label{5}
                \rho_{R, N_j, \pm}(\F L^{s,p} \setminus \Sigma^{i,j'}_{N_j, D, s,p})\le 2^{-i-j'}, \quad \forall j'\in\{1, \cdots, j\}.
            \end{align}
            In fact, we note that by the definition of $\Sigma_{N_j, D, s,p}^{i, j}$ we have that, for fixed $D>0$, 
            \begin{align}
                \rho_{R, N_j, \pm} (\F L^{s,p} \setminus \Sigma_{N_j, D, s,p}^{i, j}) &= \rho_{R, N_j, \pm} \left(\bigcup_{h=-\left\lfloor \frac{2^{j} \langle D \sqrt{i+j} \rangle^{\beta}}{c} \right\rfloor}^{\left\lfloor \frac{2^{j} \langle D \sqrt{i+j} \rangle^{\beta}}{c} \right\rfloor}
\FL^{s,p} \setminus \Phi_{N}\!\left(h c \langle D \sqrt{i+j} \rangle^{-\beta}\right)\!\left(B^{i,j}_{D, s,p}\right) \right) \\
&\le \sum_{h=-\left\lfloor \frac{2^{j} \langle D \sqrt{i+j} \rangle^{\beta}}{c} \right\rfloor}^{\left\lfloor \frac{2^{j} \langle D \sqrt{i+j} \rangle^{\beta}}{c} \right\rfloor} \rho_{R, N_j, \pm} \left(\Phi_{N}\!\left(h c \langle D \sqrt{i+j} \rangle^{-\beta}\right) (\FL^{s,p} \setminus B^{i,j}_{D, s,p})\right).
            \end{align}
            Next, by Proposition \ref{almost invariance prop} for any $\varepsilon>0, j\in \mathbb{N}$, we have the existence of $N_{\varepsilon, j, R}$ such that for $N\ge N_{\varepsilon, j, R}$, 
            \begin{align}
                \sup_{|t|\le 2^j} \left|\rho_{R,N,\pm}\left(\Phi_{N}\!\left(t\right) (\FL^{s,p} \setminus B^{i,j}_{D, s,p})\right) -\rho_{R,N,\pm}\left( \FL^{s,p} \setminus B^{i,j}_{D, s,p}\right) \right| \le \varepsilon.
            \end{align}
As a result, we may estimate
\begin{align}
 & \sum_{h=-\left\lfloor \frac{2^{j} \langle D \sqrt{i+j} \rangle^{\beta}}{c} \right\rfloor}^{\left\lfloor \frac{2^{j} \langle D \sqrt{i+j} \rangle^{\beta}}{c} \right\rfloor}  \rho_{R,N,\pm} \left(\Phi_{N}\!\left(h c \langle D \sqrt{i+j} \rangle^{-\beta}\right) (\FL^{s,p} \setminus B^{i,j}_{D,s,p})\right) \\
  &\le 2\left\lfloor \frac{2^{j} \langle D \sqrt{i+j} \rangle^{\beta}}{c} \right\rfloor \left( \varepsilon +\rho_{R,N,\pm} \left( \FL^{s,p} \setminus B^{i,j}_{D, s,p}\right)\right) \\
  &\le  2\left\lfloor \frac{2^{j} \langle D \sqrt{i+j} \rangle^{\beta}}{c} \right\rfloor ( \varepsilon + K e^{-kD^2(i+j+1)}),
\end{align}
where we used \eqref{sub-gaussian bound on rho_N} to deduce the last line. Note that we also have the following bound, for any $j'\in\{0, \cdots, j\}$ and $N\ge N_{\varepsilon, j, R}$, by the same argument:
\begin{align}
    \rho_{R,N,\pm} (\F L^{s,p} \setminus \Sigma_{N_j, D, s,p}^{i, j'}) \le 2\left\lfloor \frac{2^{j'} \langle D \sqrt{i+j'} \rangle^{\beta}}{c} \right\rfloor ( \varepsilon + K e^{-kD^2(i+j'+1)}).
\end{align}
Next, we choose $\varepsilon= 2^{-D(i+j)}$ with $D>0$ large enough such that
\begin{align}
      \frac{2^{j'+1} \langle D \sqrt{i+j'} \rangle^{\beta}}{c} ( 2^{-D(i+j)} + K e^{-kD^2(i+j'+1)}) \le 2^{-i-j'}
\end{align}
for $j'\in\{1, \cdots j\}$. This then implies that
\begin{align}
     \rho_{R, N, \pm}(\F L^{s,p} \setminus \Sigma^{i,j'}_{N_j, D,s,p})\le 2^{-i-j'}, \quad \forall j'\in\{1, \cdots, j\}
\end{align}
for $N_j\ge N_{2^{-D(i+j)}, j, R}$ as desired. Hence we have proved \eqref{5}. 

We define
\begin{align}
    \tilde{\Sigma}^{i,j}_{N_j, D,  s, p}= \bigcap_{j'=0}^j \Sigma^{i,j'}_{N_j, D, s, p}
\end{align}
for which we have by our above bound, 
\begin{align}\label{Section 5 eq: bound on truncated measure of complement}
     \rho_{R, N_j, \pm}(\F L^{s,p} \setminus \tilde{\Sigma}^{i,j}_{N_j, D, s, p})\le 2^{-i}
\end{align}
uniformly in $j\ge 1$. 

Now we are ready to introduce the set whose existence was announced in the statement of Proposition \ref{Prop 5.3}. 
\begin{align} \label{Section 5 eq: definition of Sigma-s-p-i-R}
    \Sigma_{i,R}^{s,p} =\{u\in \F L^{s,p} \operatorname{s.t.} \exists j_k, L_k\to\infty , u_k\in \tilde{\Sigma}^{i, j_k}_{L_k, D, s,p} \operatorname{with } \| u-u_k\|_{\F L^{s,p}}\to 0\}.
\end{align}
Note that the above set depends on $R$ through its dependence on $D$ as $D$  must be taken large enough depending on $K$ which itself varies with $R$. We also have the inclusion
\begin{align} \label{Section 5 inclusion}
    \bigcap_{J\in \mathbb{N}} \left(\bigcup_{j\ge J} \tilde{\Sigma}^{i,j}_{N_j, s, p,D}\right) \subset \Sigma_{i,R}^{s,p}
\end{align}
where the sequence $N_j$ is provided by \eqref{5}. Now, we show that $\Sigma^{s,p}_{i,R}$ satisfies all the required properties. 

By an elementary diagonal argument, $\Sigma^{s,p}_{i,R}$ is closed. By the above inclusion, we have that
\begin{align}
    \rho_{R,\pm}(\F L^{s,p} \setminus\Sigma_{i,R}^{s,p}) &\le \rho_{R,\pm} \left( \F L^{s,p}\setminus\bigcap_{J\in \mathbb{N}} \left(\bigcup_{j\ge J} \tilde{\Sigma}^{i,j}_{N_j, s, p,D}\right) \right) \\
    &= \lim_{J\to\infty} \rho_{R,\pm} \left( \F L^{s,p}\setminus\bigcup_{j\ge J} \tilde{\Sigma}^{i,j}_{N_j, s, p,D}\right)
\end{align}
where we used the fact that $\bigcup_{j\ge J} \tilde{\Sigma}^{i,j}_{N_j, s, p,D}$ is a decreasing family in $J$. Then, by \eqref{truncated measure convergence}, 
\begin{align}
    \lim_{J \to \infty} \rho_{R,\pm}\left( \F L^{s,p} \setminus \bigcup_{j \geq J} \widetilde{\Sigma}_{N_j, D, s,p}^{i,j} \right)
&= \lim_{J \to \infty} \rho_{R,N_J,\pm}\left( \F L^{s,p} \setminus \bigcup_{j \geq J} \widetilde{\Sigma}_{N_j,D ,s,p}^{i,j} \right) \\
&\leq \limsup_{J \to \infty} \rho_{R,N_J,\pm}\left( \F L^{s,p} \setminus \widetilde{\Sigma}_{N_J,D, s, p}^{i,J} \right)\\
&\leq 2^{-i},
\end{align}
where we have used \eqref{Section 5 eq: bound on truncated measure of complement} in the last line. Finally, \eqref{Section 5 property of Sigma-s-p-i-R} comes from the definition of $\Sigma^{s,p}_{i,R}$ along with Lemma \ref{Lemma 5.2 equivalent}. 
\end{proof}

\subsection{Proof of Theorem \ref{Main Result}}

\subsubsection*{Proof of ($i$) in Theorem \ref{Main Result}}
We first introduce a set $\Sigma^{s,p} \subset \F L ^{s,p}$ satisfying items $(i), (ii)$ in Theorem \ref{Main Result}. For any $R>0$, define
\begin{align} \label{eq:Sigma(s,p,R)}
    \Sigma_R^{s,p} = \bigcup_{i\in \mathbb{N}}  \Sigma^{s,p}_{i,R},
\end{align}
where $\Sigma^{s,p}_{i,R}$ is defined in Proposition \ref{Prop 5.3} (see also \eqref{Section 5 eq: definition of Sigma-s-p-i-R}).
We have that $\Sigma_R^{s,p}$ is a $F_\sigma$ set and due to \eqref{2^(-i) bound}, we have
\begin{align} \label{Measure of Sigma is zero}
    \rho_{ R, \pm} \left(\F L^{s,p} \setminus \Sigma_R^{s,p}\right)=0.
\end{align}
Next, we take $R=j$ and define
\begin{align}
    \Sigma^{s,p}= \bigcup_{j\in \mathbb{N}} \Sigma_j^{s,p}.
\end{align}
Since we have that $\rho_{j, \pm} \left(\F L^{s,p} \setminus \Sigma_j^{s,p}\right)=0$ for each $j\ge 1$, we may conclude by Proposition \ref{Prop for complements of Borel Sets} that 
\begin{align}
    \mu(\FL^{s,p} \setminus \Sigma^{s,p})=0. 
\end{align}
and item $(i)$ of Theorem \ref{Main Result} is proved.

\subsubsection*{Proof of ($ii$) in Theorem \ref{Main Result}}
Based on the local Cauchy theory, as developed in Section \ref{Section:FiniteDimApproximations}, to prove item $(ii)$, it is enough to show that for every $u\in \Sigma^{s,p}$, its  $\FL^{s,p}$ norm is bounded on any compact set in time. We will in fact show a bound that is logarithmic in time. Thanks to \eqref{Section 5 property of Sigma-s-p-i-R}, for any $u\in \Sigma^{s,p}_{i,R}$, there exists $j_k, L_k\to\infty$ and a sequence $u_k\in \FL^{s,p}$ such that $u_k\to u$ with
\begin{align}
    \sup_{|\tau|< 2^{j'}} \| \Phi_{N_k}(\tau)u_k\|_{\FL^{s,p}} \le D\sqrt{i+1+j'}, \quad j'\in \{0, \cdots, j_k\}.
\end{align}
Then applying Proposition \ref{Long-time convergence proposition}, we deduce that there exists $\Phi(\tau)u\in C([-2^j,2^j], \FL^{s,p})$ where $\Phi(\tau)$ is the flow associated to \eqref{mKdV 2Cauchy} with data $u$ such that
\begin{align} \label{sec6:equ1}
    \sup_{|\tau|< 2^{j}} \| \Phi(\tau)u\|_{\FL^{s,p}} \le D\sqrt{i+1+j}, \quad j\in \mathbb{N}_0, u\in \Sigma^{s,p}_{i,R}.
\end{align}
Using the above bound with $j=l+1$ we show that
\begin{align}
     \sup_{2^l \le |\tau|< 2^{l+1}} \| \Phi(\tau)u\|_{\FL^{s,p}} \le D\sqrt{i+2+\log_2(2^l)}, \quad l\in \mathbb{N}_0, u\in \Sigma^{s,p}_{i,R},
\end{align}
which implies that
\begin{align}
     \sup_{2^l \le |\tau|< 2^{l+1}} \| \Phi(\tau)u\|_{\FL^{s,p}} \le D\sqrt{i+2+\log_2(|\tau|)}, \quad l\in \mathbb{N}_0, u\in \Sigma^{s,p}_{i,R}.
\end{align}
Since $l$ is arbitrary, we deduce that
\begin{align}\label{Section 5: bound on FL norm for |t|>1}
     \| \Phi(\tau )u\|_{\FL^{s,p}} \le D\sqrt{i+2+\log_2(|\tau|)}, \quad |\tau|\ge1, u\in \Sigma^{s,p}_{i,R}.
\end{align}
On the other hand, letting $j=0$ in \eqref{sec6:equ1}, we conclude that
\begin{align} \label{Section 5: bound on FL norm for |t|<1}
    \sup_{|\tau|< 1} \| \Phi(\tau)u\|_{\FL^{s,p}} \le D\sqrt{i+1}, \quad u\in \Sigma^{s,p}_{i,R}.
\end{align}
Combining \eqref{Section 5: bound on FL norm for |t|>1} and \eqref{Section 5: bound on FL norm for |t|<1} yields 
\begin{align}
\| \Phi(\tau )u\|_{\FL^{s,p}} &\le 1_{|\tau|<1} D\sqrt{i+1}+1_{|\tau|\ge1}D\sqrt{i+2+\log_2(1+|\tau|)}\\
&\le 1_{|\tau|<1} D\sqrt{i+2+ \log(1+|\tau|)} + 1_{|\tau|\ge1}D\sqrt{i+2+\log_2(1+|\tau|)} \\
&=D\sqrt{i+2+\log_2(1+|\tau|)}, \quad u\in \Sigma^{s,p}_{i,R}
\end{align}
for all $\tau\in \R$. Hence part $(ii)$ is proved.

\subsubsection*{Proof of item (iii) in Theorem \ref{Main Result}}
For $1/2 < s<1-1/p, \, \, 2<p<\infty$, we will construct a full measure set $\Sigma^{s,p}\subseteq \FL^{s,p}$ with respect to $\mu$ such that $\Phi(t)(\Sigma^{s,p})=\Sigma^{s,p}$. Furthermore, the measures $\rho_{R, \pm}$ are invariant when restricted to $\Sigma^{s,p}$, for every $R>0$.  

We fix $2<p<\infty$ and $R>0$ and for $\bar{s}\in (s, 1-1/p)$, we pick a sequence
\begin{align}
    s_l \nearrow \bar{s}, \quad s_l\in (s, 1-1/p).
    \end{align}
    Next, we consider the following subset of $\FL^{s,p}$, 
    \begin{align} \label{Def of Sigma^p_R}
        \Sigma_{R}^p:=\bigcap_{l\in \mathbb{N}} \Sigma^{s_l,p}_R, 
    \end{align}
    where $\Sigma^{s,p}_R$ is defined in \eqref{eq:Sigma(s,p,R)}. Note that since $s_l>s$, we have that $\Sigma^{s_l,p}_R \subset \FL^{s,p}$ for each $l$ and therefore $\Sigma^p_R \subset \FL^{s,p}$. 
    
    Now we can show that 
    \begin{align} \label{Null measure in Proof of item (iii)}
        \rho_{R,\pm}(\FL^{s,p}\setminus \Sigma^p_R)=0.
    \end{align} 
    Indeed, we have that
    \begin{align}
        \rho_{R,\pm} \left(\FL^{s,p} \setminus\Sigma^p_R\right) &= \rho_{R,\pm}\left(\bigcup_{l\in \mathbb{N}} \FL^{s,p} \setminus \Sigma^{s_l,p}_R\right)\\
        &\le \sum_{l\in \mathbb{N}} \rho_{R,\pm} \left(\FL^{s,p} \setminus\Sigma^{s_l,p}_R\right) \\
        &=\sum_{l\in \mathbb{N}} \rho_{R,\pm} \left(\FL^{s_l,p} \cap (\FL^{s,p} \setminus\Sigma^{s_l,p}_R)\right) \\
        &=\sum_{l\in \mathbb{N}} \rho_{R,\pm} \left( \FL^{s_l,p} \setminus\Sigma^{s_l,p}_R\right),
    \end{align}
where in the last two steps, we used the fact that $\rho_{R,\pm}(\FL^{s_l,p})=1$ for each $l\in \mathbb{N}$ and the fact that $\FL^{s_l,p} \subset \FL^{s,p}$ since $s_l>s$. However, 
   \begin{align}
       \rho_{R,\pm} \left( \FL^{s_l,p} \setminus\Sigma^{s_l,p}_R\right)=0, \quad l\in \mathbb{N}
   \end{align}
   due to \eqref{Measure of Sigma is zero}. Hence 
   \eqref{Null measure in Proof of item (iii)} is proved.
   
   Next we claim that the set $\Sigma^p_R$ is invariant under the flow $\Phi(t)$ for each $t>0$, where the existence of the flow follows from Proposition \ref{Long-time convergence proposition}. Once this is proved, it is sufficient to choose
   \begin{align} \label{Sigma sp}
       \Sigma^{s,p}= \bigcup_{R\in \mathbb{N}} \Sigma^p_R.
   \end{align}
Now, we proceed to prove invariance of $\Sigma^p_R$ under the flow $\Phi(t)$. Note that by the analysis done in the previous subsection, the flow $\Phi(t)$ is globally well defined on $\Sigma^p_R$. Therefore we only have to prove the invariance part of $\Sigma^{p}_R$ along $\Phi(t)$. To do so, we will show that
   \begin{align} \label{eq: invariance of Sigma under Phi}
       u\in \Sigma^{s_{l+1}, p}_R \implies \Phi(t) u\in \Sigma^{s_{l}, p}_R, \quad l\in \mathbb{N}.
   \end{align}
   In fact, once the above is proved, we have that given $u\in \Sigma^p_R$ in \eqref{Def of Sigma^p_R}, we have that $u\in \Sigma^{s_{l+1}, p}_R$ for every $l\in \mathbb{N}$ and therefore $\Phi(t)u\in \Sigma^{s_{l}, p}_R$ for each $l\in\mathbb{N}$, showing that $\Phi(t)u\in \Sigma^p_R$. 
   
   Notice that in the definition of $\Sigma^{s,p}_{i,R}$, the quantity $D$ may need to be taken large enough depending on $s$ via the constants $k,K$ that appear in the proof Proposition \ref{Prop: Properties of weighted measures} which do indeed depend on $s$. However, the key is that we are considering $s\in (1/2, \bar{s}]$, so the dependence may be taken only with respect to $\bar{s}$, and therefore uniform in $s$.

   We now prove \eqref{eq: invariance of Sigma under Phi}. We will assume $t>0$. A similar argument for $t<0$ may be applied. We first define
   \begin{align}
       \bar{j}(t) = \min\{j\in \mathbb{N}: 2^j\ge t\}.
   \end{align}
    Note that if $u_0\in \Sigma^{s_{l+1},p}_{R}$,  then the following holds by \eqref{eq:Sigma(s,p,R)} and \eqref{Section 5 eq: definition of Sigma-s-p-i-R}: 
    \begin{align}
\exists i \in \mathbb{N}, \quad u_k \in \mathcal{F}L^{s_{l+1},p}, \quad 
j_k \to \infty, \quad 
L_k \to \infty, \quad 
u_k \overset{\mathcal{F}L^{s_{l+1},p}}{\longrightarrow} u, \\
\sup_{|\tau|<2^{j'}} \bigl\| \Phi_{L_k}(\tau) u_k \bigr\|_{\mathcal{F}L^{s_{l+1},p}} 
\le D \sqrt{i+j'+1}, 
\quad \forall j' \in \{0,\dots,j_k\}.
\end{align}
    From this bound and the group property of $\Phi_N(t)$, we deduce that
    \begin{align} \label{eq:group bound}
\sup_{|t+\tau| < 2^{j'}} 
  \bigl\| \Phi_{L_k}(\tau)\bigl(\Phi_{L_k}(t) u_k\bigr) \bigr\|_{\mathcal{F}L^{s_{l+1},p}}
  \;\le\; D \sqrt{i + j' + 1},
  \qquad \forall j' \in \{0,\dots,j_k\}.
   \end{align}
    Next, we notice that for every $j'> \bar{j}(t)$, we have that
    \begin{align}
        [-2^{j'-1}, \, 2^{j'-1}] 
\;\subset\; [-t - 2^{j'}, \, -t + 2^{j'}], 
\qquad \forall j' > \bar{j}(t).
    \end{align}
    and thus from the bound \eqref{eq:group bound}, we obtain
    \begin{align} \label{5.22}
        \sup_{|\tau| < 2^{\,j'-1}} 
  \bigl\| \Phi_{L_k}(\tau)\bigl(\Phi_{L_k}(t) u_k\bigr) \bigr\|_{\mathcal{F}L^{s_{l+1},p}}
  \;\le\; D \sqrt{i + j' + 1},
  \qquad \forall j' \in \{\bar{j}+1,\dots,j_k\}.
    \end{align}
    However, letting $j'=\bar{j}(t)+1$ in the above, we find that
    \begin{align}
        \sup_{|\tau| < 2^{\bar{j}(t)}} 
  \bigl\| \Phi_{L_k}(\tau)\bigl(\Phi_{L_k}(t) u_k\bigr) \bigr\|_{\mathcal{F}L^{s_{l+1},p}}
  \;\le\; D \sqrt{i + \bar{j}(t) + 2},
    \end{align}
    which straightforwardly implies that 
    \begin{align}\label{5.23}
        \sup_{|\tau| < 2^{\,j'}} 
  \bigl\| \Phi_{L_k}(\tau)\bigl(\Phi_{L_k}(t) u_k\bigr) \bigr\|_{\mathcal{F}L^{s_{l+1},p}}
  \;\le\; D \sqrt{i + \bar{j} + 2 + j'},
  \qquad \forall j' \in \{1,\dots,\bar{j}\}.
    \end{align}
    By combining \eqref{5.22} and \eqref{5.23} with the fact that $s_{l+1}\ge s_l$, we find that
    \begin{align}
        \sup_{|\tau| < 2^{\,j'}} 
  \bigl\| \Phi_{L_k}(\tau)\bigl(\Phi_{L_k}(t) u_k\bigr) \bigr\|_{\mathcal{F}L^{s_l,p}}
  \;\le\; D \sqrt{i + \bar{j} + 2 + j'},
  \qquad \forall j' \in \{1,\dots,j_k-1\}.
    \end{align}
    Moreover, due to Proposition \ref{Long-time convergence proposition}, we have that $\Phi_{L_k}u_k\to \Phi(t)u$ in $\FL^{s_l, p}$. Therefore, by the definition \eqref{Section 5 eq: definition of Sigma-s-p-i-R}, \eqref{Def of Sigma^p_R}, with $\Phi_{L_k}u_k$ taking the role of $u_k$  we have that $\Phi(t)u\in \Sigma^{s_l,p}_{i+\bar{j}+2,R}$  and \eqref{eq: invariance of Sigma under Phi} follows. 
    
    Now the invariance of the measure $\rho_{R, \pm}$ on $\Sigma^{s,p}$, defined in \eqref{Sigma sp} follows by the same argument used for the Benjamin-Ono equation in \cite{TV-IMRN-13}.
    
\bibliographystyle{alpha}
\bibliography{References-Submission}

@misc{KNPSV,
      title={Well-posedness and invariant measures for complex valued modified KdV equation}, 
      author={Carlos E. Kenig and Andrea R. Nahmod and Nataša Pavlović and Gigliola Staffilani and Nicola Visciglia},
      year={2025},
      eprint={2501.14920},
      archivePrefix={arXiv},
      primaryClass={math.AP},
      url={https://arxiv.org/abs/2501.14920}, 
}

@article{Chapouto2021,
title = {A remark on the well-posedness of the modified KdV equation in the Fourier-Lebesgue spaces},
journal = {Discrete and Continuous Dynamical Systems},
volume = {41},
number = {8},
pages = {3915-3950},
year = {2021},
issn = {1078-0947},
doi = {10.3934/dcds.2021022},
url = {https://www.aimsciences.org/article/id/6a948838-8cea-4c63-b1c2-864fc22a9e0e},
author = {Andreia Chapouto}
}

@article{ChapoutoJDDE,
  author  = {Chapouto, Andreia},
  title   = {A Refined Well-Posedness Result for the Modified KdV Equation in the Fourier--Lebesgue Spaces},
  journal = {Journal of Dynamics and Differential Equations},
  volume  = {35},
  number  = {3},
  pages   = {2537--2578},
  year    = {2023},
  month   = sep,
  doi     = {10.1007/s10884-021-10050-0},
  url     = {https://doi.org/10.1007/s10884-021-10050-0}
}

@article{BourgainGAFA,
	author = {Bourgain, J.},
	title = {Fourier transform restriction phenomena for certain lattice subsets and applications to nonlinear evolution equations},
	journal = {Geometric \& Functional Analysis},
	volume = {3},
	number = {2},
	pages = {107--156},
	year = {1993},
	month = mar,
	doi = {10.1007/BF01896020},
	url = {https://doi.org/10.1007/BF01896020}
}

@article{bourgain1994,
  author       = {Jean Bourgain},
  title        = {Periodic nonlinear Schrödinger equation and invariant measures},
  journal      = {Communications in Mathematical Physics},
  volume       = {166},
  number       = {1},
  year         = {1994},
  pages        = {1--26},
  doi          = {10.1007/BF02099178},
  publisher    = {Springer}
}

@article {MR3119672,
    AUTHOR = {B\'enyi, \'Arp\'ad and Oh, Tadahiro},
     TITLE = {The {S}obolev inequality on the torus revisited},
   JOURNAL = {Publ. Math. Debrecen},
  FJOURNAL = {Publicationes Mathematicae Debrecen},
    VOLUME = {83},
      YEAR = {2013},
    NUMBER = {3},
     PAGES = {359--374},
      ISSN = {0033-3883,2064-2849},
   MRCLASS = {42B35 (42B10 42B25)},
  MRNUMBER = {3119672},
MRREVIEWER = {Anna\ K.\ Savvopoulou},
       DOI = {10.5486/PMD.2013.5529},
       URL = {https://doi.org/10.5486/PMD.2013.5529},
}

@book{ST,
	address = {Chichester ;},
	author = {Schmeisser, Hans-J{\"u}rgen. and Triebel, Hans.},
	booktitle = {Topics in Fourier analysis and function spaces},
	isbn = {0471908959},
	language = {eng},
	lccn = {85022706},
	publisher = {John Wiley \& Sons},
	title = {Topics in Fourier analysis and function spaces / Hans-J{\"u}rgen Schmeisser, Hans Triebel.},
	year = {1987}}

@article{DNY21,
  author  = {Deng, Yu and Nahmod, Andrea R. and Yue, Haitian},
  title   = {Optimal Local Well-Posedness for the Periodic Derivative Nonlinear Schr{\"o}dinger Equation},
  journal = {Communications in Mathematical Physics},
  volume  = {384},
  number  = {2},
  pages   = {1061--1107},
  year    = {2021},
  month   = jun,
  doi     = {10.1007/s00220-020-03898-8},
  url     = {https://doi.org/10.1007/s00220-020-03898-8}
}

@book{Bogachev2015Gaussian,
  author    = {Bogachev, Vladimir I.},
  title     = {Gaussian Measures},
  series    = {Mathematical Surveys and Monographs},
  volume    = {62},
  publisher = {American Mathematical Society},
  address   = {Providence, RI},
  year      = {2015},
}

@article{TV-ASENS-13,
     author = {Tzvetkov, Nikolay and Visciglia, Nicola},
     title = {Gaussian measures associated to the higher order conservation laws of the {Benjamin-Ono} equation},
     journal = {Annales scientifiques de l'\'Ecole Normale Sup\'erieure},
     pages = {249--299},
     publisher = {Soci\'et\'e math\'ematique de France},
     volume = {Ser. 4, 46},
     number = {2},
     year = {2013},
     doi = {10.24033/asens.2189},
     language = {en},
     url = {https://www.numdam.org/articles/10.24033/asens.2189/}
}

@article{TV-IMRN-13,
    author = {Tzvetkov, Nikolay and Visciglia, Nicola},
    title = {Invariant Measures and Long-Time Behavior for the Benjamin–Ono Equation},
    journal = {International Mathematics Research Notices},
    volume = {2014},
    number = {17},
    pages = {4679-4714},
    year = {2013},
    month = {05},
    issn = {1073-7928},
    doi = {10.1093/imrn/rnt094},
    url = {https://doi.org/10.1093/imrn/rnt094},
    eprint = {https://academic.oup.com/imrn/article-pdf/2014/17/4679/18882421/rnt094.pdf},
}

@article {TV-JMPA-15,
    AUTHOR = {Tzvetkov, Nikolay and Visciglia, Nicola},
     TITLE = {Invariant measures and long time behaviour for the
              {B}enjamin-{O}no equation {II}},
   JOURNAL = {J. Math. Pures Appl. (9)},
  FJOURNAL = {Journal de Math\'ematiques Pures et Appliqu\'ees. Neuvi\`eme
              S\'erie},
    VOLUME = {103},
      YEAR = {2015},
    NUMBER = {1},
     PAGES = {102--141},
      ISSN = {0021-7824,1776-3371},
   MRCLASS = {37L40 (37K10)},
  MRNUMBER = {3281949},
MRREVIEWER = {Naoyuki\ Ishimura},
       DOI = {10.1016/j.matpur.2014.03.009},
       URL = {https://doi.org/10.1016/j.matpur.2014.03.009},
}

@book{FaddeevTakhtajan2007,
  author    = {L. D. Faddeev and L. A. Takhtajan},
  title     = {Hamiltonian Methods in the Theory of Solitons},
  series    = {Classics in Mathematics},
  publisher = {Springer},
  address   = {Berlin},
  year      = {2007},
  edition   = {English ed.},
  isbn      = {978-3-540-69969-9},
}

@article{NahmodOhReyBelletStaffilani2012,
  author    = {Andrea R. Nahmod and Tadahiro Oh and Luc Rey-Bellet and Gigliola Staffilani},
  title     = {Invariant weighted Wiener measures and almost sure global well-posedness for the periodic derivative NLS},
  journal   = {Journal of the European Mathematical Society},
  volume    = {14},
  number    = {4},
  pages     = {1275--1330},
  year      = {2012},
  doi       = {10.4171/JEMS/333},
  note      = {Submitted 6 January 2011, Published 3 June 2012}
}

@article{DengNahmodYue2024Gibbs2D,
  author  = {Deng, Yu and Nahmod, Andrea R. and Yue, Haitian},
  title   = {Invariant Gibbs measures and global strong solutions for nonlinear Schr\"odinger equations in dimension two},
  journal = {Annals of Mathematics},
  volume  = {200},
  number  = {2},
  pages   = {399--486},
  year    = {2024},
  doi     = {10.4007/annals.2024.200.2.1}
}

@article{DengNahmodYue2022RandomTensors,
  author  = {Deng, Yu and Nahmod, Andrea and Yue, Haitian},
  title   = {Random tensors, propagation of randomness and nonlinear dispersive equations},
  journal = {Inventiones Mathematicae},
  volume  = {228},
  year    = {2022},
  pages   = {539--686}
}

@article{Palais1997,
  author  = {Palais, Richard S.},
  title   = {The symmetries of solitons},
  journal = {Bulletin of the American Mathematical Society (New Series)},
  volume  = {34},
  number  = {4},
  year    = {1997},
  pages   = {339--403},
  doi     = {10.1090/S0273-0979-97-00732-5},
  note    = {MR1462745}

}

@article{LebowitzRoseSpeer1988,
  author  = {Lebowitz, Joel L. and Rose, Howard A. and Speer, Elliott R.},
  title   = {Statistical mechanics of the nonlinear Schr{\"o}dinger equation},
  journal = {Journal of Statistical Physics},
  year    = {1988},
  volume  = {50},
  number  = {3--4},
  pages   = {657--687},
  doi     = {10.1007/BF01026468}
}

@article{TT-2010,
  author    = {Thomann, Laurent and Tzvetkov, Nikolay},
  title     = {Gibbs measure for the periodic derivative nonlinear Schrödinger equation},
  journal   = {Nonlinearity},
  volume    = {23},
  number    = {11},
  pages     = {2771},
  year      = {2010},
  date      = {2010-10},
  doi       = {10.1088/0951-7715/23/11/003},
  url       = {https://doi.org/10.1088/0951-7715/23/11/003}
}
\end{document}